\documentclass[a4paper, 12pt]{article}

\newif\ifsubmit
\newif\ifarxiv
\submittrue
\arxivtrue

\usepackage{amsmath,amsthm,amssymb}
\usepackage{times}
\usepackage{dsfont}
\usepackage{mathrsfs} 
\usepackage[all,cmtip]{xy}
\usepackage{patchcmd}
\usepackage{url}
\usepackage{bm}
\usepackage{mathtools}
\usepackage{enumerate}
\usepackage{pbox}
\usepackage[usenames,dvipsnames,svgnames,table]{xcolor}
\usepackage[normalem]{ulem}
\usepackage[top=2.4cm, bottom=2.4cm, left=2cm, right=2cm]{geometry}
\usepackage{array,multirow}
\ifsubmit\else
\usepackage{showkeys}
\fi
\usepackage{tikz}

\def\titlerunning#1{\gdef\titrun{#1}}
\makeatletter
\def\author#1{\gdef\autrun{\def\and{\unskip, }#1}\gdef\@author{#1}}
\def\address#1{{\def\and{\\\hspace*{18pt}}\renewcommand{\thefootnote}{}%
\footnote {#1}}%
\markboth{\autrun}{\titrun}}
\makeatother
\def\email#1{e-mail: #1}
\def\subjclass#1{{\renewcommand{\thefootnote}{}%
\footnote{\emph{Mathematics Subject Classification (2010):} #1}}}
\def\keywords#1{\par\medskip
\noindent\textbf{Keywords.} #1}

\ifarxiv\else
\fi

\usetikzlibrary{arrows}
\usetikzlibrary{calc}

\tikzset{
auto,x=1cm,y=1cm,
arrow/.style={-stealth'},
axis arrow/.style={-stealth',thick},
axis noarrow/.style={thick},
S0text/.style={color=black!30!blue},
S0set/.style={fill=black!5!white!85!blue},
S0setT/.style={fill=black!6!white!82!blue,opacity=.8},
S0set1/.style={color=black!20!white!60!blue,line width=2.5},
S1text/.style={color=black!50!green},
S1set/.style={ultra thick,black!30!green},
S2set/.style={red!80!black},
S2text/.style={red!80!black},
dash3d/.style={dashed,preaction={draw,ultra thick,white}},
}

\def\tangentplane(#1)(#2)(#3){ %
($(#1)+(#2)$) -- ($(#1)+(#3)$) --
($(#1)-(#2)$) -- ($(#1)-(#3)$) -- cycle}

\def\tangentline(#1)(#2){($(#1)+(#2)$) -- ($(#1)-(#2)$)}

\newenvironment{ctikzpicture}
  {\leavevmode\hspace*{\fill}\begin{tikzpicture}}
  {\end{tikzpicture}\hspace*{\fill}\par}

\newcommand{\valring}[1][\bmdl]{\mathcal{O}_{#1}} 

\newbox\removebox
\long\def\bigsout#1{%
\par\setbox\removebox=\vbox{#1}%
\vbox{%
\vbox to0pt{\hbox{\tikz\draw[thick] (0,0) -- (\wd\removebox,-\ht\removebox)  (\wd\removebox,0) -- (0,-\ht\removebox);}}
\box\removebox
}
}

\ifsubmit\else

\newcount\revision
\def\parserevision$Revision: #1${\revision#1}
\parserevision$Revision: 235 $
\newcount\outfadePercentage

\makeatletter
\long\def\colorAndFade#1#2{\@ifnextchar[{\colorAndFade@Yes{#1}{#2}}{\colorAndFade@No{#1}{#2}}}
\long\def\colorAndFade@No#1#2#3{\strut{\color{#1}#3}}
\long\def\colorAndFade@Yes#1#2[#3]{%
  \outfadePercentage-#3%
  \advance\outfadePercentage by \revision%
  \advance\outfadePercentage by -1\relax%
  \multiply\outfadePercentage by 6\relax
  \ifnum\outfadePercentage>70\relax\outfadePercentage70\fi
  \ifnum\outfadePercentage<0\relax\outfadePercentage0\fi
  \colorAndFade@No{#2!\the\outfadePercentage!#1}{#2}%
}
\makeatother

\definecolor{private color}{rgb}{0,.6,.1}
\definecolor{immi's favourite color}{rgb}{1,.3,0}
\definecolor{a change by immi}{rgb}{.7,0,.7}
\definecolor{a comment by immi}{rgb}{.9,.6,0}

\definecolor{armygreen}{rgb}{0.29, 0.33, 0.13}
\definecolor{antiquefuchsia}{rgb}{0.57, 0.36, 0.51}
\definecolor{slightly darker green}{rgb}{0, .8, 0}

\fi

\ifsubmit
\newcommand{\private}[1]{}
\else
\newcommand{\private}[1]{\par{\color{private color}\scriptsize
\noindent \emph{Private comment (not to be published):}
#1\par}}
\fi

 \theoremstyle{plain}
 \newtheorem{thm}{Theorem}[subsection]
 \newtheorem{cor}[thm]{Corollary}
 \newtheorem{lem}[thm]{Lemma}
 
 \newtheorem{prop}[thm]{Proposition}

\theoremstyle{definition}
 \newtheorem{defn}[thm]{Definition}
 \newtheorem{notn}[thm]{Notation}
 \newtheorem{ass}[thm]{Assumption}

\theoremstyle{remark}
 \newtheorem{rem}[thm]{Remark}

 \newtheorem{exam}[thm]{Example}

 \numberwithin{equation}{section}

 \numberwithin{figure}{section}

\theoremstyle{plain}
\newtheorem{theo}{Theorem}[]

\DeclareMathOperator{\CC}{\mathcal{C}}
\DeclareMathOperator{\UU}{\mathcal{U}}
\DeclareMathOperator{\rk}{rk}

 \DeclareMathOperator{\id}{id}

 \DeclareMathOperator{\Hom}{Hom}
 \DeclareMathOperator{\im}{im}

 \DeclareMathOperator{\pr}{pr}

\DeclareMathOperator{\Jac}{Jac}
\DeclareMathOperator{\GL}{GL}
\DeclareMathOperator{\Or}{O}

\def\XXint#1#2#3{{\setbox0=\hbox{$#1{#2#3}{\int}$}
\vcenter{\hbox{$#2#3$}}\kern-.5\wd0}}

\newcommand{\Q}{\mathds{Q}}

\newcommand{\R}{\mathds{R}}

\newcommand{\omin}{$o$\nobreakdash}

\newcommand{\T}{$T$\nobreakdash}

\DeclareMathAlphabet{\mathpzc}{OT1}{pzc}{m}{it}

 \DeclarePairedDelimiter\norm{\lVert}{\rVert}

 \newcommand{\lan}[3]{\mathcal{L}_{#1 \textup{#2} #3}}

\newcommand{\mdl}[1]{\mathcal{#1}}  

\newcommand{\rest}{\upharpoonright}

\newcommand{\fun}{\longrightarrow}
\newcommand{\efun}{\longmapsto}
\newcommand{\sub}{\subseteq}

\newcommand{\mi}{\smallsetminus}

\DeclareMathOperator{\rv}{rv}

\DeclareMathOperator{\vv}{val}

\DeclareMathOperator{\rad}{rad}

\DeclareMathOperator{\ito}{int}
\DeclareMathOperator{\cl}{cl}

\DeclareMathOperator{\dist}{dist}
\DeclareMathOperator{\valdist}{valdist}

\newcommand{\langua}[1]{\ifmmode #1\else $#1$\nobreakdash\fi}
\newcommand{\LT}{\langua{\lan{}{}{}}}
\newcommand{\LTv}{\langua{\lan{\vv}{}{}}}
\newcommand{\LTd}{\langua{\lan{}{}{}(d)}}
\newcommand{\LTA}{\langua{\lan{}{}{}(A)}}

\newcommand{\LTq}{\langua{\lan{}{}{}(q)}}
\newcommand{\LTy}{\langua{\lan{}{}{}(y)}}

\DeclareMathOperator{\gra}{Gr}

\newcommand{\grao}{\gra^\circ}

\newcommand{\aligner}{\kappa} 

\newcommand{\forb}{\mathit{Fo}}

\newcommand{\U}[1]{_{\underline{#1}}} 

\newcommand{\prep}[1]{\ensuremath{(\mathrm{#1})}}

\makeatletter
\def\udots{\mathinner{%
    \mkern1mu%
    \raise\p@\vbox{\kern7\p@\hbox{$\m@th.$}}%
    \mkern2mu%
    \raise4\p@\hbox{$\m@th.$}%
    \mkern2mu%
    \raise7\p@\hbox{$\m@th.$}%
    \mkern1mu}%
}

\newcommand\hsmash[1]{\protect\mathpalette{\protect\hsmash@}{#1}}
\def\hsmash@#1#2{\hbox to 0pt{$#1#2$}}

\makeatother

\newcommand{\bbar}[1]{\bar{\bar{#1}}}
\newcommand{\bbbar}[1]{\bar{\bbar{#1}}}

\newcommand{\smdl}{\mdl{R}_0}  
\newcommand{\bmdl}{{\mdl R}}      

\begin{document}

\ifarxiv\else
\baselineskip=17pt
\fi


\ifsubmit
\titlerunning{Lipschitz stratifications in power-bounded o-minimal fields}
\title{Lipschitz stratifications in power-bounded o-minimal fields}
\else
\titlerunning{Lipschitz stratifications in power-bounded o-minimal fields, version \the\revision}
\title{Lipschitz stratifications in power-bounded o-minimal fields, version \the\revision}
\fi

\author{Immanuel Halupczok
\and 
Yimu Yin}

\ifarxiv
\date{9th Feb.\ 2016}
\else
\date{9th Feb.\ 2016}
\fi

\maketitle

\address{I. Halupczok%
\ifarxiv\else{} (corresponding author)\fi%
: School of Mathematics, University of Leeds, Leeds, LS2~9JT, UK; \email{math@karimmi.de}
\and
Y. Yin: 607 Xi Chang Hall, Department of Philosophy,
Sun Yat-sen University,
135 Xingang Road West,
Guangzhou, China,
510275; \email{yimu.yin@hotmail.com}}



\begin{abstract}
We propose to grok Lipschitz stratifications from a non-archimedean point of view and thereby show that they exist for closed definable sets in any power-bounded \omin-minimal structure on a real closed field. Unlike the previous approaches in the literature, our method bypasses resolution of singularities and Weierstra\ss{} preparation altogether; it transfers the situation to a non-archimedean model, where the quantitative estimates appearing in Lipschitz stratifications are sharpened into valuation-theoretic inequalities.
Applied to a uniform family of sets, this approach automatically yields a family of stratifications which satisfy the Lipschitz conditions in a uniform way.

\end{abstract}

\ifsubmit\else
\tableofcontents
\fi

In this paper we prove the existence of Lipschitz stratifications for any closed definable set in a polynomially bounded \omin-minimal structure on $\R$, and, in fact, even more generally, in a power-bounded \omin-minimal structure on a real closed field $\bmdl$.
The notion of a Lipschitz stratification was introduced by Mostowski in his dissertation \cite{most:phd:lip}. It is much stronger than Whitney's conditions or Verdier's condition (w) formulated in \cite{verd:stra:w}; it imposes a global condition and
ensures that the Lipschitz type of the stratified set is locally constant along each stratum.

Throughout this paper, $\bmdl$ is a power-bounded real closed field.
A classical example of such a structure is $\R_{\mathrm{an}}$: the reals with restricted analytic functions
as described in \cite{DMM94}; beyond this (subanalytic) level, there is e.g.\ the class of quasianalytic structures; see \cite{Rolin:quasi}.
If the field $\bmdl$ is just $\R$, then power-bounded is equivalent to polynomially bounded. In other real closed fields, power-boundedness is more general and more natural;
we recall that notion in Definition~\ref{defn.powBd}.

Here is a first version of our main result.

\begin{theo}[Lipschitz stratifications]\label{main}
Let $X \sub \bmdl^n$ be a closed definable subset in a power-bounded real closed field $\bmdl$. Then there exists a definable Lipschitz stratification of $X$.
\end{theo}

The notion of Lipschitz stratification is recalled in Definition~\ref{defn.lip}, and Subsection~\ref{sect.notn.smdl} clarifies how the terminology should be adapted in the case $\bmdl \ne \R$.
For compact sets $X$ and in the case $\bmdl = \R$, the semi-analytic case of this theorem was established in \cite{Paru1988} and the subanalytic case in \cite{Parusinski1994}.
Recently, Nguyen and Valette \cite{NV.lip} generalized Parusinski's proof to polynomially bounded structures on $\R$. (In \cite{NV.lip} the result is stated for compact $X$, but their proof also goes through for arbitrary closed $X$; see \cite{Ngu.thes}.)

A main motivation for Lipschitz stratifications is that one has local bilipschitz triviality along strata, which in turn implies that any two points within the same stratum have neighborhoods which are in bilipschitz bijection. The proof of this result is rather easy in $\R$, but the argument uses integration along vector fields; this is highly non-definable, and it does not generalize to other real closed fields. We believe that local bilipschitz triviality (along strata of a Lipschitz stratification) can also be obtained in $\bmdl \ne \R$, but the argument might be much more involved. More precisely, a proof of the existence of \emph{definable} local bilipschitz trivializations within $\R$ would probably directly generalize to $\bmdl$. Some results in that direction exist. For example, Valette \cite{Val.lipTriang} proved the existence of definable bilipschitz trivializations in polynomially bounded \omin-minimal structures, but using certain triangulations instead of Lipschitz stratifications.

Using the existence of Skolem functions and the Compactness Theorem, one easily deduces that
Theorem~\ref{main} also works uniformly in families, in the sense that given a uniformly definable family of sets, one finds a uniformly definable family of Lipschitz stratifications.
However, the notion of Lipschitz stratifications involves a constant $C$ (a stratification is Lipschitz if some conditions hold for sufficiently big $C$), and a natural question is whether that $C$ can be chosen to be the same for an entire family. In this paper, we obtain uniform Lipschitz stratifications in families in this strong sense; the precise statement is Theorem~\ref{thm.main}.

Our approach to the construction of Lipschitz stratifications is quite different from all previous ones.
The main difference is that we use the technique from non-standard analysis of replacing $\bmdl$ by a bigger real closed field $\bmdl'$ (an elementary extension). The 
infinite and infinitesimal elements in $\bmdl'$ make it possible to simplify the formulation of statements involving limits. In particular, we obtain simpler characterizations of Lipschitz stratifications: Whereas the original definition of a Lipschitz stratification uses subtle inequalities depending on two different constants $c$ and $C$, we obtain an equivalent definition, formulated using $\bmdl'$, which needs neither $c$ nor $C$ (see Definition~\ref{defn.nalip} and Proposition~\ref{prop.new-lip-na}).
The aforementioned strong uniformity in families is obtained as a side effect of using that approach: We prove that 
Lipschitz stratifications exist in families within $\bmdl'$. The fact that the parameters of the family are allowed to run over the bigger field $\bmdl'$ allows us to deduce the strong uniformity result within $\bmdl$.

On our way, we also obtain various other equivalent characterizations of Lipschitz stratifications: Proposition~\ref{prop.new-lip} provides some characterizations
purely within the standard model, where $c$ and $C$ are used in a less subtle way, and Proposition~\ref{prop.flag-lip} provides a new characterization of Lipschitz stratifications in terms of partial flags, which is invariant under $\GL_n$. (To our knowledge, the only previously known $\GL_n$-invariant characterization was the one terms of vector bundles given e.g. in \cite[Proposition~1.5]{Paru1988}.)

Typically, proofs carried out using non-standard analysis in an elementary extension $\bmdl'$ can be translated back to ``classical'' proofs within $\bmdl$ (at the cost of making them much less readable). However, for one key ingredient to our proof -- a precise estimate of the gradient of functions near a singular locus;
cf.\ Corollary~\ref{cor.jac} and Remark~\ref{rem.jac} -- we use some deeper model theoretic results. More precisely, $\bmdl'$ naturally carries a valuation, which specifies the order of magnitude of elements. The proof of our estimate builds on model theory of 
$\bmdl'$ as a \emph{valued} field, i.e., we consider definable sets in a language including the valuation. This setting has been studied by van den Dries and Lewenberg \cite{DriesLew95,Dries:tcon:97} under the name of ``$T$-convex fields''. In that setting, the second author of the present paper obtained a result which is somewhat related to Weierstra\ss{} Preparation in valued fields (Proposition~\ref{prop.jac}) and that in turn implies the above-mentioned Corollary~\ref{cor.jac}.

In Section~\ref{sect.lip-defs} we recall the notion of Lipschitz stratifications and prove the equivalence of its various characterizations. We also give an overview of the proof of existence of Lipschitz stratifications (in Subsection~\ref{sect.overview}).
The entire remainder of the paper is devoted to the details of that proof.
Section~\ref{sect.ingred} discusses the various ingredients
and Section~\ref{sect.proof} contains the proof itself.

\subsection{Acknowledgement}

The research reported in this paper has received financial support from the ERC Advanced Grant 246903 NMNAG, from the IHES, from the European Research Council under the European Community's Seventh Framework Programme (FP7/2007-2013) with ERC Grant Agreement no.\ 615722 MOTMELSUM, from the Labex CEMPI (ANR-11-LABX-0007-01), from the Fund for Scientific Research of Flanders, Belgium (grant G.0415.10) and from the laboratoire de math\'ematiques de l'universit\'e
de Savoie Mont Blanc.

\section{Characterizations of Lipschitz Stratifications}
\label{sect.lip-defs}

In this section, we recall the definition of Lipschitz stratifications, we formulate several alternative
definitions and we prove that all those definitions are equivalent. This does not yet use any deep model theory; the only model theoretic ingredient we use is the notion of elementary extensions (and their existence).

\subsection{Basic Notation}
\label{sect.notn.smdl}

We fix some notation which will be used throughout this paper.

Recall that an \omin-minimal structure on $\R$ is polynomially bounded if every definable function $\R \to \R$ is ultimately bounded by a polynomial. One essential aspect of this notion is the dichotomy obtained by Miller \cite{Mil.growth-dichotomy}: In any structure that is not polynomially bounded, one can already define exponentiation. To obtain a similar dichotomy for other real closed fields $\bmdl$, one needs a generalization of polynomially bounded \cite{Mil.powBd}: a definable function only needs to be bounded by a kind of generalized power function. Here is the precise definition.

\begin{defn}[Power bounded]\label{defn.powBd}
Suppose that $\bmdl$ is an \omin-minimal real closed field. A \emph{power function} in $\bmdl$ is a definable endomorphism of the multiplicative group $\bmdl^\times$.
We call $\bmdl$ \emph{power bounded} if for every definable function $f\colon \bmdl \fun \bmdl$, there exists a power function $g$ such that $|f(x)| \le g(x)$ for all sufficiently big $x$.
\end{defn}

There is a precise sense in which a power function is of the form $x \efun x^\lambda$, where $\lambda$
is an element of a certain subfield of $\bmdl$. Since we will use power-boundedness only indirectly, we do not elaborate on this; see \cite{Mil.powBd} for details.

\begin{notn}[Structures and language]
Throughout this paper, we fix a power-bounded \omin-minimal
real closed field $\bmdl$ in a language $\LT$ expanding the ring language. (At some point, we will impose that $\bmdl$ is, without loss, sufficiently big).

By \emph{definable} we mean definable with arbitrary parameters; in contrast, \emph{\LT-definable} means definable without parameters (apart from those which are constants in the language).
\end{notn}

\begin{rem}
It is somewhat customary, in \omin-minimal geometry, to not specify a language and only work with the notion of definable sets. However, specifying a language allows us to keep track of the parameters needed to define sets, and this will be needed for some model theoretic arguments.
For the moment, the reader unfamiliar with our approach may assume that $\LT$ contains a constant for each element of $\bmdl$, so that \LT-definable means the same as definable.
\end{rem}

\begin{notn}[Coordinate projections]
Given $d \le n$, we write
$\pr_{d}\colon \bmdl^n \fun \bmdl$ for the projection to the $d$-th coordinate,
$\pr_{\le d}\colon \bmdl^n \fun \bmdl^d$ for the projection to the first $d$ coordinates
and $\pr_{> d}\colon \bmdl^n \fun \bmdl^{n-d}$ the for projection
to the last $n-d$ coordinates.
\end{notn}

We use the usual notation and conventions for \omin-minimal expansions of real closed fields; see e.g.\ \cite{dries:1998}. We quickly recall the most important ones.

\begin{notn}[Infima and suprema]
By \omin-minimality, any definable subset $X \subseteq \bmdl$ has an infimum and a supremum (which may be $\pm \infty$); we denote them by $\inf(X)$ and $\sup(X)$.
\end{notn}

\begin{notn}[Norms and distances]
We write $|\cdot|$ for the absolute value on $\bmdl$,
$\norm{a}$ for the Euclidean Norm of $a \in \bmdl^n$
($\norm{a}$ is an element of $\bmdl_{\ge 0}$)
and $\norm{M}$ for the operator norm of a matrix $M$,
i.e., $\norm{M} = \sup\{\norm{Ma} : \norm{a} = 1\}$.
Given a point $a \in \bmdl^n$ and a definable set $X \subseteq \bmdl^n$, we write $\dist(a, X) \coloneqq \inf \{\norm{a - x} : x \in X\}$
for the distance from $a$ to $X$; we define that distance to be $\infty$ if $X$ is empty.
\end{notn}

\begin{notn}[Topology]
The real closed field $\bmdl$ comes with a natural topology induced by the order on $\bmdl$; this also induces a topology
on $\bmdl^n$. Given a definable set $X \subseteq \bmdl^n$, we write $\cl(X)$ for its topological closure, $\ito(X)$ for its interior,
and $\partial X \coloneqq \cl(X) \setminus X$ for its frontier (not to be mixed up with the boundary, which is also sometimes denoted by $\partial X$).
We call $X$ \emph{definably connected} if $X$ is not the disjoint union of two relatively closed (in $X$) definable subsets. The \emph{definable connected components} of $X$ are defined accordingly. (Any definable set 
in an \omin-minimal structure has finitely many definable connected components.)
\end{notn}

The topology on $\bmdl$ might be totally disconnected, so the usual notion of connectedness does not behave as desired. However, in the case $\bmdl = \R$, definably connected is the same as connected.

\begin{notn}[Derivatives]
For an open set $X \subseteq \bmdl^n$, derivatives of functions $f \colon X \fun \bmdl^m$ are defined as the usual limits. By \omin-minimality, derivatives exist almost everywhere.
For functions $f\colon X \fun \bmdl$, we write $\partial_i f$ for the derivative with respect to the $i$-th variable ($1 \le i \le n$), and
for $f = (f_1, \dots, f_m)\colon X \fun \bmdl^m$ and $a \in X$, we write
\[
\Jac_a f \coloneqq \begin{pmatrix}
\partial_1 f_1(a) & \cdots & \partial_n f_1(a)\\
\vdots &  & \vdots \\
\partial_1 f_m(a) & \cdots & \partial_n f_m(a)
\end{pmatrix}
\]
for the Jacobian matrix of $f$ at the point $a$. In the case $m = 1$, we also write $\nabla f(a)$ instead of $\Jac_a f$.
We define the class $C^p$ of $p$-fold continuously differentiable functions in the usual way.
\end{notn}

The notion of manifolds makes sense over (\omin-minimal) fields $\bmdl \ne \R$ only if one restricts to definable manifolds. All manifolds we will encounter will moreover be embedded.

\begin{notn}[Manifolds and tangent spaces]
A \emph{$d$-dimensional definable $C^p$ submanifold} of $\bmdl^n$ (for $d \le n$ and $p \ge 1$) is a
definable set $X \subseteq \bmdl^n$ such that there exists
a finite definable open cover of $X$ by sets $U_i$, each of which is in definable $C^p$-bijection with an open set $V_i \subseteq \bmdl^d$. The tangent space of $X$ at some $a \in X$ is denoted by $\bm T_a(X)$. (We consider $\bm T_a(X)$ as a subspace of $\bmdl^n$.)
\end{notn}

Note that $\bm T_a(X)$ is definable uniformly in $a$.

\subsection{Various definitions of Lipschitz stratifications}

We use the following notation and conventions for stratifications:

\begin{defn}[Stratifications]\label{defn.strat}
Let $X \sub \bmdl^n$ be a definable subset of dimension $d$. A \emph{definable stratification} of $X$ is a family
$\mdl X = (X^0 \subseteq X^1 \subseteq \dots \subseteq X^d = X)$ of closed definable subsets of $X$ satisfying the properties below. We set $X^{-1} \coloneqq \emptyset$. For $0 \le i \le d$, the set $\mathring X^i \coloneqq X^i \setminus X^{i-1}$ is called the \emph{$i$-th skeleton}, and each definably connected component
of each skeleton is called a \emph{stratum}. We call $\mdl X$ a stratification if the following conditions hold.
\begin{itemize}
  \item
   For each $i$, $\dim X^i \le i$;
  \item
   for each $i$, $\mathring X^i$ is either empty or a definable $C^1$ submanifold of $\bmdl^n$
   of dimension $i$ (not necessarily connected);
  \item
    for each stratum $S$, the topological closure $\cl(S)$ is a union of strata.
 \end{itemize}
\end{defn}

(Note that in the generality of power-bounded \omin-minimal structures, one cannot expect to obtain smooth strata.)

Mostowski's original definition of when a stratification is a Lipschitz stratification uses the notion of a chain: a sequence of points $(a^\ell)_{0 \le \ell \le m}$ that starts with an arbitrary point $a^0 \in X$, and where the remaining points lie in lower dimensional skeletons, but ``not too far from $a^0$'', and only in ``those skeletons $\mathring X^i$ which are much closer to $a^0$ than $X^{i-1}$''.
The precise inequalities specifying these distances are quite subtle.
There exists an equivalent definition involving Lipschitz vector fields \cite[Proposition~1.5]{Paru1988}, which avoids the subtleties of bounding the aforementioned distances. However, that definition quantifies over vector fields, which makes it less suitable for our model theoretic approach. Therefore, in this paper, we use the original definition in terms of chains. (More precisely, we use the simplified variant of that original definition given in \cite{Paru1988}.)

As already mentioned in the introduction, we will use methods from non-standard analysis to simplify the definition of Lipschitz stratifications: After having replaced $\bmdl$ by an elementary extension, we will define a valuation on $\bmdl$, which will allow us to replace the subtle bounds on distances by simple valuative inequalities.
However, that valuative definition is \emph{not} a straight-forward translation of Mostowski's definition in the usual non-standard analysis way. To make such a translation possible, one needs to
first modify Mostowski's definition in such a way that certain quantifiers become simpler.

To prove that our new definition is equivalent to the old one,
our strategy is as follows. We introduce two new variants of Mostowski's definition: one of them \emph{a priori} weaker and one of them \emph{a priori} stronger.
Both variants have simpler quantifiers, so that they can be translated to valuative versions. For those valuative versions, it will not be very hard to prove that the weak one implies the strong one, hence implying that all definitions are equivalent.

In the following, we start by giving all those definitions of Lipschitz stratifications which do not use the valuation.
The valuative versions are stated in Subsection~\ref{sect.lip-defs-na}, and the proofs of the equivalences are given in Subsection~\ref{sect.transl}.

Lipschitz stratifications are defined in terms of projections to the tangent spaces of the skeletons $\mathring X^i$; we first fix notation for those projection maps.

\begin{defn}
Given a definable stratification $\mdl X$ of a definable subset $X \sub \bmdl^n$ and
a point $a \in \mathring X^i$, let
\[
P_a : \bmdl^n \fun \bm T_a \mathring X^i
\]
be the orthogonal projection onto the tangent space of $\mathring X^i$ at $a$, considered as a map $\bmdl^n \fun \bmdl^n$.
\end{defn}

The various definitions of Lipschitz stratifications only differ in the way that certain constants are treated. To avoid writing almost the same definition three times (and to make it clear how exactly the definitions differ), we
introduce a general notion of a stratification $\mdl X$ ``satisfying the Mostowski Conditions for given constants''.
For readers who just want to understand one single definition of Lipschitz stratifications, one possible definition is encoded in the notation used for the constants: Increasing lowercase constants and decreasing uppercase constants both makes the Mostowski Conditions more restrictive; and $\mdl X$ is a Lipschitz stratification if no matter how big the lowercase constants are chosen, one can find values for the uppercase constants such that the Mostowski Conditions are satisfied (see Proposition~\ref{prop.new-lip} (2)).

Note: The Mostowski Conditions impose conditions on all chains, so a more restrictive notion of chains yields a less restrictive notion of Mostowski Conditions.

\begin{defn}[Chains and Mostowski Conditions]\label{defn.most}
	Let $\mdl X = (X^i)_i$ be a definable stratification (of a definable set $X \sub \bmdl^n$), and
	let $c, c', C', C'', C'''  \in \bmdl$ be given.

    A \emph{plain chain} (in $\mdl X$)
	is a sequence of points $a^0, a^1, \dots, a^m$ ($m \ge 0$) with
	$a^\ell \in \mathring X^{e_\ell}$, $e_0 > e_1 > \dots > e_m$
	satisfying the following conditions.
	\begin{enumerate}
		\item
		For $\ell = 1, \dots, m$, we have:
		\[
		\norm{a^0 - a^\ell} < c\cdot \dist(a^0, X^{e_\ell}).
		\]
		\item
		For each $i$ with $e_m \le i < e_0$, we have one of two different conditions (which should be considered as specifying which $i$ should be among the $e_\ell$ and which should not):
		\[
		\begin{cases}
		\dist(a^0, X^{i - 1}) \ge C'\cdot \dist(a^0, X^{i}) &
		\text{if } i \in \{e_1, \dots, e_m\}\\
		\dist(a^0, X^{i - 1}) < c'\cdot \dist(a^0, X^{i}) &
		\text{if } i \notin \{e_1, \dots, e_m\}.
		\end{cases}
		\]
	\end{enumerate}
	An \emph{augmented chain} (in $\mdl X$) consists of a plain chain $a^0, a^2, a^3, \dots, a^m$ ($m \ge 1$, $a^\ell \in \mathring X^{e_\ell}$) together with an additional point $a^1 \in \mathring X^{e_1}$, where $e_1 := e_0$, satisfying 
	\begin{equation}\label{eq.aug-cond}
	\norm{a^0 - a^1} \le \frac{\dist(a^0, X^{e_1 - 1})}{C''}.
	\end{equation}
	We say that $\mdl X$ \emph{satisfies the Mostowski Conditions for $(c, c', C', C'', C''')$}
    if the following two conditions hold:

    For every plain chain $(a^i)_{0 \le i \le m}$ with $m \ge 1$, we have
    \begin{equation}\tag{m1}
    \norm{(1 - P_{a^0})P_{a^1}P_{a^2} \ldots P_{a^m}} < \frac{C''' \norm{a^0 - a^1}}{\dist(a^0, X^{e_m - 1})},
    \end{equation}
    For every augmented chain $(a^i)_{0 \le i \le m}$ (with $m \ge 1$), we have
    \begin{equation}\tag{m2}
    \norm{(P_{a^0} - P_{a^1})P_{a^2}P_{a^3} \ldots P_{a^m}} < \frac{C''' \norm{a^0 - a^1}}{\dist(a^0, X^{e_m - 1})}.
    \end{equation}
We use the convention that if $X^{e_m-1}$ is empty, then in (m1) and (m2), we require the left hand side to be equal to $0$.
\end{defn}

Concerning nomenclature, note that what Parusinski calls a $c$-chain in \cite{Paru1988} is what we would call a plain chain of maximal length, using the same constant $c$ and $c' \coloneqq C' \coloneqq 2c^2$.
Also, we use different conventions regarding
the case when $X^{e_m-1}$ is empty.
(Our convention seems more natural to us, though it almost implies $X^0 \ne \emptyset$.)

Parusinskis's version of the definition of a Lipschitz stratification is the following:

\begin{defn}[Lipschitz stratifications]\label{defn.lip}
	Let $c > 1$ ($c \in \bmdl$) be given. A definable stratification $\mdl X$ is a \emph{Lipschitz stratification} if there exists a $C \in \bmdl$ such that $\mdl X$ satisfies the Mostowski Conditions for
	$(c, 2c^2, 2c^2, 2c, C)$.
\end{defn}

A priori, this notion seems to depend on the choice of $c$. However, it turns out that different choices of $c$ yield equivalent notions; this follows e.g.\ from \cite[Proposition~1.5]{Paru1988}. 

That we have $c' = C'$ in Definition~\ref{defn.lip}
means that Condition (2) in Definition~\ref{defn.most} uniquely specifies the set $\{e_1, e_2, \dots, e_m\}$ in terms of the initial point $a^0$ and the length $m$ of the chain.
However, a side effect of identifying $c'$ with $C'$ is
that the strength of the condition is not a monotone function in the value of the constant: a chain for some $c' = C'$ might neither stay a chain when making $c' = C'$ bigger, nor when making them smaller.
This has various disadvantages, the main one for us being that only monotone conditions can nicely be simplified by reformulating them in an elementary extension. Another consequence is that one has to be quite precise about the relations between the various constants: $c$ vs.\ $2c^2$ vs.\ $2c$.

In contrast, the following two equivalent characterizations are monotone in $c$ and $C$ in the above sense, and they are much more robust with respect to small modifications of Definition~\ref{defn.most}.

\begin{prop}[Characterizations of Lipschitz stratifications]\label{prop.new-lip}
	The following conditions on an \LT-definable stratification $\mdl X$ are equivalent:
	\begin{enumerate}
	\item $\mdl X$ is a Lipschitz stratification (in the sense of Definition~\ref{defn.lip}).
	\item For every $c \in \bmdl$, there exists a $C \in \bmdl$ such that $\mdl X$ satisfies the Mostowski Conditions for
	$(c, c, C, C, C)$.
	\item For every $c \in \bmdl$, there exists a $C \in \bmdl$ such that $\mdl X$ satisfies the Mostowski Conditions for
	$(c, c, 1, \frac1{c}, C)$.
	\end{enumerate}
\end{prop}

\private{I think there is an even weaker equivalent statement, which has the property that it follows from Definition~\ref{defn.lip} for any single $c$. But we don't need that weaker statement, and it's a bit less natural than (2) and (3). Is it worth mentioning that statement nevertheless? (The advantage would be that it provides a new proof that the $c$ in Mostowksi's definition doesn't matter.)}

\private{Newest thoughts about possible other equivalent defs:

very strong def: entirely remove the condition (2) from the defn of chains in 1.2.3. Thus allow even more chains.

On the other hand: very weak def: in cond (1) from the defn of chains, require $\norm{a - a^\ell} < (1 + 1/C) \dist (a, X^{e_\ell})$

Combining both new thoughts yields a defn completely without $c$, so less quantifiers.
}

The monotonicity in $c$ and $C$ means that both (2) and (3) in the proposition can be considered as statements about big $c$ and $C$, namely: ``No matter how big $c$ is, the Mostowski Conditions hold for all sufficiently big $C$.''

Characterization (2) imposes conditions only on very few chains: since $C$ can be assumed to be big compared to $c$, for most points $a^0 \in X$, neither of the two inequalities in Definition~\ref{defn.most} (2) holds, hence forbidding those $a^0$ as starting points of chains.
In contrast, \emph{every} sequence of points in decreasing skeletons is relevant in (3) for some $c$.
(Note that putting $C' = 1$ makes the first condition of Definition~\ref{defn.most} (2) trivially true.)
For these reasons, the implications (3) $\Rightarrow$ (1) $\Rightarrow$ (2) are very easy to prove,
assuming that we read Definition~\ref{defn.lip} as ``for every $c > 1$ there exists $C$'' (which we can, using the result that it is independent of $c$). The proof (2) $\Rightarrow$ (3) is harder; this will follow from Proposition~\ref{prop.new-lip-na}. In fact, this is a good example of a proof which becomes much easier after translating the statements to valuative ones in an elementary extension.

\begin{proof}[Proof of Proposition~\ref{prop.new-lip}
(3) $\Rightarrow$ (1)]
Let $c > 1$ be given (from Definition~\ref{defn.lip}). Then (3) yields a $C$ such that the Mostowski
Conditions hold for $(2c^2, 2c^2, 1, \frac1{2c^2}, C)$. Thus they also hold for $(c, 2c^2, 2c^2, 2c, C)$.
\end{proof}

\begin{proof}[Proof of Proposition~\ref{prop.new-lip}
(1) $\Rightarrow$ (2)]
Let $c$ be given (from (2)); without loss, $c > 1$. By Definition~\ref{defn.lip}, there exists $C$ such that the Mostowski Conditions hold for $(c, 2c^2, 2c^2, 2c, C)$.
Hence they also hold for $(c, c, C', C', C')$, where $C' \coloneqq \max\{C, 2c^2\}$.
\end{proof}

\begin{rem}
It is possible to translate the valuative proof of (2) $\Rightarrow$ (3)
given in Subsection~\ref{sect.transl} into a ``conventional'' proof within the original field $\bmdl$; we leave the details of this to the interested reader as an exercise. Such a translation in particular yields how, given a function $f_{(2)} \colon c \efun C$  witnessing (2), one obtains a function $f_{(3)} \colon c \efun C$ witnessing (3). Roughly,
$f_{(3)} = \underbrace{(f_{(2)} \circ g)\circ \dots \circ (f_{(2)} \circ g)}_{\dim X \text{ times}}$ for some simple function $g$.
\end{rem}

\subsection{Uniform families of Lipschitz stratifications}

As mentioned in the introduction, we will obtain Lipschitz stratifications uniformly in families, in a very strong sense. We now make this precise.

\begin{notn}[Definable families]
For the whole subsection, we fix a definable set $Q$ (say, a subset of $\bmdl^N$); all definable families are parametrized by $Q$:
A \emph{definable family} of subsets of $\bmdl^n$ is simply a definable subset $X \sub \bmdl^n \times Q$,
where we write
\[
X_q := \{x \in \bmdl^n : (x, q) \in X\}
\]
for the fiber at $q \in Q$.
\end{notn}

We also define families of stratifications in the obvious way:

\begin{defn}
Suppose that $X$ is a definable family of $d$-dimensional subsets of $\bmdl^n$ (for some fixed $d \le n$).
A \emph{definable family of stratifications} of $X$ is a tuple $\mdl X = (X^i)_{0 \le i \le d}$ of families of definable sets such that for each $q \in Q$,
$\mdl X_q \coloneqq (X^i_q)_{0 \le i \le d}$ is a stratification of $X_q$; $\mdl X$ is a \emph{definable family of Lipschitz stratifications} if each $\mdl X_q$ is a Lipschitz stratification.
\end{defn}

The more interesting concept is that of a family of stratifications that are \emph{uniformly} Lipschitz; this says that the constant $C$ appearing in the definition of Lipschitz stratifications can be chosen uniformly for the entire family. Here is the precise definition.

\begin{defn}[Uniformly Lipschitz stratifications]\label{defn.lip-unif}
A definable family $\mdl X$ of stratifications (of a definable family $X$ of sets) is a family of \emph{uniformly Lipschitz stratifications} if one of the following equivalent conditions holds:
\begin{enumerate}
\item For every $c \in \bmdl$ there exists a $C \in \bmdl$ such that for every $q \in Q$, $\mdl X_q$ satisfies the Mostowski Conditions for $(c, 2c^2, 2c^2, 2c, C)$.
\item For every $c \in \bmdl$ there exists a $C \in \bmdl$ such that for every $q \in Q$, $\mdl X_q$ satisfies the Mostowski Conditions for $(c, c, C, C, C)$.
\item For every $c \in \bmdl$ there exists a $C \in \bmdl$ such that for every $q \in Q$, $\mdl X_q$ satisfies the Mostowski Conditions for $(c, c, 1, \frac1c, C)$.
\end{enumerate}
\end{defn}

The above proofs of the non-uniform implications (3) $\Rightarrow$ (1) $\Rightarrow$ (2) (of Proposition~\ref{prop.new-lip}) also work without modification in the uniform case. The implication (2) $\Rightarrow$ (3)
is re-stated as (a part of) Proposition~\ref{prop.new-lip-na} and will be proved in Subsection~\ref{sect.transl}.

\begin{rem}
The reader may have noticed that in Definition~\ref{defn.lip-unif}~(1), we wrote ``for every $c$'', instead of fixing a $c > 1$, as in Definition~\ref{defn.lip}.
We believe that also the \emph{a priori} weaker versions with fixed $c$ are equivalent, but we didn't check that carefully.
\end{rem}

Now we can finally state the full version of the main result of this paper.

\begin{thm}[Uniformly Lipschitz stratifications]\label{thm.main}
Fix a power-bounded real closed field $\bmdl$ in a language $\LT$.
Suppose that $X$ is an \LT-definable family of closed, $d$-dimensional subsets of $\bmdl^n$ (i.e., $X$ is an \LT-definable subset of $\bmdl^{n} \times Q$, whose fibers $X_q \sub \bmdl^n$ are closed and $d$-dimensional, for $q \in Q$).
Then there exists an \LT-definable family $\mdl X = (X^i)_{0 \le i \le d}$ of uniformly Lipschitz stratifications of $X$
(in the sense of Definition~\ref{defn.lip-unif}).
\end{thm}

\subsection{Enlarging the model}
\label{sect.enlarge}

The conditions in Definition~\ref{defn.lip-unif} are clearly first order properties. Therefore,
when proving the implication (2) $\Rightarrow$ (3) and the existence of uniformly Lipschitz stratifications, we may work in an elementary extension.
More precisely, we will take the point of view that without loss, $\bmdl$ itself is already large, so that in particular, it possesses an elementary substructure $\smdl \precneqq \bmdl$.
It is not difficult to check that the convex closure of $\smdl$ within $\bmdl$ is a (non-trivial) valuation ring of $\bmdl$; we denote it by $\valring$.
Intuitively, elements of $\bmdl \setminus \valring$ may be regarded as ``infinite'' and elements in the maximal ideal of $\valring$ as ``infinitesimal''; more generally, bigger valuation means smaller order of magnitude, where two elements are considered as having the same order of magnitude if they differ at most by a factor from $\smdl^\times$. (Note that even if $\smdl$ is non-archimedean, we consider all its elements as having the same order of magnitude.)

It is a standard technique to study $\R$ by passing to an elementary extension. This implicitly uses the above valuation, but one usually considers definability only in the original language $\LT$. In contrast, in this paper, we will explicitly consider $\bmdl$ as a structure in the language expanded by a predicate for $\valring$. The model theory of such structures has been studied by van den Dries and Lewenberg \cite{DriesLew95,Dries:tcon:97}, and a
key ingredient to our proof of existence of Lipschitz stratifications builds on those results.

\begin{notn}[Valuation]\label{notn.val}
For the remainder of Section~\ref{sect.lip-defs}, we suppose that we have two \LT-structures $\smdl \precneqq \bmdl$.
We write $\valring \subseteq \bmdl$ for the valuation ring obtained as the convex closure of $\smdl$ in $\bmdl$, i.e.,
\[
\valring = \{a \in \bmdl : -b < a < b \text{ for some } b \in \smdl\}.
\]
We write $\Gamma \coloneqq \bmdl^\times / \valring^\times$ for the value group and
$\vv\colon \bmdl \fun \Gamma \cup \{\infty\}$ for the valuation.
Let $\LTv$ be the expansion of the language $\LT$ by a predicate for $\valring$.
\end{notn}

In \cite{DriesLew95,Dries:tcon:97}, the language $\LTv$ is denoted by $\mdl L_{\mathrm{convex}}$ and an \LTv-structure obtained from \omin-minimal structures $\smdl \precneqq \bmdl$ as in Notation~\ref{notn.val} is called ``$T$-convex'',
where $T$ is the theory of $\bmdl$ as an \LT-structure.
It has been proved in \cite{DriesLew95} that being $T$-convex is an elementary property, i.e., that for any \LTv-structure 
$\bmdl'$ which is elementarily equivalent to $\bmdl$,
the valuation ring $\valring[\bmdl'] \sub \bmdl'$ is also the convex closure of an \LT-elementary substructure $\smdl' \precneqq \bmdl'$. In particular, we can assume that $\bmdl$ is sufficiently saturated as an $\LTv$-structure (by possibly further enlarging both, $\smdl$ and $\bmdl$); this will be useful for (model theoretic) compactness arguments.

\begin{ass}\label{ass.sat}
For the remainder of the paper,
we assume that $\bmdl$ is sufficiently saturated, as a structure in the language $\LTv$.
\end{ass}

(To be precise, we will need $\bmdl$ to be $|\LTv|^+$-saturated.)

\begin{rem}
The result that being $T$-convex is an elementary property is only used for convenience, to be able to fix $\bmdl$ once and for all. In reality,
in those parts of the paper where we do need to consider elementary extensions of $\bmdl$ as an \LTv-structure (namely Theorem~\ref{thm.main-na} and its proof), we do not need $\valring$ to be the convex closure of an elementary substructure.
\end{rem}

\subsection{Valuative Notation}
\label{sect.notn.bmdl}

We fix some notation related to the newly introduced valuation. First of all, note that even when working with the language $\LTv$, all stratifications
we consider are $\LT$-definable (instead of $\LTv$-definable), and the notions of definable connectedness and definable manifolds still refer to the language $\LT$.

Now that we have a valuation, it is useful to also have valuative versions
of norms and distances; we use the following notation.
Note that by \cite[Proposition~4.3]{Dries:tcon:97}, the value group $\Gamma$ (with the induced structure) is \omin-minimal. In particular, suprema and infima of definable subsets of $\Gamma$ exist.

\begin{notn}[Valuative norms and distances]
For $a = (a_1, \dots, a_n) \in \bmdl^n$, we set
$\vv(a) \coloneqq \min_i \vv(a_i) = \vv(\norm{a})$.
If in addition, we have a definable set $X \sub \bmdl^n$, we set $\valdist(a, X) \coloneqq \sup_{x \in X}\vv(a - x) = \vv(\dist(a, X))$, where $\valdist(a, \emptyset) \coloneqq -\infty$.
For a matrix $M = (m_{ij})_{ij}$, we set $\vv(M) \coloneqq \min_{i,j} \vv(m_{ij})$.
\end{notn}

We recall some facts about those definitions.

\begin{lem}\label{lem.vvM}
Let $M$ and $N$ be matrices with coefficients on $\bmdl$. Then we have the following (where some statements implicitly impose conditions on the numbers of rows/columns of $M$ and $N$):
\begin{enumerate}
\item We have $\vv(MN) \ge \vv(M) + \vv(N)$ (and in particular $\vv(Ma) \ge \vv(M) + \vv(a)$ for $a \in \bmdl^n$).
\item The matrix $M$ lies in $\GL_n(\valring)$ iff $M \in \GL_n(\bmdl)$ and we have both $\vv(M) \ge 0$ and $\vv(M^{-1}) \ge 0$.
\item If $M \in \GL_n(\valring)$, then
$\vv(MN) = \vv(N)$ (and in particular $\vv(Ma) = \vv(a)$ for $a \in \bmdl^n$).
\item We have $\vv(M)=\vv(\norm{M})$, where $\norm{M}$ is the operator norm of $M$ (or, in fact, any other of the usual norms).
\end{enumerate}
\end{lem}

\begin{proof}
(1) Easy computation.

(2) Clear.

(3) Follows from (1) and (2).

(4) We have
\begin{equation}\label{eq.operator}
\vv(\norm{M}) + \vv(\norm{a}) \overset{(\star)}{\le} \vv(\norm{Ma})\overset{(\star\star)}{\ge}
\vv(M) + \vv(\norm{a})
\end{equation}
(using the definition of the operator norm to get $(\star)$, and using (1) to get $(\star\star)$).
By choosing $a$ such that $\norm{Ma} = \norm{M}\cdot \norm{a}$, we obtain an equality at $(\star)$ and hence
$\vv(\norm{M}) \ge \vv(M)$. To obtain $\vv(\norm{M}) \le \vv(M)$, we choose an $a$ which yields an equality at $(\star\star)$: if the $j$-th column of $M$ has an entry $m_{ij}$ satisfying $\vv(M) = \vv(m_{ij})$, then we can take $a$ to be the $j$-th standard basis vector.
\end{proof}

All balls we consider in this paper are valuative balls. We use the following notation.

\begin{notn}[Balls]
Given $a \in \bmdl^n$ and $\lambda \in \Gamma$, we write
\begin{align*}
B_{>\lambda}(a) &\coloneqq \{x \in \bmdl^n : \vv(x - a) > \lambda\}
\qquad \text{and}\\
B_{\ge\lambda}(a) &\coloneqq \{x \in \bmdl^n : \vv(x - a) \ge \lambda\}
\end{align*}
for the open and closed ball of valuative radius $\lambda$.
\end{notn}

\subsection{Valuative Lipschitz Stratifications}
\label{sect.lip-defs-na}

The valuation allows us to simplify Conditions~(2) and (3) of
Definition~\ref{defn.lip-unif} in the ``usual non-standard analysis way''. This leads to a valuative version of chains and Lipschitz stratifications, which we now introduce.

\begin{defn}[val-chains]\label{defn.na}
Fix a definable stratification $\mdl X = (X^i)_i$ of a definable set $X \subseteq \bmdl^n$.
A \emph{plain val-chain} (in $\mdl X$) is a sequence of points $a^0, \dots, a^m$
($m \ge 0$) with $a^\ell \in \mathring X^{e_\ell}$,
$e_0 > e_1 > \dots > e_m$ such that for all $1 \leq \ell \leq m$, we have
\begin{align}
\lambda_\ell := \vv(a^0 - a^{\ell})
 &> \valdist(a^0, X^{e_\ell - 1})
 \qquad \text{and}
 \label{eq.nai}
 \\
\vv(a^0 - a^{\ell}) &= \valdist(a^0, X^{e_{\ell-1}-1}).
 \label{eq.nae}
\end{align}
An \emph{augmented val-chain} (in $\mdl X$) is a sequence of points $a^0, \dots, a^m$ ($m \ge 1$) with $a^\ell \in \mathring X^{e_\ell}$,
$e_0 = e_1 > \dots > e_m$ such that
(\ref{eq.nai}) holds for $1 \leq \ell \leq m$ and
(\ref{eq.nae}) holds for $2 \leq \ell \leq m$.
By a \emph{val-chain}, we mean either a plain or an augmented one.

The numbers $e_\ell$ (for $0 \le \ell \le m$) are the \emph{dimensions} of the val-chain, and its
\emph{distances} are the valuations $\lambda_\ell$ ($1 \le \ell \le m$) together with $\lambda_{m+1} \coloneqq \valdist(a^0, X^{e_m - 1})$
(which might be $-\infty$).
\end{defn}

\begin{rem}\label{rem.val-chain-dim}
An equivalent way of characterizing a plain val-chain is the following. Choose any point $a^0$ in any skeleton $\mathring X^{e_0}$. Then choose the remaining points $a^\ell$ ($1 \le \ell \le m$) in skeletons $\mathring X^{e_\ell}$ as close as possible to $a^0$ in the valuative sense, where
$\{e_1, \dots, e_m\}$ consists of the $m$ biggest elements of
the set
$\{j \le e_0 : \valdist(a^0, X^{j}) > \valdist(a^0, X^{j - 1})\}$.
\end{rem}

\begin{rem}\label{rem.sub-val-chain}
By (\ref{eq.nai}), we have $\lambda_1 > \dots > \lambda_{m+1}$. This implies
\begin{alignat*}{2}
  \vv(a^k - a^{\ell}) &= \vv(a^0 - a^{\ell}) \qquad&& \text{for $0 \le k < \ell \le m$ and}\\
  \valdist(a^k, X^{j}) &= \valdist(a^0, X^{j}) \qquad&& \text{for } 0 \le k \le m \text{ and } j < e_k.
\end{alignat*}
In particular,
if $(a^\ell)_{0 \le \ell \le m}$ is a val-chain, then
any sub-sequence of the form $(a^\ell)_{k \le \ell \le m'}$ for $0 \le k \le m' \le m$ is also a val-chain
(which is always plain if $k \ge 1$). Moreover, if $(a^\ell)_{0 \le \ell \le m}$ is an augmented val-chain,
then $a^0, a^2, a^3, \dots, a^{m'}$ is a plain val-chain (for $1 \le m' \le m$).
\end{rem}

\begin{defn}[valuative Mostowski Conditions]
Let $\mdl X$ be a definable stratification and $(a^\ell)_{0 \le \ell \le m}$ a val-chain with distances $\lambda_\ell$. By the
\emph{valuative Mostowski Condition at $(a^\ell)_{\ell}$}, we mean one of the following two properties of $\mdl X$.
If $(a^\ell)_{\ell}$ is a plain val-chain, the condition is
  \begin{equation*}\tag{vm1}
  \vv((1 - P_{a^0})P_{a^1} \cdots P_{a^m}) \ge \lambda_1 - \lambda_{m+1};
  \end{equation*}
if $(a^\ell)_{\ell}$ is an augmented val-chain, the condition is
  \begin{equation*}\tag{vm2}
  \vv((P_{a^0} - P_{a^1})P_{a^2} \cdots P_{a^m}) \ge \lambda_1 - \lambda_{m+1}.
  \end{equation*}
\end{defn}

In the case $\lambda_{m+1} = -\infty$, the conditions are supposed to be read as ``$\vv(\ldots) =  \infty$'', i.e., the composition of the maps is $0$. If $X^0 \ne \emptyset$, then
$\lambda_{m+1} = -\infty$ implies $a^m \in X^0$, and we anyway have $P_{a^m} = 0$. However, if $X^0 = \emptyset$, then this is a very strong condition, so as for classical Lipschitz stratifications, one can almost never have $X^0 = \emptyset$.

\begin{defn}[valuative Lipschitz stratifications]\label{defn.nalip}
A definable stratification $\mdl X = (X^i)_i$ (of a definable set $X \sub \bmdl^n$) is a \emph{valuative Lipschitz stratification} if it satisfies the valuative Mostowski conditions at every val-chain.
\end{defn}

\begin{rem}\label{rem.0-trivial}
Whether or not in Definition~\ref{defn.nalip} one considers val-chains consisting of a single point (i.e., with $m = 0$) does not make a difference, since in that case, (vm1) is trivially true (since the right hand side is $0$).
\end{rem}

This is the notion of stratification we will use in the main proof in this paper, i.e., we will prove the existence of valuative Lipschitz stratifications. We will do this not only for \LT-definable sets $X$, but also for sets definable with additional parameters from $\bmdl$. By usual compactness arguments, this implies a family version of the result, and that in turn implies Theorem~\ref{thm.main} about the existence of uniformly Lipschitz stratifications. The details of these implications are given at the end of this subsection.

\begin{thm}[valuative Lipschitz stratifications]\label{thm.main-na}
Suppose that $\bmdl$ is a real closed field which is \omin-minimal and power-bounded as a structure in a language $\LT$, and suppose that $\LTv$ is an expansion of $\LT$ by
a predicate for the convex closure of an elementary substructure $\smdl \precneqq \bmdl$ (so $\bmdl$ is $T$-convex
in the sense of \cite{DriesLew95}).
Suppose that $X \subseteq \bmdl^n$ is a closed, \LTA-definable set for some parameter set $A \sub \bmdl$. Then there exists an \LTA-definable valuative Lipschitz stratification of $X$.
\end{thm}

The notion of a valuative Lipschitz stratification is just a reformulation of Proposition~\ref{prop.new-lip} (2) using the valuation, as we shall see below.
To provide a
similar reformulation of
Proposition~\ref{prop.new-lip} (3), we introduce ``weak val-chains''.
(Those are only used here and in the next subsection.)
Roughly, a weak val-chain is the same as a val-chain, except that additional intermediate points in skeletons of intermediate dimensions are allowed.

\begin{defn}
A \emph{weak val-chain} (plain or augmented) is the same as a val-chain,
except that the (strict) inequality (\ref{eq.nai}) is replaced by a weak one:
\begin{equation}
\vv(a^0 - a^{\ell})
\ge \valdist(a^0, X^{e_\ell - 1}).
 \label{eq.nawi}
\end{equation}
The val-chains from Definition~\ref{defn.na} will sometimes be called \emph{strict val-chains}, to emphasize the difference. The dimensions $e_\ell$, the distances $\lambda_\ell$ and the valuative Mostowski Conditions are defined in the same way as for strict val-chains.
\end{defn}

\begin{rem}\label{rem.wna}
In fact, imposing (\ref{eq.nawi}) is necessary only for $\ell = 1$ in augmented val-chains; in all other cases, (\ref{eq.nawi}) follows from
(\ref{eq.nae}) and $X^{e_{\ell-1}-1} \supseteq X^{e_\ell - 1}$.
\end{rem}

\begin{rem}
For weak val-chains, we only have weak inequalities
$\lambda_1 \ge \dots \ge \lambda_{m+1}$, and a weak val-chain is
strict iff all those inequalities between the $\lambda_i$ are strict.
\end{rem}

\begin{prop}\label{prop.new-lip-na}
Suppose that $\mdl X$ is an \LT-definable family of stratifications (of an \LT-definable family $X$ of subsets of $\bmdl^n$), parametrized by $q \in Q$ for some \LT-definable
$Q \sub \bmdl^N$.
Then following conditions are equivalent:
\begin{enumerate}
\item[(2)] Condition (2) of Definition~\ref{defn.lip-unif}.
\item[(3)] Condition (3) of Definition~\ref{defn.lip-unif}.
\item[(2')] For each $q$, $\mdl X_q$ is a valuative Lipschitz stratification (in the sense of Definition~\ref{defn.nalip}).
\item[(3')]
For each $q$, $\mdl X_q$
satisfies the valuative Mostowski Condition at every weak val-chain.
\end{enumerate}
\end{prop}

Note that for the implications ($x$) $\Rightarrow$ ($x$') to hold ($x = 2,3$), it is essential that $\mdl X$ is \LT-definable without parameters outside of $\smdl$; cf.\ Remark~\ref{rem.only-R0} below.
However, the implications ($x$') $\Rightarrow$ ($x$) seem to hold even for \LTA-definable $\mdl X$, where $A \sub \bmdl$. (We did not check the details.)

\private{
Here is a sketch:

First, prove that one use even weaker val-chains, which are allowed to skip skeletons. This seems to be easiest to prove using the characterization from \ref{prop.flag-lip}.

Moreover, prove that it suffices to consider chains satisfying $\vv(\norm{a^0 - a^\ell} - \dist(a^0, X^{e_\ell})) > \vv(\norm{a^0 - a^\ell}$. I'm not so sure how to do that.

Translating the notion of valuative chains obtained in this way into a non-valuative notion yields a condition using only $C$ (and no $c$ at all).
Then the implication ($x$') $\Rightarrow$ ($x$) becomes trivial.
}

\begin{exam}\label{exam.cone}
If $X \subseteq \bmdl^3$ is the cone defined by $r^2x^2 = y^2 + z^2$ for some $r \in \bmdl$ of strictly positive valuation, then $X^0 = X^1 = \{(0,0,0)\}$ defines a Lipschitz stratification of $X$, but not a valuative Lipschitz stratification; see Figure~\ref{fig.cone}.
\end{exam}

\begin{figure}
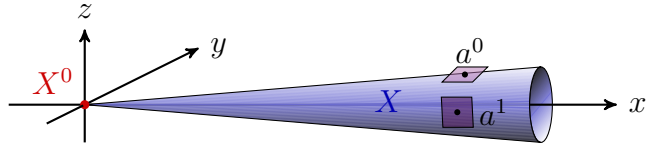

\begingroup
\def\co{white!70!black!60!blue}
\begin{ctikzpicture}
  \draw[preaction={bottom color=white,top color=black,middle color=\co,opacity=.7,shading angle=20}] (6,0) ellipse (0.15 and 0.5);

  \draw[axis arrow] (-1,0) -- (7,0) node[right]{$x$};
  \draw[axis arrow] (0,-.5) -- (0,1) node[above]{$z$};
  \draw[axis arrow] (0,0) -- (1.5,.75) node[right]{$y$};

  \def\conepath{(0,0) -- (6,.5) arc (90:270: 0.15 and 0.5) -- cycle}
  \fill[color=white,opacity=.4] \conepath;

  \begin{scope}
    \clip (0,0) rectangle (7,.6);
    \shade[bottom color=black,top color=white,middle color=\co,opacity=.7, shading angle=5] \conepath;
  \end{scope}
  \begin{scope}
    \clip (0,0) rectangle (7,-.6);
    \shade[bottom color=black,top color=white,middle color=\co,opacity=.7, shading angle=-5] \conepath;
  \end{scope}
  \draw[line join=round] \conepath;

  \node[S0text] at (4,.05) {$X$};

  \draw[preaction={fill,color=white!10!black!50!violet,opacity=.5}]
    (4.71,-.3) -- (5.11,-.3) -- (5.09,.1) -- (4.69,.1) -- cycle;
  \fill(4.9,-.1) circle (.04) node[right] {$\,\,a^1$};

  \draw[preaction={fill,color=white!90!black!50!violet,opacity=.5}]
    (4.7,.3) -- (5.1,.3) -- (5.3,.5) -- (4.9,.5) -- cycle;
  \fill(5,.4) circle (.04) node[above] {$\,\,a^0$};

  \fill[S2set] (0, 0) circle (.06) node[S2text,anchor=south east] {$X^0$};
  \draw[axis noarrow] (-.5,-.25) -- (-.02,-.01);

\end{ctikzpicture}
\endgroup
\caption{This is a Lipschitz stratification but not a valuative Lipschitz stratification,
since the two tangent spaces of the augmented val-chain $a^0 = (1, 0, r), a^1 = (1,r,0)$ are too far apart from each other; see Example~\ref{exam.cone}.}
\label{fig.cone}
\end{figure}

As promised, here is the precise argument on how to deduce Theorem~\ref{thm.main} from Theorem~\ref{thm.main-na} and Proposition~\ref{prop.new-lip-na}.

\begin{proof}[Proof of Theorem~\ref{thm.main}]
Let an \LT-definable family $X$ of closed $d$-dimensional subsets of $\bmdl^n$ be given (parametrized by $q \in Q$);
we would like to find a family $\mdl X$ of uniformly Lipschitz stratifications (Defintion~\ref{defn.lip-unif}) of $X$. By Proposition~\ref{prop.new-lip-na}, this is equivalent to $\mdl X_q$ being a valuative Lipschitz stratification for each $q \in Q$.

For each $q \in Q$, Theorem~\ref{thm.main-na} provides an
\LTq-definable valuative Lipschitz stratification $\mdl X_q$ of
$X_q$. By a standard compactness argument, we may assume that those $\mdl X_q$ are definable uniformly in $q$, i.e., that they are the fibers of an \LT-definable family $\mdl X$ of stratifications, as desired.

The details of the compactness argument are as follows.
For each $\hat q\in Q$, there exist \LT-formulas $\phi^i_{\hat q}(x, y)$ ($0 \le i \le d$) such that 
the $\phi^i_{\hat q}(x, \hat q)$ define a valuative Lipschitz stratification of $X_{\hat q}$. Fix one $\hat q$ and consider the set $U_{\hat q}$
of those $q \in Q$ such that $(\phi^i_{\hat q}(x, q))_i$ defines a valuative Lipschitz stratification of $X_{q}$. Since being a valuative Lipschitz stratification is a first order property, 
$U_{\hat q}$ is \LT-definable. Finitely many sets $U_{\hat q_1}, \dots, U_{\hat q_\ell}$ suffice to cover $Q$, since otherwise, the complements $Q \setminus U_{\hat q}$ would form a partial type, which is satisfied by some $q_0 \in Q$ (since $\bmdl$ is sufficiently saturated by Assumption~\ref{ass.sat}), contradicting $q_0 \in U_{q_0}$.
Now use the formulas $\phi^i_{\hat q_1}(x, y), \dots, \phi^i_{\hat q_\ell}(x, y)$ to define $\mdl X$; more precisely, given $q \in Q$, let $X^i_q$ be defined by $\phi^i_{\hat q_j}(x, q)$, where $j$ is minimal with $q \in U_{\hat q_j}$.
\end{proof}

\subsection{Equivalence of various definitions}
\label{sect.transl}

We will now prove Proposition~\ref{prop.new-lip-na}.
More precisely, we will prove the following implications:
\begin{equation}\label{eq.new-lip-na}
\begin{array}{ccc}
(2) & \iff & (2')\\
 & & \Downarrow\\
(3) & \iff & (3')
\end{array}
\end{equation}
Note that the right hand $\Uparrow$ is trivial, and anyway,
we already proved $\Uparrow$ on the left hand side.
Both horizontal $\iff$ are simple applications of a standard method from non-standard analysis which we recall now:

\begin{lem}[Translating: with/without valuation]\label{lem.nsa}
	Suppose that $Z$ is \LT-definable and that $f, g\colon Z \fun {\bmdl}_{\ge 0}$ are two \LT-definable functions.
	Then the following are equivalent:
	\begin{enumerate}
		\item For every $c \in \bmdl_{\ge 0}$, there exists $C \in \bmdl_{\ge 0}$
		such that for every $z \in Z$, $f(z) \le c$ implies $g(z) \le C$.
		\item
		For every $z \in Z$, $\vv(f(z)) \ge 0$ implies
		$\vv(g(z)) \ge 0$.
	\end{enumerate}
\end{lem}

\begin{proof}
Statement (1) is an \LT-sentence in $\bmdl$, so it is equivalent to the same sentence in $\smdl$; we will use this version of (1). For the proof of this lemma, we assume without loss that all elements of $\smdl$ are constants of $\LT$.

(1) $\Rightarrow$ (2): Let $z_0 \in Z$ be given such that
$\vv(f(z_0)) \ge 0$. Then $f(z_0) \le c$ for some $c \in \smdl$
(by definition of the valuation). By (1) in $\smdl$, there exists a $C \in (\smdl)_{\ge 0}$ such that
$\smdl \models \forall z \in Z: (f(z) \le c \rightarrow g(z) \le C)$. This sentence also holds in $\bmdl$ (where $c, C$ are considered as constants from $\LT$), hence $f(z_0) \le c$ implies
$g(z_0) \le C$. This in turn implies $\vv(g(z_0)) \ge 0$.

(2) $\Rightarrow$ (1): Let
$c \in (\smdl)_{\ge 0}$ be given. We consider
``$\exists C: \forall z \in Z: (f(z) \le c \rightarrow g(z) \le C)$''
as a sentence where $c$ is a constant from $\LT$; it suffices
to prove that this sentence holds in $\bmdl$. But indeed:
since $f(z) \le c$ implies $\vv(f(z)) \ge 0$,
we have $\vv(g(z)) \ge 0$, so we can take any $C \in \bmdl_{\ge 0}$ of negative valuation.
\end{proof}

\begin{rem}\label{rem.only-R0}
For this lemma to be true, it is important that $Z$, $f$ and $g$ are definable using parameters only from $\smdl$.
\end{rem}

\begin{rem}\label{rem.nsa0}
An easy special case of Lemma~\ref{lem.nsa} is the one with $f = 0$: An \LT-definable function $g\colon Z \fun \bmdl$ is bounded iff it satisfies $\vv(g(z)) \ge 0$ for all $z \in Z$.
\end{rem}

\begin{proof}[Proof of Proposition~\ref{prop.new-lip-na}, (2) $\iff$ (2')]
This is just a straight-forward application of Lemma~\ref{lem.nsa}. The details are as follows.

Let the family $\mdl X = (X^i)_i$ of stratifications be fixed
(parametrized by $q \in Q$), and let $Z$ be the
set of all tuples
$z$ of the form $(q, (a^\ell)_{0 \le \ell \le m})$, with $q \in Q$, $a^\ell \in \mathring X_q^{e_\ell}$, $e_0 \ge e_1 > e_2 > \dots > e_m$,
and $m \ge 1$. (We consider $Z$ as an \LT-definable set.)
Given $c, C \in \bmdl$,
such a $z \in Z$ witnesses that our family $\mdl X$ violates the Mostowski conditions for $(c, c, C, C, C)$ if
\begin{itemize}
  \item
  $(a^\ell)_{0 \le \ell \le m}$ is a chain in $\mdl X_q$ (either plain or augmented), i.e.:
  \begin{align}
  \frac{\norm{a^0 - a^\ell}}{\dist(a^0, X_q^{e_\ell})} &< c
  \qquad\text{for } \ell =
  \begin{cases}
   1, \dots, m  &\text{if } e_0 > e_1\\
   2, \dots, m  &\text{if } e_0 = e_1\\
  \end{cases}
  \label{eq.c1}
  \\
  \frac{\dist(a^0, X_q^{i - 1})}{\dist(a^0, X_q^{i})} &\ge C
  \qquad\text{for } i \in
  \begin{cases}
   \{e_1, \dots, e_m\}  &\text{if } e_0 > e_1\\
   \{e_2, \dots, e_m\}  &\text{if } e_0 = e_1\\
  \end{cases}
  \label{eq.C1}
	  \\
  \frac{\dist(a^0, X_q^{i - 1})}{\dist(a^0, X_q^{i})} &< c
  \qquad\text{for }
 e_m \le i \le e_0,
  i \notin \{e_0, \dots, e_m\}
  \label{eq.c2}
  \\
  \frac{\dist(a^0, X_q^{e_1 - 1})}{\norm{a^0 - a^1}} &\ge C
  \qquad \text{in the case $e_0 = e_1$}
  \label{eq.C2}
  \end{align}
  \item and either (m1) or (m2) is violated:
  \begin{align}
  \frac{\norm{(1 - P_{a^0})P_{a^1} \ldots P_{a^m}}\dist(a^0, X_q^{e_m - 1})}{\norm{a^0 - a^1}} &\ge C
  \qquad \text{in the case $e_0 > e_1$}
  \label{eq.C3}
  \\
  \frac{\norm{(P_{a^0} - P_{a^1})P_{a^2} \ldots P_{a^m}}\dist(a^0, X_q^{e_m - 1})}{\norm{a^0 - a^1}} &\ge C
  \qquad \text{in the case $e_0 = e_1$}.
  \label{eq.C4}
  \end{align}
\end{itemize}
Define $f(z)$ to be the maximum of all the left hand sides of (\ref{eq.c1}) and (\ref{eq.c2}) (for all $\ell$ and $i$) and
$g(z)$ to be the minimum of all the (relevant) left hand sides of
(\ref{eq.C1}), (\ref{eq.C2}), (\ref{eq.C3}), (\ref{eq.C4}).
Then Condition~(2) of Definition \ref{defn.lip-unif} is exactly (1) of Lemma~\ref{lem.nsa}, and (2) of Lemma~\ref{lem.nsa} says that there is no $z \in Z$ satisfying the following modification of
(\ref{eq.c1}) -- (\ref{eq.C4}): replace ``$\star < c$'' by
``$\vv(\star) \ge 0$'' and ``$\star \ge C$'' by
``$\vv(\star) < 0$''.

In this modified version, (\ref{eq.c1}) -- (\ref{eq.C2}) state that $(a^i)_i$ is a val-chain and (\ref{eq.C3}), (\ref{eq.C4}) state that the corresponding valuative Mostowksi Condition is violated. Thus Lemma~\ref{lem.nsa}~(2) expresses that $\mdl X$ is a valuative Lipschitz stratification.
\end{proof}

\begin{proof}[Proof of Proposition~\ref{prop.new-lip-na}, (3) $\iff$ (3')]
The proof is almost the same as for for (2) $\iff$ (2'). The only differences are that (\ref{eq.C1}) disappears and that
(\ref{eq.C2}) is replaced by
\begin{equation}\label{eq.cw}
\frac{\norm{a^0 - a^1}}{\dist(a^0, X_q^{e_1 - 1})} \le c.
\end{equation}
Lemma~\ref{lem.nsa} turns (\ref{eq.cw}) into (\ref{eq.nawi}) for $\ell = 1$, so we obtain exactly weak val-chains (see also Remark~\ref{rem.wna}).
\end{proof}

\begin{proof}[Proof of Proposition~\ref{prop.new-lip-na}, (2') $\Rightarrow$ (3')]
We assume that every strict val-chain satisfies the valuative Mostowski Conditions, and we have to prove the same for weak val-chains.
Let $a^0, \dots, a^m$ be a weak val-chain with dimensions $e_i$ and distances $\lambda_i$. We do an induction over $m$.
If this is already a strict val-chain, there is nothing to prove.
Otherwise, choose any $\ell$ such that $\lambda_\ell = \lambda_{\ell + 1}$ ($1 \le \ell \le m$). Let us first suppose that $(a^i)_i$ is a plain (weak) val-chain. Set
\begin{align*}
Q &\coloneqq
(1 - P_{a^0})P_{a^1}\cdots P_{a^{\ell-1}} \qquad \text{and}\\
Q' &\coloneqq P_{a^{\ell+1}}\cdots P_{a^{m}};
\end{align*}
we need to show that
\begin{equation}\label{eq.weak-val-chain-show}
\vv(Q P_{a^\ell}Q') \ge \lambda_1 - \lambda_{m+1}.
\end{equation}
The sub-sequence $a^0, \dots, a^{\ell - 1}, a^{\ell+1},\dots, a^m$
is still a weak val-chain, and by induction, it satisfies the Mostowski Conditions, i.e.:
\begin{equation*}
\vv(Q Q') \ge \lambda_1 - \lambda_{m+1}.
\end{equation*}
Moreover, we have $\vv(Q) \ge \lambda_1 - \lambda_{\ell}$
(by the inductive hypothesis for $a^0, \dots, a^{\ell - 1}$)
and $\vv((1 - P_{a^\ell})Q') \ge \lambda_{\ell+1} - \lambda_{m+1}$ (by the inductive hypothesis for
$a^{\ell},\dots, a^m$).
Combining these three inequalities (and using $\lambda_\ell = \lambda_{\ell + 1}$) yields
(\ref{eq.weak-val-chain-show}), since
$Q P_{a^\ell}Q' = Q (1 - P_{a^\ell})Q' - Q Q'$.

Now suppose that $a^0, \dots, a^m$ is an augmented val-chain.
If $\ell \ge 2$, then the argument is exactly the same as for plain val-chains, with
\[
Q = (P_{a^0}-P_{a^1})P_{a^2}\cdots P_{a^{\ell-1}}.
\]
In the case $\ell = 1$, define $Q'\coloneqq P_{a^{2}}\cdots P_{a^{m}}$ (as before).
The Mostowski conditions for $a^0,a^2,\dots, a^m$ and $a^1,a^2,\dots, a^m$ imply
$\vv((1-P_{a^0})Q') \ge \lambda_1 - \lambda_{m+1}$
and
$\vv((1-P_{a^1})Q') \ge \lambda_1 - \lambda_{m+1}$;
this implies
\begin{equation*}
\vv((P_{a^0} - P_{a^1}) Q') \ge \lambda_1 - \lambda_{m+1},
\end{equation*}
which is what we had to show.
\end{proof}

\subsection{A $\GL_n$-invariant definition}

To prove the existence of valuative Lipschitz stratifications, we will use yet another (equivalent) definition, which is more natural in the sense that it is clearly invariant under $\GL_n(\valring)$. Note that Definition~\ref{defn.nalip} (the definition of valuative Lipschitz stratifications) is already pretty close to being $\GL_n(\valring)$-invariant, since $\GL_n(\valring)$ preserves valuations (by Lemma~\ref{lem.vvM}). To make it entirely $\GL_n(\valring)$-invariant, one only needs to get rid of the orthogonal projections used to express that certain
tangent spaces are close to each other; this is what we will do now.

That valuative Lipschitz stratifications
are $\GL_n(\valring)$-invariant directly implies that classical Lipschitz stratifications are $\GL_n(\bmdl)$-invariant; even though this is not a new result, we formulate it as Corollary~\ref{cor.GLR-invar}.

There exists a natural valuative metric on the Grassmannians. It can be defined in many equivalent ways, some of which use orthogonal projections, and others being clearly $\GL_n(\valring)$-invariant. We leave the proof of the equivalences to the reader.

\begin{defn}\label{defn.Delta}
For subspaces $W_1, W_2 \subseteq \bmdl^n$ of the same dimension, set $\Delta(W_1, W_2) \coloneqq \vv(P_1 - P_2)$, where $P_i$ is the orthogonal projection onto $W_i$.
\end{defn}

\begin{lem}\label{lem.Delta}
For subspaces $W_1, W_2 \subseteq \bmdl^n$, both of dimension $d$ and for any $\lambda \in \Gamma$, the following are equivalent:
\begin{enumerate}
\item $\Delta(W_1, W_2) \ge \lambda$
\item There exist $\phi_1, \phi_2 \in \Hom(\bmdl^d, \bmdl^n)$
   with $\vv(\phi_1 - \phi_2) \ge \lambda$ and $\im \phi_i = W_i$.
\item For every $w_1 \in W_1$ there exists $w_2 \in W_2$
such that $\vv(w_2 - w_1) \ge \vv(w_1) + \lambda$.
\end{enumerate}
\end{lem}

\private{
Idea of proof: (1) $\Rightarrow$ (2) is easy, and (2) $\Rightarrow$ (3) is not difficult either. For (3) $\Rightarrow$ (1), I had to get my hands dirty: assume without loss that $W_1 = \bmdl^d \times \{0\}^{n-d}$; check that $\Delta(W_1^\perp, W_2^\perp) \ge \lambda$; do some computations using the standard basis, etc.
}

The Mostowski Condition bounding $\vv((1 - P_{a^0})P_{a^1})$
can be considered as the statement that
$\bm T_{a^0}\mathring X^0$ contains a subspace which is a good approximation of $\bm T_{a^1}\mathring X^1$.
The following characterization of valuative Lipschitz stratifications
is a generalization of this point of view to arbitrary val-chains; see
Figure~\ref{fig.flags} for an overview over all sub-spaces.

\begin{prop}[Valuative Lipschitz stratifications using flags]\label{prop.flag-lip}
The following conditions on a definable stratification $\mdl X = (X^i)_i$ are equivalent:
\begin{enumerate}
\item $\mdl X$ is a valuative Lipschitz stratification (in the sense of Definition~\ref{defn.nalip}).
\item For every val-chain $(a^i)_{i \le m}$ (plain or augmented) with dimensions $e_i$ and distances $\lambda_i$,
there exist vector spaces $V_{k,\ell}$ for
$0 \le k \le \ell \le m$ with the following properties:
\begin{alignat}{4}
V_{k,m} \subseteq V_{k,m-1} \subseteq \dots &\subseteq V_{k,k+1} \subseteq V_{k,k} = \bm T_{a^k}\mathring X^{e_k}
\quad
&&\text{for }0 \le k \le m
\label{eq.p-flags-inc}
\\
\dim V_{k,\ell} &= e_\ell
&& \text{for } 0 \le k \le \ell \le m
\label{eq.p-flags-dim}
\\
\Delta(V_{k,\ell}, V_{k+1,\ell}) &\ge \lambda_{k+1} - \lambda_{\ell+1}
&& \text{for } 0 \le k < \ell \le m ,
\label{eq.p-flags-dist}
\end{alignat}
\end{enumerate}
\end{prop}

\begin{figure}
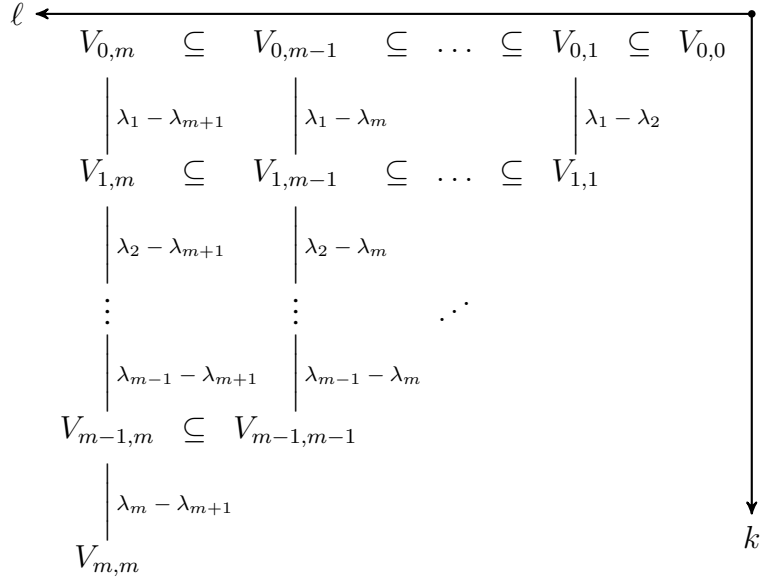

\begingroup
\newcommand{\di}[2]{\bigg|\rlap{{\scriptsize $\,\lambda_{#1} - \lambda_{#2}$}}}
\begin{ctikzpicture}
\node (all) at (0,0) {%
$\begin{array}{ccccccccc}
V_{0,m} & \subseteq &  V_{0,m-1} & \subseteq &  \dots & \subseteq & V_{0,1} & \subseteq & V_{0,0}\\[1ex]
 \di1{m+1}  &         &    \di1{m} &       &        &         &  \di12  &         &        \\[1ex]
V_{1,m} & \subseteq &  V_{1,m-1} & \subseteq &  \dots & \subseteq & V_{1,1} &         &        \\[1ex]
 \di2{m+1}  &         &     \di2{m}      \\
  \vdots     &         &   \vdots     &         &          \udots     \\
\di{m-1}{m+1}     &         &    \di{m-1}{m}   \\[1ex]
V_{m-1,m}& \subseteq &  V_{m-1,m-1}  \\[1ex]
\di{m}{m+1}    \\[1ex]
V_{m,m}
\end{array}$};
\fill (all.north east) circle (.05);
\draw[axis arrow] (all.north east) -- (all.north west) node[left] {$\ell$};
\draw[axis arrow] (all.north east) -- ($(all.south east)+(0,1cm)$) node[below] {$k$};
\end{ctikzpicture}
\endgroup

\caption{Diagramatic representation of the vector spaces appearing in Proposition~\ref{prop.flag-lip} (and Lemma~\ref{lem.flags}); the labels of the vertical lines indicate the distance between the two corresponding spaces.
}\label{fig.flags}
\end{figure}

\begin{rem}
As in Definition~\ref{defn.nalip}, the above Condition~(2) is trivial for val-chains consisting of a single point.
\end{rem}

The proof of (1) $\Rightarrow$ (2) is easy:

\begin{proof}[Proof of Proposition~\ref{prop.flag-lip}, (1) $\Rightarrow$ (2)]
Given a val-chain $(a^i)_i$, we set $V_{k,\ell} \coloneqq \im(Q_{k.\ell})$, where
\begin{equation}\label{eq.Qdef}
Q_{k,\ell} \coloneqq
\begin{cases}
P_{a^k}P_{a^{k+1}}\cdots P_{a^\ell}
  & \text{if } (a^i)_{k \le i \le \ell} \text{ is a plain val-chain}
\\
P_{a^k}P_{a^{k+2}}P_{a^{k+3}}\cdots P_{a^\ell}
  & \text{if } (a^i)_{k \le i \le \ell} \text{ is an augmented val-chain}.
\end{cases}
\end{equation}
Note that $(a^i)_{k \le i \le \ell}$ is an augmented val-chain iff the entire sequence $(a^i)_{0 \le i \le m}$ is augmented, $k = 0$ and $\ell \ge 1$. In particular, if $(a^i)_{0 \le i \le m}$ is augmented, then $Q_{0,1} = P_{a^0}$.

Condition (\ref{eq.p-flags-inc}) follows directly from
this definition of $V_{k,\ell}$.
Now fix $0 \le k < \ell \le m$.
Then the valuative Mostowski Conditions for the subchain $(a^i)_{k \le i \le \ell}$ imply
\begin{equation}\label{eq.most-imp}
\vv(Q_{k,\ell} - Q_{k+1,\ell}) \ge \lambda_{k+1} - \lambda_{\ell+1};
\end{equation}
indeed, we have
\[
Q_{k,\ell} - Q_{k+1,\ell}
=
\begin{cases}
 -(1 - P_{a^k})P_{a^{k+1}}\cdots P_{a^\ell}
& \text{if $(a^i)_{k \le i \le \ell}$ is plain}
\\
(P_{a^{k}} - P_{a^{k+1}})P_{a^{k+2}}\cdots P_{a^\ell}
& \text{if $(a^i)_{k \le i \le \ell}$ is augmented.}
\end{cases}
\]

From (\ref{eq.most-imp}), we first deduce (\ref{eq.p-flags-dim}): On the one hand, (\ref{eq.Qdef}) directly implies $\dim V_{k,\ell} \le e_\ell$.
(Note that in the above case where $Q_{0,1} = P_{a^0}$, we have $e_0 = e_1$.) On the other hand, repeatedly using (\ref{eq.most-imp}) yields $\vv(Q_{k,\ell} - Q_{\ell,\ell}) \ge \min_{k \le i < \ell} (\lambda_{i+1} - \lambda_{\ell+1}) > 0$, so since $Q_{\ell,\ell} = P_{a^\ell}$ is the identity on
$V_{\ell,\ell} = \bm T_{a^\ell}\mathring X^\ell$,
we have $\ker Q_{k,\ell} \cap V_{\ell,\ell} = 0$
and hence $\dim V_{k,\ell} = \rk Q_{k,\ell} \ge e_\ell$.

Finally, (\ref{eq.most-imp}) implies (\ref{eq.p-flags-dist})
using Lemma~\ref{lem.Delta} (2) $\Rightarrow$ (1).
\end{proof}

We formulate the main part of the proof of the other direction as a general lemma about flags.

\begin{lem}\label{lem.flags}
Fix $m \ge 1$ and $\lambda_1 \ge \dots \ge \lambda_{m+1} \in \Gamma$.
Suppose that for each $0 \le k \le m$, we have a (partial) flag
\begin{equation}
 V_{k,m} \subseteq V_{k,m-1} \subseteq \dots \subseteq V_{k,k+1} \subseteq V_{k,k} \subseteq \bmdl^n
\end{equation}
satisfying
\begin{equation}\label{eq.flags-dist-cond}
\Delta(V_{k,\ell}, V_{k+1,\ell}) \ge \lambda_{k+1} - \lambda_{\ell+1}  \quad \text{for } 0 \le k < \ell \le m.
\end{equation}
(In particular, we assume $\dim V_{k,\ell} = \dim V_{k+1,\ell}$.)
Let $P_{k,\ell}\colon \bmdl^n \fun \bmdl^n$ denote the orthogonal projection onto $V_{k,\ell}$.
Under those assumptions, we have
 \begin{equation}\label{eq.flags1}
 \vv((1 - P_{0,0})P_{1,1}P_{2,2} \cdots P_{m,m}) \ge \lambda_1 - \lambda_{m+1}.
\end{equation}
If moreover $\dim V_{1,1} = \dim V_{0,0}$ (which in particular implies $\dim V_{0,1} = \dim V_{0,0}$ and hence $V_{0,1} =  V_{0,0}$),
then we moreover have
\begin{equation}\label{eq.flags2}
\vv((P_{0,0} - P_{1,1})P_{2,2}P_{3,3} \cdots P_{m,m}) \ge \lambda_1 - \lambda_{m+1}.
\end{equation}
\end{lem}

Before proving that lemma, we quickly check that it indeed implies the other direction of the proposition.

\begin{proof}[Proof of Proposition~\ref{prop.flag-lip}, (2) $\Rightarrow$ (1)]
Let $(a^i)_i$ be a val-chain. By (2) of the proposition,
we have vector spaces $V_{k,\ell}$ for $0 \le k \le \ell \le m$ satisfying the prerequisites of Lemma~\ref{lem.flags}. If $(a^i)_i$ is plain, then
the Mostowski Condition (vm1) is (\ref{eq.flags1});
if $(a^i)_i$ is augmented, then $\dim V_{1,1} = \dim V_{0,0}$ and the Mostowski Condition (vm2) is (\ref{eq.flags2}).
\end{proof}

\begin{proof}[Proof of Lemma~\ref{lem.flags}]
We will prove the following two inequalities by downwards induction on $k$:
\begin{align}
 \vv((1 - P_{k,k})\cdot P_{k+1,k+1} \cdots P_{m,m}) &\ge \lambda_{k+1} - \lambda_{m+1}
\qquad\text{for } 0 \le k \le m \text{ and}
\label{eq.flags-no-ind}
\\
 \vv((P_{k,i} - P_{k,i+1})\cdot P_{k+1,k+1} \cdots P_{m,m}) &\ge \lambda_{i+1} - \lambda_{m+1}
\qquad\text{for } 0 \le k \le i < m
\label{eq.flags-ind}
.
\end{align}
Note that (\ref{eq.flags-ind}) will be needed in the inductive proof of (\ref{eq.flags-no-ind}).
Before we carry out this induction, let us already check that (\ref{eq.flags-no-ind}) implies the lemma:
(\ref{eq.flags1}) is just (\ref{eq.flags-no-ind}) for $k = 0$. To get (\ref{eq.flags2}), we plug in
\begin{equation}
P_{0,0} - P_{1,1} = P_{0,0}\cdot(1 - P_{1,1}) + (P_{0,0} - 1)\cdot P_{1,1}
.
\end{equation}
The second summand obtained in this way is just (\ref{eq.flags1}) (up to sign) and hence has valuation as required.
In the first summand, we repeat the factor $(1 - P_{1,1})$ twice (which we may, since it is a projection), so it is equal to
\begin{equation}
\underbrace{P_{0,0}\cdot(1 - P_{1,1})}_{\text{(a)}}\cdot\underbrace{(1 - P_{1,1})P_{2,2}P_{3,3} \cdots P_{m,m}}_{\text{(b)}}.
\end{equation}
By (\ref{eq.flags-no-ind}), (b) has valuation at least $\lambda_2 - \lambda_{m+1}$
so it suffices to show that (a) has valuation at least $\lambda_1 - \lambda_2$.
But indeed, $\vv(P_{1,1} - P_{0,1}) \ge \lambda_1 - \lambda_2$ and
$P_{0,0}\cdot(1 - P_{0,1}) = 0$ since $V_{0,1} = V_{0,0}$
(by the assumption $\dim V_{1,1} = \dim V_{0,0}$).
Thus it remains to prove (\ref{eq.flags-no-ind}) and (\ref{eq.flags-ind}).

For $k = m$, (\ref{eq.flags-no-ind}) is trivial
(since $\lambda_{m+1} - \lambda_{m+1} = 0$) and (\ref{eq.flags-ind}) is void,
so suppose $k < m$. We give the details for (\ref{eq.flags-ind});
the proof of (\ref{eq.flags-no-ind}) works analogously; see below.

We will prove
\begin{equation}\label{eq.flags-Q}
 \vv((P_{k,i} - P_{k,i+1})\cdot Q\cdot P_{k+2,k+2} \cdots P_{m,m}) \ge \lambda_{i+1} - \lambda_{m+1}
\end{equation}
for $Q = P_{k+1,j} - P_{k+1,j+1}$ ($j = k+1, \dots, m-1$) and for $Q = P_{k+1,m}$. The sum of all those $Q$ is equal to $P_{k+1,k+1}$, so
taking the sum of (\ref{eq.flags-Q}) for all those $Q$ then yields (\ref{eq.flags-ind}).

Case $Q = P_{k+1,m}$:
Since $\vv(P_{k,m} - P_{k+1,m}) \ge \lambda_{k+1} - \lambda_{m+1}\ge \lambda_{i+1} - \lambda_{m+1}$,
we can replace $Q$ by $P_{k,m}$ in (\ref{eq.flags-Q}). Now (\ref{eq.flags-Q}) follows, since
$(P_{k,i} - P_{k,i+1})P_{k,m} = P_{k,m} - P_{k,m} = 0$.

Case $Q = P_{k+1,j} - P_{k+1,j+1}$:
By induction, we have
\begin{equation}\label{eq.flags-Q-ind}
 \vv(Q \cdot P_{k+2,k+2} \cdots P_{m,m}) \ge \lambda_{j+1} - \lambda_{m+1}
.
\end{equation}
If $j \le i$, we are done, since $\lambda_{j+1} - \lambda_m \ge \lambda_{i+1} - \lambda_{m+1}$,
so suppose now $j > i$.
In that case, we have the following (``$\approx$'' explained below):
\begin{align*}
  &(P_{k,i} - P_{k,i+1})\cdot Q\cdot P_{k+2,k+2} \cdots P_{m,m}\\
  = \, &(P_{k,i} - P_{k,i+1})\cdot Q\cdot Q\cdot P_{k+2,k+2} \cdots P_{m,m}\\
  \approx
   \, &(P_{k,i} - P_{k,i+1})\cdot  (P_{k,j} - P_{k,j+1})\cdot Q\cdot P_{k+2,k+2} \cdots P_{m,m}
\end{align*}
Since $i \ne j$, we have $(P_{k,i} - P_{k,i+1})(P_{k,j} - P_{k,j+1}) = 0$, so to obtain
(\ref{eq.flags-Q}), it remains to verify that the difference between the two sides of ``$\approx$''
has valuation at least $\lambda_{i+1} - \lambda_{m+1}$.
This follows from (\ref{eq.flags-Q-ind}) and the following:
\begin{align*}
 \vv(Q - (P_{k,j} - P_{k,j+1}))
&\ge \min\{\vv(P_{k+1,j} - P_{k,j}), \vv(P_{k+1,j+1} - P_{k,j+1})\}\\
&\ge \min\{\lambda_{k+1} -\lambda_{j+1}, \lambda_{k+1} -\lambda_{j+2}\} \ge \lambda_{i+1} -\lambda_{j+1}.
\end{align*}

This finishes the proof of (\ref{eq.flags-ind}). The proof of
(\ref{eq.flags-no-ind}) is exactly the same: just replace $(P_{k,i} - P_{k,i+1})$
by $(1 - P_{k,k})$ everywhere in the proof and then plug in $k$ for the remaining $i$'s in the proof.
(Concerning the case $Q = P_{k+1,j} - P_{k+1,j+1}$, note that one then automatically has $j > i = k$.)
\end{proof}

From Proposition~\ref{prop.flag-lip}, one can easily deduce that the notion of Lipschitz stratifications is invariant under $\GL_n$. More precisely, we obtain the following.

\begin{cor}[$\GL_n(\valring)$-invariance]\label{cor.GLO-invar}
If $\mdl X = (X^i)_i$ is a valuative Lipschitz stratification of a definable set $X \subseteq \bmdl^n$ and $M \in \GL_n(\valring)$,
then $M(\mdl X) \coloneqq (M(X^i))_i$ is a valuative Lipschitz stratification
of $M(X)$.
\end{cor}

\begin{proof}
We use the characterization of valuative Lipschitz stratification from Proposition~\ref{prop.flag-lip} (2).
By Lemma~\ref{lem.vvM}, $M$ preserves valuations, so applying $M$ to a val-chain $(a^i)_i$ for $\mdl X$ yields a val-chain for $M(\mdl X)$. Moreover, if $V_{k,\ell}$ are vector spaces satisfying the conditions (\ref{eq.p-flags-inc}) -- (\ref{eq.p-flags-dist})
with $V_{k,k} = \bm T_{a^k}\mathring X^{e_k}$,
then the spaces
$M(V_{k,\ell})$ satisfy the same conditions with $M(V_{k,k}) = \bm T_{M(a^k)}M(\mathring X^{e_k})$.
\end{proof}

\begin{cor}\label{cor.GLR-invar}
If $\mdl X = (X^i)_i$ is a Lipschitz stratification of a definable set $X \subseteq \bmdl^n$ and $M \in \GL_n(\bmdl)$,
then $M(\mdl X) \coloneqq (M(X^i))_i$ is a Lipschitz stratification
of $M(X)$.
\end{cor}

\begin{proof}
We may assume that $\mdl X$ and $M$ are \LT-definable. (This works in the same way as in the proof of Theorem~\ref{main} from Theorem~\ref{thm.main-na}:
We first choose a language containing all the constants we need, and then we choose the models $\smdl \precneqq \bmdl$.)

By Proposition~\ref{prop.new-lip-na}, $\mdl X$ is a valuative Lipschitz stratification; by Corollary~\ref{cor.GLO-invar}, $M(\mdl X)$ is a
valuative Lipschitz stratification
(which is still \LT-definable, since $M$ is),
and finally, using Proposition~\ref{prop.new-lip-na} again,
we deduce that $M(\mdl X)$ is a Lipschitz stratification.
\end{proof}

\subsection{Overview of the main proof}
\label{sect.overview}

Here is an overview of the proof of Theorem~\ref{thm.main-na}, describing the main ideas in an informal way. Several technicalities are omitted.

We can easily stratify the given set $X \subseteq \bmdl^n$ in such a way that each stratum is the graph of a function.
More precisely, given a $d$-dimensional stratum $S$, after a suitable coordinate transformation, $S$ is the graph of some function
$\rho\colon \bar S \fun \bmdl^{n-d}$, where $\bar S \coloneqq \pr_{\le d}(S) \subseteq \bmdl^d$. 
Our final goal is to obtain bounds on valuative distances
$\Delta(V_1, V_2)$ (see Definition~\ref{defn.Delta}) between certain subspaces $V_i$ of tangent spaces. To be able to easily express those distances in terms of the functions $\rho$, we need that $\rho$ satisfies
\begin{equation}\label{eq.int.jac}
\vv(\Jac_{\bar a} \rho) \ge 0
\qquad\text{for every }\bar a \in \bar S.
\end{equation}
Indeed, for instance, under this assumption, we have, for $a_1, a_2 \in S$:
\[
\Delta(\bm T_{a_1} S, \bm T_{a_2} S) =
\vv(\Jac_{\bar a_1} \rho - \Jac_{\bar a_2} \rho),
\]
where $\bar a_i = \pr_{\le d}(a_i)$.

Most of the bounds of the form $\Delta(V_1, V_2)$ we aim for involve more than one stratum: Given a val-chain
$a^0, \dots, a^m$ with $a^\ell \in S^\ell$, we need to relate the tangent spaces
of all those strata $S^0, \dots, S^m$.
To be able to apply our above approach (of considering strata as graphs of functions and expressing distances of spaces in terms of Jacobians), we need to find a single coordinate transformation such that afterwards, each $S^\ell$ is the graph of a function $\rho^\ell$ satisfying (\ref{eq.int.jac}). (We call such a coordinate transformation an ``aligner'' of $S^0, \dots S^m$.) In Subsection~\ref{sect.brady} (Proposition~\ref{prop.brady-dec}), we obtain stratifications admitting aligners for any choice of $n+2$ strata. This is enough, since a val-chain consists of at most $n+2$ points.

To illustrate the remainder of the proof, we start by considering a plain val-chain consisting of only two points $a^0 \in S^0$, $a^1 \in S^1$. We suppose that we have already applied an aligner, so that $S^\ell$, $\ell = 1, 2$, is the graph of some function $\rho^\ell$ on $\bar S^\ell \coloneqq \pr_{\le e_\ell} S^\ell$, where $e_\ell \coloneqq \dim S^\ell$.

The next step in the proof consists in reducing the case of arbitrary plain val-chains (of length 2) to plain val-chains satisfying $\pr_{\le e_1}(a^0) = \pr_{\le e_1}(a^1)$; in the following, we assume this. In particular, this means that $a^0$ determines $a^1$ (assuming that $S^0$ and $S^1$ are fixed).

\begin{figure}
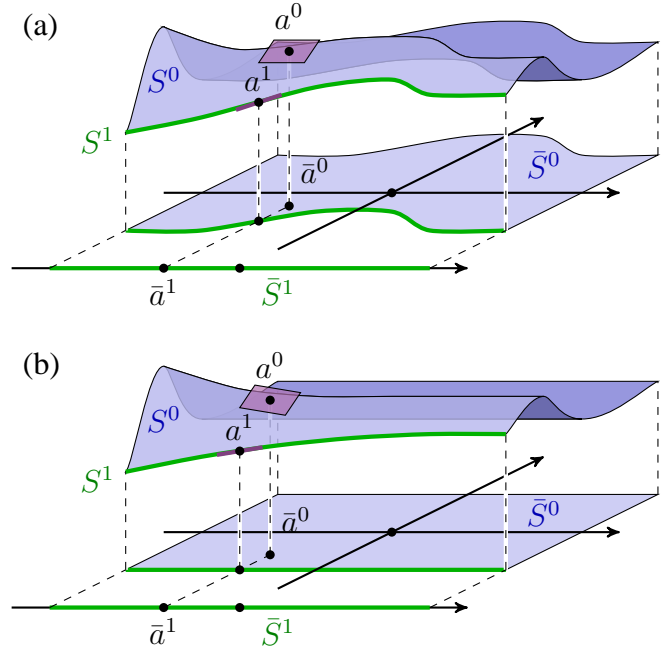

\begingroup
\def\wid{5}
\def\piece(#1,#2)(#3)(#4)(#5,#6){
  plot[smooth,xshift=#1cm,yshift=#2cm] coordinates {\cright}
  --
  (\wid+#1,#2+\yf) .. controls (\wid+#3+\yf) and (\wid+#4+\yyf) .. (\wid+#5,#6+\yyf)
  --
  plot[smooth,xshift=#5cm,yshift=#6cm] coordinates {\cleft}
  --
  (#5,#6+\yya) .. controls (#4+\yya) and (#3+\ya) .. (#1,#2+\ya)
  -- cycle
}
\def\sety(#1/#2/#3/#4/#5/#6){
\def\ya{#1}\def\yb{#2}\def\yc{#3}\def\yd{#4}\def\ye{#5}\def\yf{#6}
}
\def\setyy(#1/#2/#3/#4/#5/#6){
\def\yya{#1}\def\yyb{#2}\def\yyc{#3}\def\yyd{#4}\def\yye{#5}\def\yyf{#6}
}

\def\everything{
\def\cright{(0,\ya) (1,\yb) (2+2*\deform,\deform+\yc) (3+2*\deform,\deform+\yd) (4,\ye) (5,\yf)}
\def\cleft{(5,\yyf) (4,\yye) (3+2*\deform,\deform+\yyd) (2+2*\deform,\deform+\yyc) (1,\yyb) (0,\yya)}


\coordinate(a0pr) at (1.9+\deform,-1.3+0.5*\deform);
\coordinate(a1pr) at (1.5+\deform,-1.5+0.5*\deform);
\coordinate(a0) at (1.9+\deform,.75+0.5*\deform);
\coordinate(a1) at (1.5+\deform,.075+0.5*\deform);


\sety(0/0/0/0/0/0)
\setyy(0/0/0/0/0/0)
\draw
  [preaction={fill,S0set}]
  \piece(0,-1.5)(0,-1.5)(2,-.5)(2,-.5);

\draw[S1set] plot[smooth,yshift=-1.5cm] coordinates {\cright};


\begin{scope}[xshift=2.5cm,yshift=-1.5cm]
  \draw[axis arrow] (-.5,-.25) -- (3,1.5);
  \draw[axis arrow] (-2,.5) -- (4,.5);
  \fill(1,.5) circle (.06);
\end{scope}


\draw[dashed]
  (2,-.5) -- (2,1)
  (2+\wid,-.5) -- (2+\wid,1);

\draw[dash3d]
  (0,-1.5) -- (0,-.2)
  (0+\wid,-1.5) -- (0+\wid,.3);

\node[S0text] at (5.5,-.75+.5*\deform) {$\bar S^0$};

\begin{scope}[xshift=-1cm,yshift=-2cm]
  \draw[axis arrow] (-.5,0) -- (\wid+.5,0);
  \draw[S1set] (0,0) -- (\wid,0);
  \fill(2.5,0) circle (.06);
  
  \coordinate(a1prpr) at (1.5,0);
  
  \fill(a1prpr) circle (.06) node[below] {$\bar a^1$};
  
  \draw[dashed]
    (0,0) -- (1,.5)
    (\wid,0) -- (1+\wid,.5);

  \draw[dashed] (a1prpr) -- (a0pr);
  
  \node[below,S1text] at (3,0) {$\bar S^1$};
\end{scope}

\sety(0/0/0/0/0/0)
\setyy(0/0/0/0/0/0)
\draw
  [preaction={fill,color=white,opacity=.4}]
  [preaction={fill,color=white!70!black!60!blue,opacity=.7}]
  \piece(1,.5)(1.4,.5)(1.8,.9)(2,1);
\sety(0.4/0.1/0/0/0/0)
\setyy(0/0/0/0/0/0)
\draw
  [preaction={fill,color=white,opacity=.4}]
  [preaction={fill,color=white!40!black!70!blue,opacity=.7}]
  \piece(.5,.8)(.6,.8)(.7,.5)(1,.5);
\draw[dash3d]
  (a0pr) -- (a0)
  (a1pr) -- (a1);
\sety(-.2/0/.15/.25/0.3/.3)
\setyy(0.4/0.1/0/0/0/0)
\draw
  [preaction={fill,color=white,opacity=.4}]
  [preaction={fill,color=white!90!black!75!blue,opacity=.7}]
  \piece(0,0)(.1,.1)(.3,.8)(.5,.8);
\node[S0text] at (.5,.5) {$S^0$};

\draw[S1set] plot[smooth] coordinates {\cright};
\node[S1text,left] at (0,-.3) {$S^1$};

\fill(a0pr) circle (.06) node[anchor=south west] {$\strut\bar a^0$};
\fill(a1pr) circle (.06);


\draw[preaction={fill,color=white!90!black!50!violet,opacity=.7}]
 \tangentplane(a0)(-.4,-.1-.2*\deform)(.2,-.2+.2*\deform);
\draw[ultra thick,color=white!50!black!50!violet]
 \tangentline(a1)(-.3,-.05-.2*\deform);
\fill(a0) circle (.06) node[above] {$\strut a^0$};
\fill(a1) circle (.06) node[above] {$a^1$};

} 
\begin{ctikzpicture}
\begin{scope}
  \def\deform{.25}
  \everything
  \node[right] at (-1.5,1.2) {(a)};
\end{scope}
\begin{scope}[yshift=-4.5cm] 
  \def\deform{0}
  \everything
  \node[right] at (-1.5,1.2) {(b)};
\end{scope}
\end{ctikzpicture}
\endgroup
\caption{%
In this example, we consider a plain val-chain $a^0$, $a^1$ with dimensions $e_0 = 2$ and $e_1 = 1$. We assume that $\pr_{\le e_1}(a^0) = \pr_{\le e_1}(a^1)$.
(a) We need to find a subspace of the tangent plane $\bm T_{a^0}(S^0)$ which is close to the tangent line $T_{a^1}(S^1)$.
(b) To simplify this, we first deform the whole picture in such a way that the projection $\pr_{\le e_0}(S^1)$ becomes a straight line.
}
\label{fig.a0a1}
\end{figure}

To establish the conditions from Proposition~\ref{prop.flag-lip} concerning the val-chain $a^0, a^1$, we need to find a subspace $V \subseteq \bm T_{a^0} S^0$ that is sufficiently close to $\bm T_{a^1} S^1$; see Figure~\ref{fig.a0a1}~(a). We choose $V$ to be the subspace of $\bm T_{a^0} S^0$
satisfying $\pr_{\le e_0}(V) = \pr_{\le e_0}(\bm T_{a^1} S^1)$. (From (\ref{eq.int.jac}), one can deduce that this is a best possible approximation to $\bm T_{a^1} S^1$.)
The distance $\Delta(V, \bm T_{a^1} S^1)$ can directly be expressed in terms of Jacobians of the functions $\rho^0$ and $\rho^1$, but this becomes simpler if we first apply a ``rectilinearization'': a transformation which translates the coordinates $e_1 + 1, \dots, e_0$ in such a way that
$\pr_{\le e_0}(S^1)$ is sent to a subset of $\bmdl^{e_1} \times \{0\}^{e_0 - e_1}$ and which preserves all the other coordinates; see Figure~\ref{fig.a0a1}~(b).
After the rectilinearization has been applied, $V$ is determined by the first $e_1$ derivatives of $\rho^0$,
and we obtain
\[
\Delta(V, \bm T_{a^1} S^1)
= \vv((\Jac_{\bar a_0} \delta)\rest \bmdl^{e_1} \times \{0\}^{e_0-e_1})
= \min_{1 \le i \le e_1}\vv(
\partial_i\delta(\bar a_0)),
\]
where $\bar a^0 \coloneqq \pr_{\le e_0}(a^0)$ and $\delta \colon \bar S^0 \to \bmdl^{n - e_0}$ is the difference of $\rho^0$ and the last $n - e_0$ coordinates of $\rho^1$.

The desired bound on $\Delta(V, \bm T_{a^1} S^1)$ depends on the valuative distance of $a^1$ to a lower-dimensional stratum. We ensure that this bound holds by removing a lower-dimensional subset from $\bar S^1$. In terms of the function $\delta$ defined above, this means that we need to find a 
set $Z \subseteq \bmdl^{e_1}$ of dimension less than $e_1$ such that the first $e_1$ partial derivatives of $\delta$ at $x \in \bar S^0$ are bounded in terms of the distance
of $\pr_{\le e_1}(x)$ to $Z$. More precisely, the bound we end up needing is
\begin{equation}\label{eq.end-up}
\vv(\partial_i\delta(x)) \ge
\underbrace{\min\{\vv(\pr_{> e_1}(x)), \vv(\delta(x))\}}_{=\vv(a^0 - a^1)} - \valdist(\pr_{\le e_1}(x), Z)
\qquad\text{for }1 \le i \le e_1.
\end{equation}
The heart of the construction of valuative Lipschitz stratifications is Proposition~\ref{prop.sedate}, which provides such a lower-dimensional set $Z$ for arbitrary functions $\delta$.

For longer val-chains, the arguments are similar: Given a plain val-chain $a^0 \in S^0, \dots, a^m \in S^m$ (with $\dim S^\ell = e_\ell$), we rectilinearize with respect to some of the coordinates of $S^\ell$ for $\ell = 1, \dots, m$ and we obtain a function $\delta$ on (a certain subset of) $\bar S^0$ whose first $e_m$ derivatives need to be bounded by removing a lower-dimensional subset $Z$ from $\bar S^m$. Together with an inductive assumption that everything already works well for the sub-chain $a^0, \dots, a^{m-1}$, we obtain the subspaces $V_{k,\ell}$ needed by Proposition~\ref{prop.flag-lip}.

For augmented val-chains, the outline of the argument is the same; the biggest differences arise when the two first points $a^0,a^1$ lie in the same stratum $S^0$,
which, say, is the graph of $\rho^0$. In that case, instead of bounding first derivatives, we need to bound the second derivatives of $\rho^0$ to obtain a bound 
$\Delta(\bm T_{a^0}S^0, \bm T_{a^1}S^0)$.
Those bounds are obtained in essentially the same way as
(\ref{eq.end-up}), namely by applying Proposition~\ref{prop.sedate} to all first derivatives of $\rho$.

\medskip

We end this overview by mentioning an issue related to aligners (i.e., the coordinate transformation ensuring (\ref{eq.int.jac})).
Given a sequence $S^0, \dots, S^m$ of strata, the set $Z$ to be removed from $S^m$ according to the above procedure may depend on the chosen aligner.
When inductively assuming that this has already been done for
$S^0, \dots, S^{m-1}$, we need that it has been done using the same aligner as the one we use for $S^0, \dots, S^m$. However, an aligner for $S^0, \dots, S^{m-1}$ might not be suitable for $S^m$.
The solution is that Proposition~\ref{prop.brady-dec} states that all aligners can be found in a \emph{finite} set $\CC_n$ of coordinate transformations (depending only on the ambient dimension $n$).
By applying the above procedure to every possible aligner of $S^0, \dots, S^{m-1}$ in $\CC_n$, we in particular ensure that we included aligners working for $S^0, \dots, S^m$.

\section{Ingredients to the main proof}
\label{sect.ingred}

The entire remainder of the article is devoted to the proof of Theorem~\ref{thm.main-na}.
We continue to use the notation introduced in
Subsections~\ref{sect.notn.smdl}, \ref{sect.enlarge} and \ref{sect.notn.bmdl}, though we make a slight change concerning the language: to avoid having to mention the parameters $A$ from Theorem~\ref{thm.main-na} everywhere, we now allow $\LT$ to contain additional constants from $\bmdl$. Thus the general assumptions for the remainder of the paper are the following.

\begin{ass}\label{ass.sat2}
For the remainder of the paper,
we assume that $\bmdl$ is a real closed field which is power-bounded and \omin-minimal as an $\LT^0$-structure and
$T$-convex as an $\LTv^0$-structure. Moreover, we set
$\LT := \LT^0(A)$ and $\LTv := \LTv^0(A)$ for some finite set of parameters $A \sub \bmdl$, and we assume (without loss) that $\bmdl$ is sufficiently saturated.
\end{ass}

Note that there is a hidden quantifier here: We will prove everything for \emph{every} finite set $A$ of parameters. This in particular means that we can use previously proved results for different $A$.

\subsection{Alignable Bradycell Decompositions}
\label{sect.brady}

The first step in the construction of a valuative
Lipschitz stratification of a set $X$ consists in partitioning $X$ into  pieces that can be ``aligned'': After a suitable transformation of the coordinate system, they are graphs of functions whose derivatives have non-negative valuation.

\begin{defn}[Aligned sets]\label{defn.aligned}
Let $S$ be an \LT-definable subset of $\bmdl^n$. We say that $S$ is \emph{aligned} if, for $d \coloneqq \dim S$, the set $\bar S := \pr_{\le d}(S)$ is open in $\bmdl^d$ and $S$ is the graph of an \LT-definable $C^1$  function $f\colon \bar S \fun \bmdl^{n-d}$ satisfying
\begin{equation}\label{eq.brady}
\vv(\Jac_a f) \ge 0 \qquad \text{for all } a \in \bar S.
\end{equation}
The open set $\bar S$ is referred to as the \emph{base} of $S$.
We say that 
$\aligner \in \GL_n(\valring)$ is an \emph{aligner} of
an \LT-definable set $S \subseteq \bmdl^n$ if $\kappa(S)$ is aligned.
\end{defn}

\begin{rem}\label{rem.clopen-infty}
If such an aligned set $S$ is a lowest dimensional stratum of a stratification of a closed definable set $X \sub \bmdl^n$, then $\bar S$ is both, open and closed and hence $\bar S = \bmdl^{\dim S}$. This fits together with the fact that Lipschitz stratifications almost never have $X^0 = \emptyset$ and it will also fit together with the convention that in val-chains, we set $\lambda_{m+1} = -\infty$ if $X^{e_m - 1} = \emptyset$.
\end{rem}

\begin{rem}
These aligned sets are somewhat related to
the $L$-regular cells of
\cite[\S 1]{kurd:paru:quasi},
to the regular $M$-cells of
\cite[\S 1]{paw:lip}, and
to the $\Lambda^m$-regular cells of
\cite[Definition~1.2]{fisch:reg:strat},
though the latter are more sophisticated and
control more derivatives, and all of these notions impose additional conditions on the base.
\end{rem}

It will not be enough to partition our given set $X$ into sets which can be aligned using some arbitrary $\aligner \in \GL_n(\valring)$; we will also need some good control
of these $\aligner$:
\begin{enumerate}
\item We need to find a \emph{finite} set $\CC_n \subseteq \GL_n(\valring)$, depending only on $n$, such that all aligners $\kappa$ can be taken from $\CC_n$. The precise set $\CC_n$ does not matter, so we postpone choosing it to Definition~\ref{defn.Cn}.
\item We need all $\kappa$ to be \LT-definable without additional parameters. To ensure this, we will choose
$\CC_n \subseteq \GL_n(\Q)$.
\item We need that a single $\kappa$ works for several (given) pieces of the partition at once. More precisely, any $n+2$ pieces should have a common aligner in $\CC_n$.
(This number $n+2$ is what we need for the proof of existence of Lipschitz stratifications. The proofs in Subsection~\ref{sect.brady} would work equally well for any other fixed number.)
\end{enumerate}

It is problematic that Item (3) above is not a condition on individual pieces, but on the partition as a whole: This makes it unclear whether, given a partition $\mdl S = (S_i)_i$ satisfying (3), one can refine parts of $\mdl S$ in a way that (3) is preserved without modifying the remainder of $\mdl S$. To solve this problem,
we will introduce the notion of ``bradycells'' (Definition~\ref{defn.brady}). Bradycells will have the property that $n+2$ of them always have a common aligner. Using that notion, we can state the main result of this subsection, which provides the desired partitions. Since the precise notion of bradycells is irrelevant for the remainder of the article, we postpone it.

\begin{defn}\label{defn.brady-dec}
A \emph{bradycell decomposition} is a partition of $\bmdl^n$ into bradycells (see Definition~\ref{defn.brady}).
\end{defn}

\begin{prop}[Bradycell decompositions]\label{prop.brady-dec}
  \begin{enumerate}
  \item
  Every finite partition of $\bmdl^n$ into \LT-definable sets can be refined to a bradycell decomposition.
  \item
  For any set of at most $n+2$ bradycells
  $S_1, \dots, S_k \subseteq \bmdl^n$, there exists a common aligner $\aligner \in \CC_n$.
  \end{enumerate}
\end{prop}

\begin{rem}
This entire subsection would become much simpler if all $S_1, \dots, S_k$ in Proposition~\ref{prop.brady-dec}~(2)
could be assumed to have different dimension; in particular, one could then choose $\CC_n$ to consist only of the coordinate
permutations. In applications of the proposition, this will almost be the case: at most two of the bradycells will have the same dimension. It would probably be possible to also get rid of this (an approach like this has been used in \cite{i.whit}), but this would require considerably more work.
\end{rem}

For the remainder of this subsection, we fix the following notation.

\begin{notn}
For $d \le n$, we write $\gra_{n, d}(\bmdl)$ for the Grassmannian variety, i.e., for the space of $d$-dimensional sub-vector spaces of $\bmdl^n$.
\end{notn}

\begin{defn}
For $d \le n$, let $\grao_{n, d}(\bmdl) \subseteq \gra_{n, d}(\bmdl)$ be the open subset of those $V \subseteq \bmdl^n$
such that $\pr_{\le d}(V) = \bmdl^d$, i.e., which project surjectively onto the first $d$ coordinates.
Such a $V \in \grao_{n, d}(\bmdl)$ can be considered as the graph of a linear map $M_V\colon \bmdl^d \fun \bmdl^{n-d}$;
we set $J(V) \coloneqq \norm{M_V}$ (the operator norm of the matrix); for $V \in \gra_{n, d}(\bmdl) \setminus \grao_{n, d}(\bmdl)$, we set $J(V) \coloneqq \infty$.
\end{defn}

The following lemma is the main tool to find the finitely many transformations $\kappa \in \GL_n(\Q)$. It is a purely geometrical-combinatorial result, closely related to \cite[Lemma~1.8]{kurd:paru:quasi}. Even though formulated in $\bmdl$, it is just a statement about $\R$ (as will become visible in the proof). Note also that the $\kappa$ provided by the lemma are even elements of $\Or_n(\Q)$ (and not just in $\GL_n(\Q)$).

\begin{lem}\label{lem.gra-cover}
Fix arbitrary natural numbers $n$ and $\ell$. Then we can find, for each $d \le n$, a finite open covering of $\gra_{n, d}(\bmdl)$
by \LT-definable sets $\Theta_{\nu}^{d}$
such that for any choice of $\ell$ many of these sets $\Theta_{\nu_1}^{d_1}$, \dots, $\Theta_{\nu_\ell}^{d_\ell}$,
there exists an orthogonal transformation $\aligner \in \Or_n(\Q)$ such that for every $i \le \ell$, we have
$\vv(J(\kappa\Theta_{\nu_i}^{d_i})) \ge 0$.
\end{lem}

(Here, $\vv(\infty) \coloneqq -\infty$ and $\vv(J(\kappa\Theta_{\nu_i}^{d_i})) \ge 0$ is a short hand
notation for: $\vv(J(\kappa V)) \ge 0$ for every $V \in \Theta_{\nu_i}^{d_i}$.)

\begin{proof}[Proof of Lemma~\ref{lem.gra-cover}]
We will prove the stronger claim that the sets $\Theta_{\nu}^{d}$ can be taken definable in the pure ring language. In that case, to get $\vv(J(\kappa\Theta_{\nu_i}^{d_i})) \ge 0$, it suffices to prove that the map $V \efun J(V)$ is bounded on $\kappa\Theta_{\nu_i}^{d_i}$
(by Remark~\ref{rem.nsa0}). This boundedness is also a statement in the ring language, so we may as well assume $\bmdl = \R$.
In particular, any closed subset of $\gra_{n, d_i}(\R)$ is compact, so we can obtain boundedness of $V \efun J(V)$ by proving
\[
 \cl(\aligner\Theta_{\nu_i}^{d_i}) \subseteq \grao_{n, d_i}(\R).
\]

Given a subset $\Xi \subseteq \gra_{n, d}(\R)$ (for any $d \le n$), we write $\forb(\Xi)$ for the set of orthogonal transformations $\kappa$ ``forbidden by a space in $\cl(\Xi)$'', i.e.:
\[
\forb(\Xi) \coloneqq \{\aligner \in \Or_n(\R) : \cl(\aligner \Xi) \not\subseteq \grao_{n, d}(\R)\}.
\]
Intuitively, we just need to choose the sets $\Theta_{\nu}^{d}$ so small that no $\ell$ of the
sets $\forb_{\nu}^d \coloneqq \forb(\Theta_{\nu}^{d})$ cover all of $\Or_n(\Q)$.
To make this argument precise, let $\mu$ be the Haar measure on the compact group $\Or_n(\R)$, normalized such that
$\mu(\Or_n(\R)) = 1$.
It is enough to ensure that $\mu(\forb_{\nu}^d) < 1/\ell$ for each $\nu$ and $d$.
(Then $\Or_n(\R) \setminus \bigcup_{i=1}^\ell \forb_{\nu_i}^{d_i}$
is non-empty and open, and hence contains a $\kappa \in \Or_n(\Q)$, as desired.)

To find finitely many sets $\Theta_{\nu}^d$ with that property
covering $\gra_{n,d}(\R)$, we fix any definable metric on $\gra_{n, d}(\R)$
inducing the usual topology.
Moreover, we fix any element $V_0 \in \gra_{n, d}(\Q)$. Since the set $\forb(\{V_0\}) \subseteq \Or_n(\R)$ is a compact subset of lower dimension,
we can find an open ball $\Xi \subseteq \gra_{n, d}(\R)$ around $V_0$ such that $\mu(\forb(\Xi)) < 1/\ell$. (First choose any open set $\UU \supseteq \forb(\{V_0\})$ with $\mu(\UU) < 1/\ell$, and then, using compactness of
$\forb(\{V_0\})$, choose the radius of $\Xi$ small enough to ensure $\forb(\Xi) \subseteq \UU$.) We may moreover assume that $\Xi$ has rational radius.

Now choose finitely many $\aligner_\nu \in \Or_n(\Q)$ such that
the sets $\Theta_{\nu}^{d} \coloneqq \aligner_\nu(\Xi)$ cover
$\gra_{n, d}(\R)$. Then indeed, $\mu(\forb_\nu^d)
= \mu(\forb(\aligner_\nu(\Xi))) < 1/\ell$.
\end{proof}

Using Lemma~\ref{lem.gra-cover}, we can now choose our finite set $\CC_n \subseteq \GL_n(\Q)$ and introduce
the notion of bradycells.

\begin{defn}[The set $\CC_n$]\label{defn.Cn}
For the remainder of this subsection, fix subsets $\Theta_{\nu}^{d} \subseteq \gra_{n, d}(\R)$ as provided by Lemma~\ref{lem.gra-cover}
using $\ell = n + 2$. Moreover, let $\CC_n \subseteq \GL_n(\Q)$
be a finite subset containing, for each choice of $n+2$ many sets $\Theta_{\nu_1}^{d_1}, \dots, \Theta_{\nu_{n+2}}^{d_{n+2}}$, an element $\aligner$ satisfying
$\vv(J(\kappa\Theta_{\nu_i}^{d_i})) \ge 0$ ($i = 1, \dots, n+2$).
(For any $n+2$ of the sets, the existence of such $\aligner \in  \GL_n(\Q)$ is asserted by the lemma, and there are only finitely many choices of $n+2$ sets.)
\end{defn}

\begin{defn}[Bradycells]\label{defn.brady}
A \emph{bradycell} is an \LT-definable set $S \subseteq \bmdl^n$ such that for (at least) one of the sets $\Theta_{\nu}^{d}$
chosen in Definition~\ref{defn.Cn} (where $d = \dim S$),
we have the following:
\begin{enumerate}
\item For every $x \in S$, the tangent space $\bm T_x S$ is an element of $\Theta_{\nu}^{d}$.
\item For every $\aligner \in \CC_n$ satisfying $\vv(J(\kappa\Theta_{\nu}^{d})) \ge 0$, 
$\aligner(S)$ is aligned.
\end{enumerate}
\end{defn}
The content of Condition (2) is just that the projection $\pr_{\le d}(\kappa(S))$ is open and that $\kappa(S)$ is the graph of a function on that projection; the bound on the derivatives of the function is automatic by $\vv(J(\kappa\Theta_{\nu}^{d})) \ge 0$ and Condition (1).

Now that $\CC_n$ and bradycells are defined, we can finally prove the main result of this subsection.

\begin{proof}[Proof of Proposition~\ref{prop.brady-dec}]
(1) We repeatedly refine the partition,
ensuring that each piece of dimension $d$ becomes a bradycell, proceeding
from $d = n$ downwards to $d = 0$. Thus fix $d \le n$, and fix any $d$-dimensional piece $S$.
It suffices to check that we can subdivide $S$ into (finitely many) $d$-dimensional bradycells and an arbitrary lower-dimensional set.
	
After a first partitioning, we may assume that $S$ is a definable $C^1$ manifold and satisfies
Condition~(1) from Definition~\ref{defn.brady} for some set $\Theta_{\nu}^{d}$.
To obtain Condition~(2), we further partition $S$ for each of those $\aligner \in \CC_n$
for which $\vv(J(\kappa\Theta_{\nu}^{d})) \ge 0$:
By a first partition, we ensure that $\kappa(S)$ is the graph of a function $f \colon \pr_{\le d}(\kappa(S)) \to \bmdl^{n-d}$.
Then we remove a lower-dimensional set to ensure that $\pr_{\le d}(\kappa(S))$ is open and that $f$ is $C^1$.

\medskip

(2)
Consider bradycells $S_1, \dots, S_{k}$ for some $k \le n+2$; for each $i \le k$,
let $\Theta_{\nu_i}^{d_i}$ be a corresponding set provided by Definition~\ref{defn.brady}.
By our choice of $\CC_n$ (Definition~\ref{defn.Cn}), there exists a $\aligner \in \CC_n$ such that for each $i$,
we have $\vv(J(\kappa\Theta_{\nu_i}^{d_i})) \ge 0$.
By Definition~\ref{defn.brady} (2), $\aligner(S_i)$ is aligned.
\end{proof}

We end this subsection by proving a useful property of aligned sets.

\begin{lem}\label{lem.frontier}
Let $S \sub \bmdl^n$ be a $d$-dimensional aligned set,
and suppose that $B \subseteq \bmdl^n$ is a valuative ball (open or closed) with
$B \cap S \ne \emptyset$ but $B \cap \partial S = \emptyset$.
Then $\bar B \coloneqq \pr_{\le d}(B)$ is a subset of the base $\bar S = \pr_{\le d}(S)$ of $S$.
\end{lem}

\begin{proof}
Suppose that $\bar B \not\subseteq \bar S$. Choose $a \in B \cap S$, set $\bar a \coloneqq \pr_{\le d}(a) \in \bar B \cap \bar S$ and choose $\bar b \in \bar B \setminus \bar S$. 
Let $L \coloneqq \{(1-t) \bar a + t\bar b \mid 0 < t < 1\}$ be the open line segment connecting $\bar a$ and $\bar b$.
We may assume $L \subseteq \bar S$; otherwise, replace
$\bar b$ by the point of $L \cap \partial \bar S$ which is closest to $\bar a$. (Such a point exists by \omin-minimality, and using that $\partial \bar S$ is \LT-definable.)

Let $f$ be the function whose graph is $S$, and consider the function $g\colon [0,1) \to S$ sending $t$ to $f((1-t) \bar a + t\bar b)$. Using $\vv(\Jac f) \ge 0$, we obtain $\vv(g'(t)) \ge \vv(\bar b - \bar a)$, so using the Mean Value Theorem, we deduce, for any $t_1, t_2 \in [0,1)$:
\[
\vv(g(t_2) - g(t_1)) \ge \vv(t_2 - t_1) + \vv(\bar b - \bar a)
\geq \rad(B),
\]
where the last inequality is strict if $B$ is an open ball.
This implies that $b' \coloneqq \lim_{t \to 1} g(t)$ exists and that $b := (\bar b, b')$ satisfies
$\vv(b - a) \ge \vv(\bar b - \bar a)$. In particular, $b \in B \cap \partial S$, contradicting the assumption that this intersection is empty.
\end{proof}

\begin{rem}\label{rem.int-val}
Given an \LT-definable $C^1$ function $f\colon X \fun \bmdl^{n-d}$ on an \LT-definable set $X \subseteq \bmdl^d$,
a similar kind of Mean Value Theorem argument on a line segment allows us to bound $\vv(f(a_1) - f(a_2))$ by $\vv(a_1 - a_2) + \vv(\Jac f)$ under suitable assumptions:
If $a_1$ and $a_2$ both lie in a valuative ball $B$ that is entirely contained in $X$, and $\vv(\Jac_a f) \ge \lambda$ for all $a \in B$, then
\[
\vv(f(a_1) - f(a_2)) \ge \vv(a_1 - a_2) + \lambda.
\]
In particular, for $S$ and $B$ as in Lemma~\ref{lem.frontier},
the entire preimage $\pr_{\le d}^{-1}(\bar B) \cap S$ is contained in $B$.
\end{rem}

\subsection{Bounding derivatives using power-boundedness}

A key ingredient to our proof of the existence of Lipschitz stratifications is the following proposition, which has been proved in \cite{Yin.tcon}. This is the only (but crucial) place in the present paper where power-boundedness is used.

\begin{prop}[{\cite[Corollary~2.17]{Yin.tcon}}]\label{prop.jac}
Suppose that $f\colon \bmdl^n \fun \bmdl$ is an \LT-definable function.
Then there exists a finite \LT-definable partition of $\bmdl^n$ into sets $Y_\nu$ such that if $B$ is an open valuative ball entirely contained in one of the sets $Y_\nu$,
then either $f(B) = \{0\}$ or $f(B)$ is an open valuative ball not containing $0$.
\end{prop}

Note that if $f(B)$ is an open valuative ball not containing $0$, then for any $y_1, y_2 \in B$, we have $\vv(f(y_1)) = \vv(f(y_2))$,
and even $\vv(f(y_1) - f(y_2)) > \vv(f(y_1))$.

Here, we have rewritten Proposition~\ref{prop.jac} in the language of the present paper; the map $\rv$ appearing in \cite{Yin.tcon} is defined in such a way that $\rv(a) = \rv(a')$ iff either $a = a' = 0$ or $\vv(a - a') > \vv(a)$ for $a, a' \in \bmdl$ (and a valuative polydisc is just a product of valuative balls of possibly different radii).
Note that the language used in \cite{Yin.tcon}
is, up to interdefinability, the same as ours;
see \cite[Definition~1.2 and Convention~1.11]{Yin.tcon}.

Instead of using Proposition~\ref{prop.jac} directly, we will use the following corollary:

\begin{cor}\label{cor.jac}
Suppose that $f\colon \bmdl^n \fun \bmdl$ is an \LT-definable function. Then there exists an
\LT-definable set $Z \subseteq \bmdl^n$ of dimension less than $n$ such that for every $y \in \bmdl^n \setminus Z$,
$\partial_i f(y)$ exists and we have
\begin{equation}\label{eq.jac}
\vv(\partial_i f(y)) \ge \vv(f(y)) - \valdist(y, Z)
\end{equation}
for $i = 1, \dots, n$.
\end{cor}

\begin{proof}
Apply Proposition~\ref{prop.jac} to $f, \partial_1 f, \dots,\partial_n f$,
where the partial derivatives are extended by $0$ to those points of $\bmdl^n$ where they do not exist; then set
$Z \coloneqq Z_0 \cup \bigcup_{i,\nu} \partial Y_{i,\nu}$, where
$(Y_{i,\nu})_\nu$ is the partition obtained for the $i$th of the above functions ($i = 1, \dots, n+1$) and $Z_0$ is the set of points where $f$ is not differentiable; we claim that this set $Z$ works.

Fix a $y \in \bmdl^n \setminus Z$ and set $\zeta \coloneqq \valdist(y, Z)$ and
$B \coloneqq B_{>\zeta}(y)$. Then for each $i$, there exists a $\nu$ such that $B \subseteq Y_{i,\nu}$ for some $\nu$;
in particular, $\vv(f(B))$ and $\vv(\partial_i f(B))$ are
singletons.

To prove (\ref{eq.jac}), we use an Mean Value Theorem argument similar to the one in Remark~\ref{rem.int-val}, but in the opposite direction:
Suppose for contradiction that $y$ is a witness to the failure of (\ref{eq.jac}), i.e.,
$\vv(f(B)) - \vv(\partial_i f(B)) > \zeta$ for some $i$. We
choose $y_1, y_2 \in B$ differing only in the $i$-th coordinate
with $\vv(y_1 - y_2) = \vv(f(B)) - \vv(\partial_i f(B))$.
The Mean Value Theorem yields a $y_3 \in B$ such that
\[
f(y_1) - f(y_2) = (y_1 - y_2)\cdot \partial_i f(y_3).
\]
This leads to a contradiction: On the one hand, we have
$\vv(f(y_1) - f(y_2)) > \vv(f(B))$ (by our application of Proposition~\ref{prop.jac} to $f$);
on the other hand,
\[
\vv((y_1 - y_2)\cdot \partial_i f(y_3))
=  \vv(f(B)) - \vv(\partial_i f(B)) + \vv(\partial_i f(B)) = \vv(f(B)).\qedhere
\]
\end{proof}

\begin{rem}\label{rem.jac}
Using Remark~\ref{rem.nsa0},
Corollary~\ref{cor.jac} may be reformulated without making reference to the valuation. Since Remark~\ref{rem.nsa0} only applies to functions defined without parameters outside of $\smdl$ (but we have made the change at the beginning of this section so that $\LT$ now might contain such parameters), one first needs to formulate the corollary for families of functions. In this way, one obtains that Corollary~\ref{cor.jac} is equivalent to the following statement: For any
\LT-definable family of functions $f_q:\bmdl^n \to \bmdl$ (where $q$ runs over some \LT-definable set $Q$), there exists a
constant $c \in \bmdl$ (not depending on $q$) and an \LT-definable family of sets $Z_q \subseteq \bmdl^n$ of dimension less than $n$ such that
\begin{equation}\label{eq.fisch}
|\partial_i f_q(y))| \le \frac{c|f_q(y)|}{\dist(y, Z_q)}
\qquad \text{for all $i \le n$, all $q \in Q$ and all
$y \in \bmdl^n \setminus Z_q$}.
\end{equation}
Note that this bears some similarities to the $\Lambda_L^1$-regular functions in \cite[Definition~1.1]{fisch:reg:strat} (though (\ref{eq.fisch}) is false in, e.g., structures with exponential function).
One has the feeling that there should be a more direct proof
of (\ref{eq.fisch}), avoiding the machinery of \T-convexity. For $n = 1$ and when $Q$ is a singleton, it is not too difficult to deduce it from power-boundedness. However, we do not know how to prove the general case more directly.
\end{rem}

Here is another lemma, which does not really have anything to do with the previous results of this subsection, but which will be useful in conjunction with them.

\begin{lem}\label{lem.choose-max}
  Suppose that $X \subseteq \bmdl^n$  is a non-empty
  \LT-definable set and $f\colon X \fun \bmdl$ is an \LT-definable function such that $|f|$ is bounded
  (by an element of $\bmdl$). Then there exists an \LT-definable element $x_0 \in X$ such that
  $\vv(f(x_0)) = \min \{\vv(f(x)) : x \in X\}$. In particular, that minimum exists.
\end{lem}
\begin{proof}
  Set $s \coloneqq \sup_{x}|f(x)|$.
  The set $X' \coloneqq \{x \in X : |f(x)| \ge \frac12 s\}$
  is \LT-definable and non-empty, and every $x \in X'$ satisfies $\vv(x) = \vv(s)$.
  Using definable choice (in the \omin-minimal language $\LT$),
  we find an \LT-definable $x_0 \in X'$.
\end{proof}

\subsection{Sedating functions}
\label{sect.sedate}

To construct Lipschitz stratifications, we will need precise bounds on the valuations of the first derivatives of certain functions. The goal of this subsection it to prove the key tool for this: Proposition~\ref{prop.sedate}, which will allow us to obtain the desired bounds for any definable function after refining our stratification.
We will also need bounds on second derivatives; those will be obtained in Corollary~\ref{cor.second-sedate}, by applying
Proposition~\ref{prop.sedate} to the first derivatives. Functions satisfying the desired bounds will be called ``sedated''.

Before going into the details, here is an informal explanation.
Given an \LT-definable function $f\colon X \to \bmdl$ on an \LT-definable set $X \subseteq \bmdl^n$, we can remove a lower-dimensional subset from $X$ using Corollary~\ref{cor.jac} to obtain a bound on $\vv(\nabla f(x))$ which is good whenever $x$ is not too close to the boundary of $X$:
\begin{equation}\label{eq.int.sed}
\vv(\nabla f(x)) \ge \vv(f(x)) - \valdist(x, \bmdl^n \setminus X).
\end{equation}
As a bound on $\nabla f(x)$, this is in some sense optimal, but it is often possible to get better bounds on individual partial derivatives: Very roughly, even near the boundary of $X$, one should be able to obtain good bounds on the partial derivatives in those directions which do not point towards the boundary; see Figure~\ref{fig.border} (a).
To construct stratifications, we will need such better bounds.

It is not so clear how to make this precise in general. Instead, the result in this subsection will provide the better bounds only in the rather specific situation we are in after 
the rectilinearization explained in Subsection~\ref{sect.overview}:
We only need a bound on the partial derivatives parallel to $W \coloneqq \bmdl^{n'} \times \{0\}^{n - n'}$, and that bound should not be affected by $\dist(x, W)$ being small even if $W$ contains a boundary segment of $X$.
(Such a bound makes most sense if $X$ indeed has a boundary segment in $W$; however, we will also prove and use the result when it doesn't.)
The precise statement is that after removing a lower-dimensional subset from $X$, we obtain the estimate
\begin{equation}\label{eq.int.sed2}
\vv(\partial_i f(x)) \ge \vv(f(x)) - \valdist(\pr_{\le n'}(x), \bmdl^{n'} \setminus \pr_{\le n'}(X))
\qquad\text{for } i =1, \dots,  n'
\end{equation}
for points $x \in X$ satisfying
\begin{equation}\label{eq.int.sigma}
\valdist(x, \bmdl^n \setminus X) \le \vv(\pr_{> n'}(x)).
\end{equation}
Condition~(\ref{eq.int.sigma}) ensures that $x$ it not too close to a border of $X$ different from $W$; indeed,
(\ref{eq.int.sed2}) cannot be expected for points close to a ``diagonal border'' like $x_2$ in Figure~\ref{fig.border}.

\begin{figure}
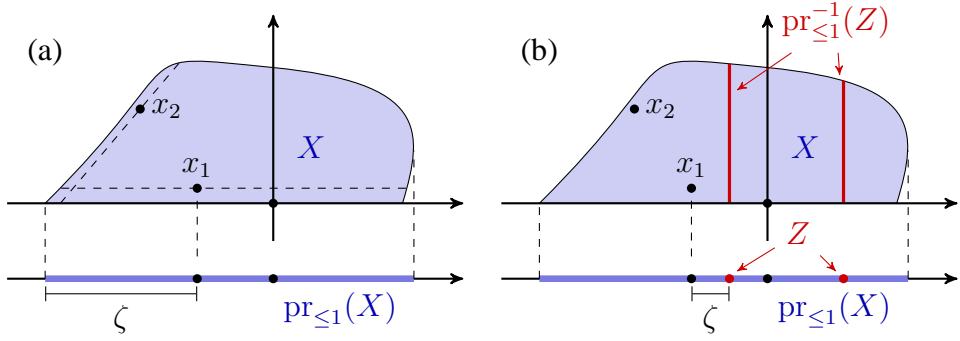

\begingroup
\def\Xpath{(0,0) --
              (4.7,0) .. controls (5,1) and (5,1.6) ..
              (3,1.8) .. controls (1,2) and (2,2) ..
              (0,0) -- cycle}
\def\everything#1#2{
  \draw[S0set] \Xpath;
  \node[S0text] at (3.5,.75) {$X$};
  \fill (2,.2) circle (.06) node[above] (x1) {$x_1$};
  \fill (1.25,1.25) circle (.06) node[right] {$x_2$};
  \begin{scope}
    \clip \Xpath;
    #1
  \end{scope}
      
  \draw[axis arrow] (-.5,0) -- (5.5,0);
  \draw[axis arrow] (3,-.5) -- (3,2.5);  
  \fill (3,0) circle (.06);
  
  \begin{scope}[yshift=-10mm]
    \draw[axis arrow] (-.5,0) -- (5.5,0);
    \draw[S0set1] (0,0) -- node[S0text,swap,pos=.8] {$\pr_{\le 1}(X)$} (4.85,0);
    \fill (3,0) circle (.06);
  
    \fill (2,0) circle (.06) node[above] (x1pr) {};
    
    \draw[dashed] (0,0) -- (0,1);
    \draw[dashed] (4.85,0) -- (4.85,1.7);
    \draw[dashed] (x1pr) -- (x1);
    
    #2

  \end{scope}
}
\begin{ctikzpicture}
\begin{scope}[xshift=0cm]
  \node at (0,2) {(a)};
  \everything{
    \draw[dashed] (0,.2) -- (5,.2); 
    \draw[dashed] (.2,0) -- (2.3,2.5); 
  }{
    \draw[|-|] (0,-.2) -- node[swap] {$\zeta$} (2,-.2);
  }
\end{scope}
\begin{scope}[xshift=6.5cm]
  \node at (0,2) {(b)};
  \everything{
    \draw[S2set,very thick] (2.5,0) -- (2.5,2)
                 (4,0) -- (4,2);
  }{
   \fill[S2set] (2.5,0) circle (.06);
   \fill[S2set] (4,0) circle (.06);
   \draw[|-|] (2,-.2) -- node[swap] {$\zeta$} (2.5,-.2);
   \node[S2text] (Z) at (3.4,.6) {$Z$};
   \draw[arrow,S2set] (Z) -- (3.9,.1);
   \draw[arrow,S2set] (Z) -- (2.6,.1);
  }
  \node[S2text] (lZ) at (3.9,2.4) {$\pr_{\le 1}^{-1}(Z)$};
  \draw[arrow,S2set] (lZ) -- (3.95,1.7);
  \draw[arrow,S2set] (lZ) -- (2.6,1.5);
\end{scope}
\end{ctikzpicture}
\endgroup
\caption{(a)
At $x_1$ and $x_2$, one can expect good bounds on the partial derivatives of $f$ in the dashed directions, but not in the directions perpendicular to that.
Proposition~\ref{prop.sedate} provides good bounds on horizontal derivatives at points close to the $x$-axis: the bound on $\partial_1 f(x_1)$ is computed using the distance $\zeta$ in the projection $\pr_{\le 1}(X)$.
Such a bound cannot be expected for $\partial_1(x_2)$, since $x_2$ is close to a border of $X$ different from the $x$-axis.
(b)
To obtain the bounds, it might be necessary to remove a lower-dimensional subset $Z$ from $\pr_{\le 1}(X)$. This
in effect weakens the condition on $\partial_1(x_1)$, since $\zeta$ becomes smaller.
}
\label{fig.border}
\end{figure}

Since it is $\pr_{\le n'}(X)$ which appears in (\ref{eq.int.sed2}) and not $X$ itself, the only lower-dimensional sets it makes sense to remove from $X$ are sets of the form
$\pr_{\le n'}^{-1}(Z)$ for some $Z \subseteq \pr_{\le n'}(X)$ (see Figure~\ref{fig.border} (b)). This is how Proposition~\ref{prop.sedate} is stated, and it is this set $Z$ which will be used in the strategy outlined in Subsection~\ref{sect.overview} to shrink the $n'$-dimensional stratum.

The bound (\ref{eq.int.sed2}) is the one we will need to treat those augmented val-chains whose first two points $a^0$, $a^1$ lie in two different strata; functions satisfying this bound will be called \prep{a}-sedated. Proposition~\ref{prop.sedate} also provides two variants of this, which are needed for other kinds of val-chains:
to treat plain val-chains, we will need
\prep{b}-sedated functions, which satisfy a bound like (\ref{eq.end-up}), and to treat
augmented val-chains whose first two points lie in the same stratum,
we will need functions whose derivatives are \prep{c}-sedated (see below).

Everything described so far is what we need for short val-chains.
For longer val-chains, say, living in strata of dimensions
$e_1 > \dots > e_m$, we still need to bound the partial derivatives
$\partial_1 f(x),\dots, \partial_{e_m} f(x)$ of a function $f$ with domain $X \subseteq \bmdl^{e_1}$, but all the intermediate dimensions $e_\ell$ also play a role, namely for the conditions specifying to which boundaries of $X$ the point $x$ is allowed to be close. To make this precise, we start by fixing some notation.
In the whole subsection, we assume the following.

\begin{ass}\label{ass.meX}
Let the following be given:
\begin{itemize}
	\item
	an integer $m \ge 1$;
	\item
	integers $0 < e_m < \dots < e_1$;
	\item
	an open \LT-definable set $X \subseteq \bmdl^{e_1}$.
\end{itemize}
\end{ass}

\begin{notn}\label{notn.zeta-xi-sigma}
We set $Y \coloneqq \pr_{\le e_m}(X)$.
For $x \in X$, we define:
\begin{itemize}
	\item $\zeta_\ell \coloneqq \zeta_\ell(x) \coloneqq \dist(\pr_{\le e_\ell}(x), \bmdl^{e_\ell} \setminus \pr_{\le e_\ell}(X))\quad$ for $1 \le \ell \le m$
	\item $\sigma_\ell \coloneqq \sigma_\ell(x) \coloneqq \max\{1, \norm{\pr_{>e_{\ell}}(x)} \cdot \zeta_{\ell-1}(x)^{-1}\}\quad$ for $2 \le \ell \le m$.
\end{itemize}
The shorter notation $\zeta_\ell$, $\sigma_\ell$ will implicitly refer to a given point $x \in X$ in context. Note
that $\zeta_\ell$ and $\sigma_\ell$ implicitly also depend on $X$.
\end{notn}

\begin{figure}
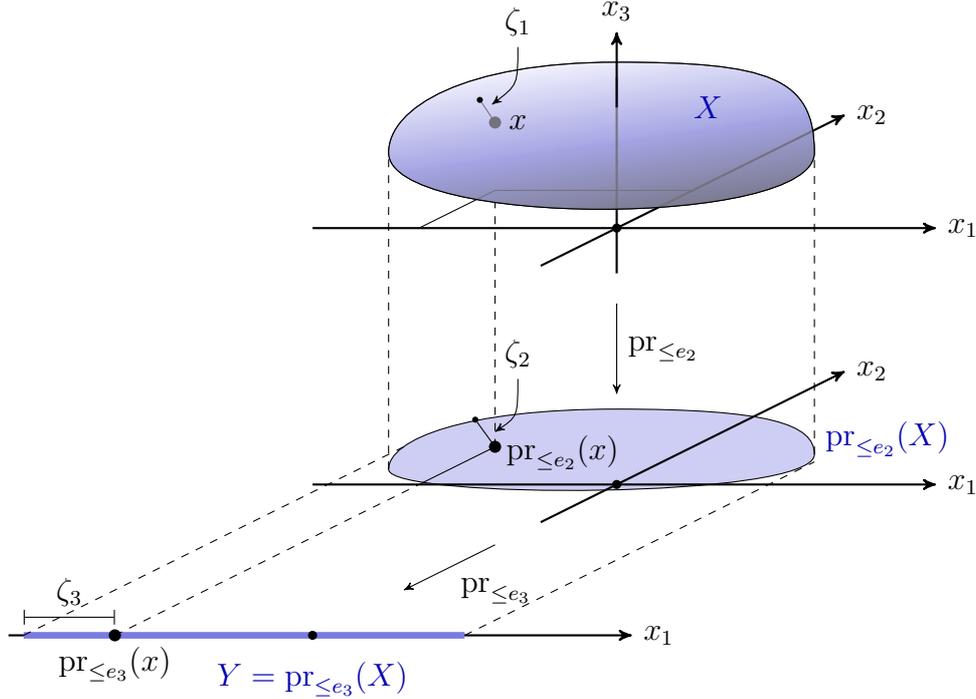

\begin{ctikzpicture}[x=2cm,y=2cm]
\begin{scope}[yshift=3.4cm]   
  \coordinate(x) at (-.8,.7);
\end{scope}   
  
\begin{scope}   
  \draw
    [preaction={S0set}]
       (-1.5,.1)
        .. controls (-1.5,-.1) and (1.3,-.1) .. (1.3,.2)
        .. controls (1.3,.4) and (1,.5) ..  (0,.5)
        .. controls (-1,.5) and (-1.5,.3) .. (-1.5,.1) -- cycle;

  \fill(0,0) circle (.03);
  \draw[axis arrow] (-2,0) -- (2.1,0) node[right]{$x_1$};
  \draw[axis arrow] (-.5,-.25) -- (1.5,.75) node[right]{$x_2$};

  \fill (-.93, .43) circle (.02);   

  \draw (-1.3,0) -- (-.8,.25); 
  \fill(-.8,.25) circle (.04);
  \draw (-.8,.25) -- (-.93, .43);

  \draw[arrow] (-.65,.7) node[above] {$\zeta_{2}$} .. controls (-.65,.45) .. (-.8,.35);

  \node at (-.8,.4) [anchor=north west] {$\pr_{\le e_2}(x)$};

  \node[S0text] at (1.3,.3) [anchor=west] {$\pr_{\le e_2}(X)$};

  \draw[dashed] (-1.5,.1) -- +(0,2.1);
  \draw[dashed] (1.3,.2) -- +(0,2);
  \draw[dashed] (-.8,.25) -- (x);
  \draw[arrow] (0,1.2) -- node {$\pr_{\le e_2}$} (0,.6);
\end{scope}   

\begin{scope}[yshift=3.4cm]   
  \def\egg#1{ %
    \draw #1 (-1.5,.5)
          .. controls (-1.5,0) and (1.3,0) .. (1.3, .5)
          .. controls (1.3,1) and (1,1.1) ..  (0,1.1)
          .. controls (-1,1.1) and (-1.5,.8) .. (-1.5,.5) -- cycle;}

  \ifarxiv
    \egg{[bottom color=black,top color=white,
        middle color=white!65!black!60!blue,
        shading angle=20]}
  \fi

  \draw (-1.3,0) -- (-.8,.25) -- (.5, .25);

  \ifarxiv
    \draw[axis noarrow] (0,-.3) -- (0,1);
  \else
    \draw[preaction={draw,ultra thick,white},axis noarrow] (0,-.3) -- (0,1);
  \fi
  \fill(0,0) circle (.03);
  \draw[axis arrow] (-2,0) -- (2.1,0) node[right]{$x_1$};
  \draw[axis arrow] (-.5,-.25) -- (1.5,.75) node[right]{$x_2$};

  \coordinate(x) at (-.8,.7);
  \draw (x) -- (-.9,.85);
  \fill(x) circle (.04);

  \ifarxiv
    \egg{[preaction={fill,color=white!90!blue,opacity=.5}]}
  \else
    \egg{[preaction={fill,color=white,opacity=.4}]
    [preaction={bottom color=black,top color=white,middle color=white!70!black!60!blue,opacity=.7,
         shading angle=20}]}
  \fi

  \fill (-.9,.85) circle (.02);   

  \node[S0text] at (.6,.8) {$X$};

  \node at (-.65,.7) {$x$};

  \draw[arrow] (-.65,1.2) node[above] {$\zeta_{1}$} .. controls (-.65,.95) .. (-.82,.82);

  \draw[axis arrow] (0,.8) -- (0,1.3) node[above]{$x_3$};
\end{scope}   

\begin{scope}[yshift=-2cm,xshift=-4cm]   

  \draw[axis arrow] (-2,0) -- (2.1,0) node[right]{$x_1$};

  \draw[S0set1] (-1.9,0) -- (1,0);

  \fill(0,0) circle (.03);

  \fill(-1.3,0) circle (.04);

  \node at (0,-.1) [anchor=north,S0text] {$Y = \pr_{\le e_3}(X)$};  \node at (-1.3,0) [anchor=north] {$\pr_{\le e_3}(x)$};

  \draw[|-|] (-1.9, .12) -- node {$\zeta_3$} (-1.3, .12);

  \draw[dashed] (-1.9,0) -- +(2.5,1.25);
  \draw[dashed] (1,0) -- +(2.3,1.15);
  \draw[dashed] (-1.3,0) -- +(2,1);
  \draw[arrow] (1.2,.6) -- node {$\pr_{\le e_3}$} +(-.6,-.3);
\end{scope}   

\end{ctikzpicture}
\caption{A picture illustrating some of Notation~\ref{notn.zeta-xi-sigma},
in the case $m = 3$, $e_1 = 3$, $e_2 = 2$, $e_3 = 1$.
}
\label{fig.zeta-xi-sigma}
\end{figure}

Some of this notation is illustrated in Figure~\ref{fig.zeta-xi-sigma}.
The purpose of $\sigma_\ell$ is the following.
One can only expect to obtain the best bounds on $\partial_1 f(x), \dots, \partial_{e_m} f(x)$ at those $x$ satisfying $\vv(\sigma_\ell) = 0$ for all $\ell$. 
(Note that in the case $m = 2$, the condition $\vv(\sigma_2) = 0$
is exactly equivalent to (\ref{eq.int.sigma}).)
However, even for $x \in X$ not satisfying those conditions, it is possible to obtain a weakened bound, where the weakening is expressed in terms of the valuations of the $\sigma_\ell$. This leads to the following definition of sedated functions.

\begin{defn}[Sedated functions]\label{defn.sedated}
	Suppose that $m$, $e_\ell$ and $X$ are given as in
	Assumption~\ref{ass.meX},
	and suppose that $f\colon X \fun \bmdl$ is an \LT-definable function. We consider three different versions: $\prep v \in \{\prep a, \prep b, \prep c\}$. In Version \prep b, we additionally assume $m \ge 2$.
	We call $f$ \emph{$e_{[1,m]}$-\prep v-sedated} (on $X$)
	if it is $C^1$ and if, for every $x \in X$
	and every $1 \le i \le e_{m}$, we have
	\begin{eqnarray}
		\vv(\partial_{i}f(x)) \ge \vv(u_{\prep v}(x)) - \vv(\zeta_{m}(x)) + \sum_{\ell=2}^{m} \vv(\sigma_\ell(x)),\qquad\text{where}\label{eq.sedated}\\
		u_{\prep a}(x) = f(x),
		\qquad
		u_{\prep b}(x) = \max\{|f(x)|, \norm{\pr_{>e_{2}}(x)}\},
		\qquad
		u_{\prep c}(x) = 1.
\end{eqnarray}

We call an \LT-definable function $X \fun \bmdl^n$ $e_{[1,m]}$-\prep v-sedated if each of its coordinate functions is $e_{[1,m]}$-\prep v-sedated.
\end{defn}

In this notation, ``$e_{[1,m]}$'' is supposed to be considered as a short hand notation for the tuple $(e_1, \dots, e_m)$. In particular, for $1 \le k \le \ell \le m$ and $f$ a function on a subset of $\bmdl^{e_k}$, we also have a notion of being
$e_{[k,\ell]}$-\prep v-sedated.

\begin{rem}\label{rem.shrink-X}
	Equation (\ref{eq.sedated}) depends on the domain $X$, since $\zeta_\ell$ does.
	Nevertheless, if $f$ is $e_{[1,m]}$-\prep v-sedated, then so
	is the restriction of $f$ to any subset of $X$. Indeed, by shrinking $X$,
	$\zeta_\ell$ can only become smaller and $\sigma_\ell$ can only become bigger,
	both of which make (\ref{eq.sedated}) easier to be satisfied.
\end{rem}

\begin{rem}
If, in \prep b-sedation, one allows $e_1 = e_2$, then \prep a can be considered as a special case of \prep b via some renumbering. However, for clarity, we wrote down the two cases separately.
\end{rem}

\begin{prop}[Sedating functions]\label{prop.sedate}
	Fix $\prep v \in \{\prep a, \prep b, \prep c\}$.
	Let $m$, $e_\ell$, $X$, $Y$ be as in Assumption~\ref{ass.meX} and Notation~\ref{notn.zeta-xi-sigma}
	(with $m \ge 2$ in Version \prep b)
	and suppose that $f\colon X \fun \bmdl$ is an
	\LT-definable function which is $e_{[1,m']}$-\prep v-sedated for $1 \le m' < m$ (or $2 \le m' < m$, in Version \prep b). Suppose moreover that
\begin{equation}\label{eq.addit-cond}
\begin{cases}
\prep a &\text{(no additional condition)} \\
\prep b &f \text{ is } C^1 \text{ and } \vv(\nabla f(x)) \ge 0 \\
\prep c &\vv(f(x)) \ge 0.
\end{cases}
\end{equation}
	Then there exists an \LT-definable set $Z \subseteq Y$ of dimension less than $e_m$ such that
	the restriction of $f$ to $X \setminus \pr_{\le e_m}^{-1}(Z)$ is $e_{[1,m]}$-\prep v-sedated.
\end{prop}

\begin{rem}\label{rem.sedate-high-dim}
	The proposition direcly implies the corresponding result for functions with range $\bmdl^n$, by applying it to each of the coordinate functions.
\end{rem}

\begin{rem}
In our application of this proposition, the bound (\ref{eq.sedated}) will only be needed on the subset
$X' \coloneqq \{x \in X :\vv(\sigma_2) = \dots = \vv(\sigma_m) = 0 \}$,
i.e., where the sum disappears and the bound is ``best possible''. Nevertheless, we need to work with a notion of sedated functions imposing a bound on all of $X$ for the following somewhat strange reason.
The proof of Proposition~\ref{prop.sedate} only works if $X$ and $f$ both are \LT-definable; in particular, the ``induction hypothesis'' (that $f$ is $e_{[1,m']}$-\prep v-sedated for $m' < m$) is needed on an \LT-definable set, so we need a formulation of that hypothesis which we can prove on all of $X$, and not just on $X'$.
\end{rem}

The strategy of the proof of Proposition~\ref{prop.sedate} is as follows. We will use Lemma~\ref{lem.choose-max} to choose,
for each $y \in Y$, an element $x = \tau(y) \in X_y \coloneqq \{x \in X : \pr_{\le e_m}(x) = y\}$ where the difference between the two sides of
(\ref{eq.sedated}) is worst, i.e., where the left hand side minus the right hand side is minimal.
In particular, it suffices to prove that (\ref{eq.sedated}) holds for those $x$. Corollary~\ref{cor.jac} allows us to shrink $Y$ in such a way that we obtain good bounds on the derivatives
of $f(\tau(y))$ in terms of $\valdist(y, \bmdl^{e_m} \setminus Y) = \vv(\zeta_m)$. We then obtain (\ref{eq.sedated}) by combining
these bounds with the assumption about $e_{[1,m']}$-\prep v-sedation for $m' < m$.

To be able to apply Lemma~\ref{lem.choose-max} as described above, we need the difference of the two sides of (\ref{eq.sedated}) to be bounded on each fiber $X_y$.
Such a bound can be obtained from Equation (\ref{eq.sedated}) for $e_{[1,m-1]}$-sedation, provided that we fix a lower bound on $\norm{\pr_{>e_{m}}(x)}$.
Thus, before applying the above strategy, we will treat points $x$ with small $\norm{\pr_{>e_{m}}(x)}$ separately.
The idea for this is that for each fixed (small) $d \ge 0$, we can apply the same strategy as before to the subset $\{x \in X : \norm{\pr_{>e_{m}}(x)} = d\}$. Different $d$ yield different sets $Z_{d}$ to be removed from $Y$ for (\ref{eq.sedated}) to hold. Instead of removing all of them from $Y$ (which would be too much), we remove the limit (in a suitable sense) of $Z_{d}$ for $d \rightarrow 0$; this does not imply (\ref{eq.sedated}) b itself, but it does allow us to bound by how much it fails, and that is enough for applying the above strategy to the remainder of $X$.

Here are the details.

\begin{proof}[Proof of Proposition~\ref{prop.sedate}]
During the proof, we will construct a set $Z$ of dimension less than $e_m$ which we will successively enlarge until
the proposition is satisfied. More precisely, we will obtain something slightly stronger: We will find a $Z \subseteq \bmdl^{e_m}$ of dimension less than $e_m$ such that
\begin{equation}\label{eq.sedated-hat}
\vv(\partial_{i}f(x)) \ge \vv(u_{\prep v}(x)) - \valdist(\pr_{\le e_m}(x), Z) + \sum_{\ell=2}^{m} \vv(\sigma_\ell(x))
\end{equation}
holds for every $x \in \hat X \coloneqq X \setminus \pr_{\le e_m}^{-1}(Z)$. This then implies that $f\rest \hat X$ is $\prep v$-sedated
(using Remark~\ref{rem.shrink-X} concerning the $\sigma_\ell$).

In a very first step, we ensure that $f$ is $C^1$: In Version \prep b, this is an assumption; in the other versions, if $m \ge 2$, it follows from the assumption that $f$ is (say) $e_{[1,1]}$-\prep v-sedated, and if $m = 1$, this can be achieved by removing a suitable subset from $X = Y$.

\medskip

Fix $i \le e_m$. Equation (\ref{eq.sedated-hat}) can be rewritten as
\begin{equation}\label{eq.g-goal}
	\vv(g_i(x)) \ge - \valdist(\pr_{\le e_m}(x), Z),
\end{equation}
where
\begin{equation}
	g_i(x) \coloneqq  \partial_{i}f(x) \cdot u_{\prep v}(x)^{-1} \cdot \prod_{\ell=2}^{m} \sigma_\ell^{-1}.
	\label{eq.def-gi}
\end{equation}

As in the above sketch of proof, given $y \in Y$, we write $X_y$ for the fiber over $X$ above $y$, and similarly, if $X' \subseteq X$ is a subset, we set $X'_y \coloneqq X' \cap X_y$.

We will prove the following.

\medskip

\textbf{Claim~1}:
Suppose that $Z \subseteq \bmdl^{e_m}$ is an \LT-definable set of dimension less than $e_m$ and that $X' \subseteq X$ is an \LT-definable subset
such that for every $y \in Y \setminus Z$ and every $i \le e_m$,
$|g_i|$ is bounded on the fiber $X'_{y}$. Then there
exists an \LT-definable set $\hat Z \supseteq Z$ of dimension less than $e_m$ such that we have
\begin{equation}\label{eq.claim1}
\vv(g_i(x)) \ge -\valdist(\pr_{\le e_m}(x), \hat Z) \qquad \text{for every }
x \in X'.
\end{equation}

\medskip

Before proving Claim~1, we show how it implies the proposition.
It suffices to prove that the set $Z$ of $y \in Y$ such that
$|g_i(x)|$ is unbounded on the fiber $X_y$ has dimension less than $m$. Indeed, then we obtain (\ref{eq.g-goal}) by applying the claim
to $X' \coloneqq X$.

If $m = 1$, then $|g_i(x)|$ is bounded on each $X_y$ for the trivial reason that $X_y$ is a singleton; thus assume $m \ge 2$.

To bound $|g_i(x)|$, we first check that for every $x \in X$, we have
\begin{equation}\label{eq.fiber-bound}
	\vv(g_i(x)) \ge -\vv(\pr_{>e_{m}}(x)).
\end{equation}
We have
\begin{equation*}
	\vv(g_i(x)) = \underbrace{\vv(\partial_{i}f(x)) - \vv(u_{\prep v}(x)) - \sum_{\ell=2}^{m-1} \vv(\sigma_\ell)}_{(*)} - \vv(\sigma_m).
\end{equation*}
In Version \prep b, if $m = 2$ then the last term in $(*)$ disappears, and hence (\ref{eq.fiber-bound}) follows from the following three items: the assumption (\ref{eq.addit-cond}), $\vv(u_{\prep b}(x)) \le \vv(\pr_{>e_{2}}(x))$, and $\vv(\sigma_2) \le 0$. In all other cases,
the assumption that $f$ is $e_{[1,m-1]}$-\prep v-sedated
implies $(*) \ge -\vv(\zeta_{m-1})$, which, together with
$\vv(\sigma_m) \le \vv(\pr_{>e_{m}}(x))-\vv(\zeta_{m-1})$ (by the definition of $\sigma_m$), implies (\ref{eq.fiber-bound}).
	
For $d \in \bmdl_{\ge 0}$, set
\begin{align*}
	X'_{d} &\coloneqq \{x \in X : \norm{\pr_{>e_{m}}(x)} = d\};
\end{align*}
by (\ref{eq.fiber-bound}), $|g_i|$ is bounded on $X'_d$ for each fixed $d$, so Claim~1 (used in the language $\LTd$) yields an \LTd-definable set $Z_d \subseteq \bmdl^{e_m}$ of dimension less than $m$ and such that we have
\begin{equation}\label{eq.claim1-d}
\vv(g_i(x))\ge -\valdist(y, Z_d)
\end{equation}
for $x \in X'_d$ and $y \coloneqq \pr_{\le e_m}(x)$.
By the Compactness Theorem,
we may assume that the sets $Z_d$ are defined uniformly in $d$,
so that the following sets are \LT-definable:
\begin{align*}
Z_\bullet &\coloneqq \bigcup_{d \ge 0} (Z_d \times \{d\}) \subseteq \bmdl^{e_m} \times \bmdl_{\ge 0}
\qquad\text{and}
\\
Z &\coloneqq \{y \in \bmdl^{e_m} : (y, 0) \in \cl(Z_\bullet)\},
\end{align*}
(where $\cl(Z_\bullet)$ denotes the topological closure of $Z_\bullet$).
Since $Z_d$ has dimension less than $e_m$ for every $d$, we have $\dim \partial Z_\bullet < \dim Z_\bullet \le e_m$ and hence $\dim Z < e_m$
(since $Z \sub \partial Z_\bullet \cup Z_0$).
We claim that $|g_i|$ is bounded on each fiber $X_y$ with $y \ne Z$.

Fix $y \in Y \setminus Z$ and
consider  $x \in X_y$ with $y \in Y \mi Z$ and
set $d \coloneqq \norm{\pr_{>e_{m}}(x)}$ (so that $x \in X'_d$).
Inequality (\ref{eq.fiber-bound}) provides a bound on $|g_i(x)|$ for big $d$, and for small $d$, we will obtain a bound from (\ref{eq.claim1-d}). More precisely, set
\[
d_0 \coloneqq \dist((y,0),Z_\bullet)
\]
(which is strictly positive, by definition of $Z$). By (\ref{eq.fiber-bound}) it suffices to bound $|g_i(x)|$ for those $x$ satisfying $\vv(\pr_{> e_m}(x)) > \vv(d_0)$.
This implies
$\valdist((y,d),Z_\bullet) = \valdist((y,0),Z_\bullet)$, and hence we must have 
\[
\valdist(y, Z_d) = \valdist((y,d), Z_d \times \{d\})  \leq \valdist((y,d),Z_\bullet) = \vv(d_0).
\]
So for such a
$d$, we obtain
\[
\vv(g_i(x))\overset{(\ref{eq.claim1-d})}{\ge} -\valdist(y, Z_d) \ge -\vv(d_0).
\]

Thus $|g_i|$ is bounded on all of $X_y$,
which finishes the proof that Claim~1 implies the proposition.

\medskip
	
\textbf{Proof of Claim~1:}
Even though the case $m = 1$ (Versions \prep a, \prep c) does not need to be treated separately, we do note that for $m = 1$, Claim~1 follows directly
by applying Corollary~\ref{cor.jac} to $f$ (and using (\ref{eq.addit-cond}) in Version \prep c).
	
Fix $i \le e_m$ for the entire proof of the claim. (We can treat each $g_i$ separately.)

Set $Y' \coloneqq \pr_{\le e_m}(X') \setminus Z$.
For $y \in Y'$, $|g_i|$ is bounded on $X'_{y}$, so we can apply Lemma~\ref{lem.choose-max} to the restriction $g_i\rest X'_{y}$, using the language $\LTy$. Doing this for all $y \in Y'$ (and applying the Compactness Theorem) yields an \LT-definable function $\tau\colon Y' \fun X$
with $\tau(y) \in X'_{y}$ such that
\begin{equation}\label{eq.is-max}
  \vv(g_i(x)) \ge \vv(g_i(\tau(y))) \qquad\text{for all } x \in X'_{y} \text{ and } y \in Y'.
\end{equation}
We will prove that after a suitable enlargement of $Z$, we obtain
\begin{equation}\label{eq.claim-goal1}
  \vv(g_i(\tau(y))) \ge -\valdist(y, Z)
   \qquad\text{for every } y \in Y';
\end{equation}
together with (\ref{eq.is-max}), this implies (\ref{eq.claim1}).

In the remainder of the proof, $\zeta_\ell$ and $\sigma_\ell$ always refer to the point $x \coloneqq \tau(y)$.
Plugging (\ref{eq.def-gi}) (the definition of $g_i$) into (\ref{eq.claim-goal1}) yields a condition on $\partial_{i} f$:
\begin{equation}\label{eq.claim-goal2}
	\vv(\partial_i f(\tau(y))) \ge \vv(u_{\prep v}(\tau(y)))   -\valdist(y, Z) +  \sum_{\ell=2}^{m} \vv(\sigma_\ell).
\end{equation}
Consider the derivative of the function $h(y) \coloneqq f(\tau(y))$ with respect to the $i$th coordinate.
Using the notation
\begin{equation}\label{eq.tau}
	\tau(y) = (y, \tau_{e_m+1}(y), \dots, \tau_{e_1}(y)) ,
\end{equation}
we can write it as
\begin{equation}\label{eq.claim-sum}
	\partial_i h(y) = \partial_{i} f(\tau(y)) + \sum_{k = e_m+1}^{e_1} \partial_{k} f(\tau(y)) \cdot \partial_i \tau_k(y),
\end{equation}
so to obtain (\ref{eq.claim-goal2}), it suffices to prove that in (\ref{eq.claim-sum}), (i) the left hand side and (ii) all summands of the sum over $k$
have valuation at least that of the right hand side of (\ref{eq.claim-goal2}).

For (i), apply Corollay~\ref{cor.jac} to $h$ (extended trivially outside of $Y'$). This yields that, by enlarging $Z$, we can achieve
\begin{equation}\label{eq.estimate-a}
	\vv(\partial_i h(y)) \ge \vv(h(y)) -\valdist(y, Z).
\end{equation}
Since the sum in (\ref{eq.claim-goal2}) is at most $0$
(by definition of $\sigma_\ell$), it remains to check that $\vv(f(\tau(y))) \ge \vv(u_{\prep v}(\tau(y)))$;
this follows from the definition of $u_{\prep v}$, and,
in Version \prep c, (\ref{eq.addit-cond}).

For (ii), fix $k$ (with $e_m < k \le e_1$) and choose $m'$ such that $e_{m'+1} < k \le e_{m'}$; note that $m' < m$. Our goal is to prove
\begin{equation}\label{eq.claim-goal-ii}
\vv(\partial_{k} f(x)) + \vv(\partial_i \tau_k(y))
\ge \vv(u_{\prep v}(x)) - \valdist(y, Z) +  \sum_{\ell=2}^{m} \vv(\sigma_\ell)
\end{equation}
(where $x = \tau(y)$). Applying Corollay~\ref{cor.jac} to $\tau_k$ (again, extended trivially outside of $Y'$) yields, after further enlarging $Z$,
\begin{equation}\label{eq.ingredient-1}
\vv(\partial_i \tau_k(y))
\ge \vv(\tau_k(y)) - \valdist(y, Z)
\ge \vv(\pr_{>e_{m'+1}}(x)) - \valdist(y, Z).
\end{equation}

In the case $m' = 1$ of Version \prep b, (\ref{eq.claim-goal-ii}) now follows from these three items: (\ref{eq.addit-cond}) (which implies $\vv(\partial_{k}f(x)) \ge 0$),
$\vv(u_{\prep v}(x)) \le \vv(\pr_{>e_{2}}(x))$, and $\vv(\sigma_\ell) \le 0$.
Thus we may now suppose that either $m' \ge 2$ or that we are not in Version \prep b. Then
the assumption that $f$ is $e_{[1,m']}$-\prep v-sedated  implies
\begin{equation}\label{eq.ingredient-2}
	\vv(\partial_{k}f(x)) \ge \vv(u_{\prep v}(x)) - \vv(\zeta_{m'}) + \sum_{\ell=2}^{m'} \vv(\sigma_\ell),
\end{equation}
and (\ref{eq.claim-goal-ii}) follows by taking the sum
of (\ref{eq.ingredient-1}) and (\ref{eq.ingredient-2})
and then noting that $\vv(\pr_{>e_{m'+1}}(x)) - \vv(\zeta_{m'}) \ge \vv(\sigma_{m'+1})$ and $\vv(\sigma_\ell) \le 0$.

This finishes the proof of (ii), and hence of (\ref{eq.claim-goal2}), and
hence of Claim~1, and hence of Proposition~\ref{prop.sedate}.
\end{proof}

The notion of \prep c-sedation will be applied to the first derivatives of a function, to control its second derivatives. We introduce a corresponding notion.
(Note that similar kinds of bounds also appear in \cite{NV.lip}.)

\begin{defn}[\prep{c_2}-sedated functions]
		Suppose that $m$, $e_\ell$ and $X$ are given as in
		Assumption~\ref{ass.meX}.
		We call an \LT-definable function $f\colon X \fun \bmdl$ \emph{$e_{[1,m]}$-\prep{c_2}-sedated} if it is $C^2$, $\vv(\Jac_x f) \ge 0$ for every $x \in X$, and for $1 \le i \le e_m, 1 \le j \le e_1$, we have
	\begin{equation}\label{eq.second-sedated}
		\vv(\partial_{ij}f(x)) \ge  -\vv(\zeta_{m}(x)) + \sum_{\ell=2}^{m} \vv(\sigma_\ell(x)),
	\end{equation}
	where $\zeta_m$ and $\sigma_\ell$ are as in Notation~\ref{notn.zeta-xi-sigma}.
	We call an \LT-definable function $X \fun \bmdl^n$ $e_{\le m}$-\prep{c_2}-sedated if each of its coordinate functions is $e_{[1,m]}$-\prep{c_2}-sedated.
\end{defn}

\begin{cor}[\prep{c_2}-sedating functions]\label{cor.second-sedate}
	Let $m$, $e_\ell$, $X$, $Y$ be as above, and suppose that $f\colon X \fun \bmdl$ is an
	\LT-definable function which is $e_{[1,m']}$-\prep{c_2}-sedated for all $m' < m$.
	Suppose moreover that $\vv(\nabla f) \ge 0$
	(this follows anyway if $m \ge 2$).
	Then there exists an \LT-definable set $Z \subseteq Y$ of dimension less than $e_m$ such that
	the restriction of $f$ to $X \setminus \pr_{\le e_m}^{-1}(Z)$ is $e_{[1,m]}$-\prep{c_2}-sedated.
\end{cor}

\begin{proof}
    If $m = 1$, we start by removing a lower-dimensional subset from $X = Y$ to ensure that $f$ is $C^2$. (If $m \ge 2$, $f$ is already $C^2$.)
   	Then we apply Proposition~\ref{prop.sedate} \prep c to each of the derivatives $\partial_j f$ ($1 \le j \le e_1$).
\end{proof}

We finish this subsection by proving that being sedated is preserved under certain kinds of transformations, which will be the building blocks of the rectilinearization maps mentioned in Subsection~\ref{sect.overview}.

\begin{lem}[Sedation and rectilinearization]\label{lem.recti-sedated}
Fix $\prep v \in \{\prep a, \prep b, \prep{c_2}\}$, and
let the following be given: 
\begin{itemize}
 \item integers $1 \le m' < m$
($2 \le m' < m$ in Version \prep b), 
 \item integers $e_1 > \dots > e_m > 0$, 
 \item \LT-definable sets $X, \hat X \subseteq \bmdl^{e_1}$,
 \item \LT-definable functions $f\colon X \fun \bmdl$
and  $\hat f\colon \hat X \fun \bmdl$.
\end{itemize}
Suppose that there exists an \LT-definable bijection
$\psi\colon \hat X \fun X$ such that
$\hat f = f \circ \psi$ and which
sends
\begin{align*}
\hat x =
&(x\U m,\hat x\U{m-1}, x\U{\star}) \in
\bmdl^{e_m} \times \bmdl^{e_{m-1} - e_m} \times \bmdl^{e_1 - e_{m-1}}
\qquad \text{to}\\
x =
&(x\U m, x\U{m-1}, x\U{\star}) = (x\U m, \hat x\U{m-1} + g(x\U{m}), x\U{\star}),
\end{align*}
where $g\colon \pr_{\le e_m}(X) \fun \bmdl^{e_{m-1} - e_m}$ is $e_{[m,m]}$-\prep{c_2}-sedated.
Then $f$ is $e_{[1,m']}$-\prep v-sedated iff $\hat f$ is $e_{[1,m']}$-\prep v-sedated.
\end{lem}

\begin{proof}
The lemma is symmetric with respect to swapping $\hat X$ and $X$; we will carry out various arguments only in one direction without further notice.

We start by verifying that for $\ell \le m - 1$, the valuations $\vv(\zeta_\ell)$ and $\vv(\sigma_\ell)$ from Notation~\ref{notn.zeta-xi-sigma} are preserved by $\psi$.
More precisely, we show that, for $\hat x = (x\U m,\hat x\U{m-1}, x\U{\star}) \in \hat X$
and $x = (x\U m, x\U{m-1}, x\U{\star}) = \psi(\hat x)$, we have
\begin{equation}
\label{eq.recti-zeta}
\valdist(\pr_{\le e_\ell}(\hat x), \bmdl^{e_\ell} \setminus \pr_{\le e_\ell}(\hat X))
=
\valdist(\pr_{\le e_\ell}(x), \bmdl^{e_\ell} \setminus \pr_{\le e_\ell}(X))
\end{equation}
for $1 \le \ell  \le m-1$. Since $\pr_{> e_{m-1}} \circ \psi = \pr_{> e_{m-1}}$, this then also implies 
\begin{equation}
\label{eq.recti-sigma}
\vv(\sigma_\ell(\hat x)) = \vv(\sigma_\ell(x)),
\end{equation}
where $\sigma_\ell(\hat x)$ is computed with respect to $\hat X$ and
$\sigma_\ell(x)$ is computed with respect to $X$.
 
To prove (\ref{eq.recti-zeta}), we assume $\ell = 1$;
for other $\ell$, the same proof applies, after replacing
$x$, $\hat x$, $X$, $\hat X$ by their projections to $\bmdl^{e_\ell}$.

Both sides of (\ref{eq.recti-zeta}) are no less than
$\mu := \valdist(x\U m, \bmdl^{e_m} \setminus \pr_{\le e_m}(X))$, so it suffices to verify that given an element $\hat y \in \bmdl^{e_1} \setminus\hat X$ satisfying $\vv(\hat y - \hat x) > \mu$, we can find an element $y \in \bmdl^{e_1} \setminus X$ satisfying $\vv(y - x) = \vv(\hat y - \hat x)$.

We write $\hat y = (y\U m,\hat y\U{m-1}, y\U{\star}) \in
\bmdl^{e_m} \times \bmdl^{e_{m-1} - e_m} \times \bmdl^{e_1 - e_{m-1}}$.
By definition of $\mu$, the function $g$ is defined on the entire ball $B := B_{> \mu}(x\U m)$. This means that, first of all, $y := (y\U m, \hat y\U{m-1} + g(y\U{m}), y\U{\star})$ is well-defined, and secondly, the 
Mean Value Theorem argument from Remark~\ref{rem.int-val} applies, yielding
\[
\vv(g(y\U m) - g(x\U m)) \ge \vv(y\U m - x\U m);
\]
now an easy computation yields $\vv(y - x) = \vv(\hat y - \hat x)$, as desired.

From (\ref{eq.recti-zeta}) and (\ref{eq.recti-sigma}), we obtain, for $x = \psi(\hat x)$:
\begin{equation*}
-\vv(\zeta_{m'}(x)) + \sum_{\ell=2}^{m'} \vv(\sigma_\ell(x)) =
-\vv(\zeta_{m'}(\hat x)) + \sum_{\ell=2}^{m'} \vv(\sigma_\ell(\hat x))
=:\lambda(x).
\end{equation*}

That the function $g$ is \prep{c_2}-sedated in particular means that $\vv(\Jac g)\ge 0$.
This yields the following equations concerning the partial derivatives of $\hat f(x\U m,\hat x\U{m-1}, x\U{\star}) = f(x\U m,  \hat x\U{m-1} + g(x\U{m}), x\U{\star})$:
\begin{equation}\label{eq.same-derivative}
	\begin{alignedat}{2}
		\min_{1 \le i \le e_{m-1}}\vv(\partial_{i}\hat f(\hat x)) &= \min_{1 \le i \le e_{m-1}}\vv(\partial_{i}f(x))
		\qquad&&\text{and}\\
		\vv(\partial_{i}\hat f(\hat x)) &= \vv(\partial_{i}f(x))
		\qquad&&\text{for } e_{m-1} < i \le e_1
		.
	\end{alignedat}
\end{equation}

Now we are ready to prove the claims of the lemma.
For $\prep v =  \prep a, \prep b$, $f$ is $e_{[1,m']}$-\prep v-sedated iff
\begin{equation}\label{eq.recti-ab}
	\min_{1 \le i \le e_{m'}}\vv(\partial_{i}f(x)) \ge \vv(u_{\prep v}(x)) + \lambda(x)
	\qquad \text{for every } x \in X,
\end{equation}
and similarly for $\hat f$.
The left hand sides of (\ref{eq.recti-ab}) are equal for $f$
and $\hat f$ by (\ref{eq.same-derivative}), and the right hand sides are equal since $\hat f(\hat x) = f(x)$ and, in Version \prep b (which implies $m \ge 3$), $\pr_{>e_{2}}(x) =\pr_{>e_{2}}(\hat x)$.

Finally, suppose that $f$ is $e_{[1,m']}$-\prep{c_2}-sedated, i.e.,
\begin{equation*}
	\vv(\Jac f) \ge 0 \qquad \text{and} \qquad
	\vv(\partial_{ij}f) \ge \lambda(x)
	\qquad \text{for } 1 \le i \le e_{m'}, 1 \le j \le e_{1}.
\end{equation*}
(To simplify notation, from now on, we omit the points at which the derivatives are taken.)
Using (\ref{eq.same-derivative}), we obtain $\vv(\Jac \hat f) \ge 0$,
and it remains to verify that $\vv(\partial_{ij}\hat f) \ge \lambda$.
Direct computation of this second derivative yields the following, where $g = (g_{e_m+1}, \dots, g_{e_{m-1}})$ and where we set $\partial_k g \coloneqq 0$ for $k > e_m$:
\begin{align*}
	\partial_{ij}(f \circ \psi)
	&= \partial_{ij} f
	+ \sum_{e_m <\ell \le e_{m-1}} \partial_{j\ell} f \cdot \partial_i g_{\ell}
	+ \sum_{e_m <\ell \le e_{m-1}} \partial_{i\ell} f \cdot \partial_j g_{\ell}\\
	&+ \sum_{e_m <\ell,\ell' \le e_{m-1}} \partial_{\ell\ell'} f \cdot \partial_i g_{\ell}\cdot \partial_j g_{\ell'}
	+ \sum_{e_m <\ell \le e_{m-1}}\partial_\ell f \cdot \partial_{ij}g_\ell
	.
\end{align*}
All the second derivatives of $f$ appearing on the right hand side have valuation at least $\lambda$
(note that $\ell, \ell' \le e_{m'}$). Together with
$\vv(\Jac g) \ge 0$, we get the desired bound for everything except the last sum.
In that one, we have $\vv(\partial_\ell f) \ge 0$
and $\vv(\partial_{ij}g_\ell) \ge -\vv(\zeta_m)$ (since $g$ is $e_{[m,m]}$-\prep{c_2}-sedated).
Now $-\vv(\zeta_m) \ge -\vv(\zeta_{m'}) \ge \lambda$ since $\vv(\sigma_\ell) \le 0$ for all $\ell$,
so also here, we get the desired bound.
\end{proof}

\section{The main proof}
\label{sect.proof}

This entire section constitutes the proof of Theorem~\ref{thm.main-na}.
We fix, once and for all, a closed, \LT-definable set
$X \subseteq \bmdl^n$.

\subsection{Some notation}
\label{sect.notn.std}

We fix some notation which will be useful at various places in the proof. Suppose that we have already fixed a
stratification $\mdl X$ of $X$ (in the sense of Definition~\ref{defn.strat}); in particular, we assume that each $X^i$ is closed.
We moreover assume that the strata, i.e., the definably connected components of the skeletons $\mathring X^i$, form a bradycell decomposition in the sense of Definition~\ref{defn.brady-dec}.
(Recall that definably connectedness always refers to the language $\LT$.)

\begin{notn}[Aligning and groups of coordinates]\label{notn.std}
Suppose that
$\mdl S = (S^\ell)_{0 \le \ell \le m}$
is a sequence of strata with
$S^\ell \subseteq \mathring X^{e_\ell}$ for some $e_0 \ge e_1 > e_2 > \dots > e_m$. These inequalities, together with $e_0 \le n$, imply $m \le n +1 $, so Proposition~\ref{prop.brady-dec}
provides an aligner $\aligner \in \CC_n$ such that
each transformed set $\aligner(S^\ell)$ is aligned in the sense of Definition~\ref{defn.aligned}.
In such a situation, i.e., when $\mdl S$ and an appropriate $\aligner$ are given, we will assume that the strata $S^\ell$ are already aligned by transforming our coordinate system using $\aligner$. (Why this assumption is harmless will be explained at the appropriate places.)
We will moreover use the following notation, where $0 \le \ell \le m$:
 \begin{itemize}
 	\item We write $\bar S^\ell \coloneqq \pr_{\le e_\ell}(S^\ell)$ for the base of $S^\ell$
 and $\rho^\ell \colon \bar S^\ell \fun \bmdl^{n - e_\ell}$ for the map whose graph is $S^\ell$.
 \item
 We introduce a notation for ``groups of coordinates'' of points $x = (x_1, \dots, x_n) \in \bmdl^n$:
\begin{align*}
x\U{m} &\coloneqq (x_{1}, x_{2}, \dots, x_{e_m}),\\
x\U{\ell} &\coloneqq (x_{e_{\ell+1}+1}, \dots, x_{e_{\ell}})
    \qquad \text{for } 0 \le \ell < m \text { and}\\
x\U\star &\coloneqq (x_{e_0+1}, \dots, x_{n}).
\end{align*}
In other words,
 \begin{equation*}
 x = (x\U m, x\U{m-1}, \dots, x\U0, x\U\star).
 \end{equation*}
(Note that $x\U0$ might be the empty tuple since possibly $e_0 = e_1$.)
We use a similar notation for points in $\bmdl^{e_\ell}$ and for the functions $\rho^\ell$ ($0 \le \ell \le m$):
 \begin{align*}
 x &= (x\U m, x\U{m-1}, \dots, x\U{\ell+1}, x\U\ell)
&&\text{for } x \in \bmdl^{e_\ell};\\
 \rho^\ell &= (\rho^\ell\U{\ell-1}, \dots, \rho^\ell\U0, \rho^\ell\U{\star}).
 \end{align*}
\end{itemize}
\end{notn}

Now suppose that we additionally have a val-chain $a^0, \dots, a^m$
with $a^\ell \in S^\ell$ and with distances $\lambda_1 > \dots > \lambda_{m+1}$ (and dimensions $e_0 \ge e_1 > \dots > e_m$).
There are natural balls $\bar B^\ell \subseteq \bar S^\ell$ associated with such a val-chain, though it requires an argument to see that the balls, as defined below, are really subsets of $\bar S^\ell$.

\begin{figure}
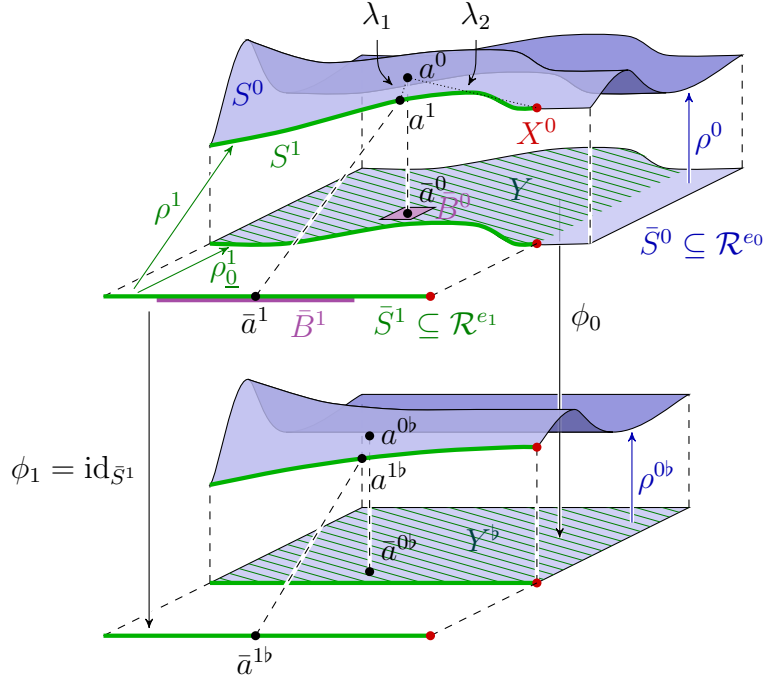

\begingroup
\def\piece(#1,#2)(#3)(#4)(#5,#6){
  plot[smooth,xshift=#1cm,yshift=#2cm] coordinates {\cright}
  --
  (\wid+#1,#2+\yf) .. controls (\wid+#3+\yf) and (\wid+#4+\yyf) .. (\wid+#5,#6+\yyf)
  --
  plot[smooth,xshift=#5cm,yshift=#6cm] coordinates {\cleft}
  --
  (#5,#6+\yya) .. controls (#4+\yya) and (#3+\ya) .. (#1,#2+\ya)
  -- cycle
}
\def\sety(#1/#2/#3/#4/#5/#6){
\def\ya{#1}\def\yb{#2}\def\yc{#3}\def\yd{#4}\def\ye{#5}\def\yf{#6}
}
\def\setyy(#1/#2/#3/#4/#5/#6){
\def\yya{#1}\def\yyb{#2}\def\yyc{#3}\def\yyd{#4}\def\yye{#5}\def\yyf{#6}
}

\tikzset{
Bset/.style={violet!40!white},
Btext/.style={violet!70!white}
}

\def\everything{

\def\crightlong{(0,\ya) (1,\yb) (2+2*\deform,\deform+\yc) (3+2*\deform,\deform+\yd) (4,\ye) (5,\yf)}
\def\cleftlong{(5,\yyf) (4,\yye) (3+2*\deform,\deform+\yyd) (2+2*\deform,\deform+\yyc) (1,\yyb) (0,\yya)}

\def\crightshort{(0,\ya) (1,\yb) (2+2*\deform,\deform+\yc) (3+2*\deform,\deform+\yd) (4,\ye) (4.3,\yf)}
\def\cleftshort{(4.3,\yyf) (4,\yye) (3+2*\deform,\deform+\yyd) (2+2*\deform,\deform+\yyc) (1,\yyb) (0,\yya)}

\def\fla{{\RI{\flat}}}

\LE{
\def\deform{.25}
\def\wid{5}

\let\cright=\crightlong
\let\cleft=\cleftlong
}

\RI{
\def\deform{0}
\def\wid{4.3}

\let\cright=\crightshort
\let\cleft=\cleftshort
}


\coordinate(a0pr) at (1.8+2*\deform+.3,-1.5+\deform+.15);
\coordinate(a2pr) at (4.3,-1.5);
\coordinate(a0) at (1.8+2*\deform+.3,.15+\deform+.3);
\coordinate(a1) at (2+2*\deform,.15+\deform);
\coordinate(a2) at (4.3,0.3);


\LE{
  \draw[ultra thick,white] (4.5,-1.5) -- +(.5,0);
}

\sety(0/0/0/0/0/0)
\setyy(0/0/0/0/0/0)
\draw
  [preaction={fill,S0setT}]
  \piece(0,-1.5)(0,-1.5)(2,-.5)(2,-.5);


\begin{scope}
  \let\cright=\crightshort
  \let\cleft=\cleftshort
  \def\wid{4.3}

  \clip
    \piece(0,-1.5)(0,-1.5)(2,-.5)(2,-.5);
  
  \foreach \ii in {-1.5,-1.4,...,2}
    \draw[very thin,green!50!black] (0,\ii) -- (7,\ii-2);
\end{scope}

\LE{
  \draw[preaction={fill,Bset}]
   \tangentplane(a0pr)(-.16-.2,-.08)(-.16+.2,-.08);
}


\draw[dashed]
  (2,-.5) -- (2,1)
  (2+\wid,-.5) -- (2+\wid,1);

\draw[dash3d]
  (0,-1.5) -- (0,-.2)
  (0+\wid,-1.5) -- (0+\wid,.3);

\draw[S1set] plot[smooth,yshift=-1.5cm]
  coordinates {\crightshort};

\fill[S2set] (a2pr) circle (.06);


\draw
  [preaction={draw,ultra thick,white}]
  [arrow,S0text] (5.55+3*\deform,-.7)
      -- node[swap] {$\!\rho^{0\fla}$}
    +(0,1.2);



\LE{
  \node[Btext] at (3.2,-1) {$\bar B^0$};
}

\node[green!50!blue!50!black]
  at (3.6+2*\deform,-.9+.5*\deform) {$Y^\fla$};
\LE{
\node[S0text,right] at (5.5,-1.5) {$\bar S^0 \sub \bmdl^{e_0}$};
}

\begin{scope}[xshift=-1.4cm,yshift=-2.2cm]
  \coordinate(a1prpr) at (2,0);
  \coordinate(a2prpr) at (4.3,0);

  \draw[dashed]
    (0,0) -- (1.4,.7)
    (4.3,0) -- (5.7,.7);

  \draw [preaction={draw,white,ultra thick}]
     [dashed] (a1prpr) -- (a1);

  \LE{
    \draw[ultra thick,Btext]
      (.7,-0.05) -- (4-.7,-0.05);
  }
  \draw[S1set] (0,0) -- (4.3,0);
  \fill[S2set] (a2prpr) circle (.06);

  \fill(a1prpr) circle (.06) node[below] {$\bar a^{1\fla}$};
  
  \LE{
    \node[below right,S1text] at (3.4,0) {$\bar S^1 \sub \bmdl^{e_1}$};
    \node[below,Btext] at (2.7,-0.05) {$\bar B^1$};
  }
  
\end{scope}

\sety(0/0/0/0/0/0)
\setyy(0/0/0/0/0/0)
\draw
  [preaction={fill,color=white,opacity=.4}]
  [preaction={fill,color=white!70!black!60!blue,opacity=.7}]
  \piece(1,.5)(1.4,.5)(1.8,.9)(2,1);
\RI{
  \draw[draw,ultra thick,white]
    (4.6,3.5) -- +(0,-4.4);
  \draw[arrow]
    (4.6,3.6) -- node[pos=.35] {$\phi_0$} +(0,-4.5);
}
\sety(0.4/0.1/0/0/0/0)
\setyy(0/0/0/0/0/0)
\draw
  [preaction={fill,color=white,opacity=.4}]
  [preaction={fill,color=white!40!black!70!blue,opacity=.7}]
  \piece(.5,.8)(.6,.8)(.7,.5)(1,.5);
\draw[dash3d]
  (a0pr) -- (a0);
\sety(-.2/0/.15/.25/0.3/.3)
\setyy(0.4/0.1/0/0/0/0)
\draw
  [preaction={fill,color=white,opacity=.4}]
  [preaction={fill,color=white!90!black!75!blue,opacity=.7}]
  \piece(0,0)(.1,.1)(.3,.8)(.5,.8);
\LE{
  \node[S0text] at (.5,.5) {$S^0$};
}

\draw[S1set] plot[smooth] coordinates {\crightshort};
\LE{
  \node[S1text,below] at (1,0) {$S^1$};
}

\fill(a0pr) circle (.06)
   node[anchor=south west] {$\bar a^{0\fla}$};


\fill(a0) circle (.06) node[anchor=west] {\LE{\vrule width0mm depth3.5mm}$\strut a^{0\fla}$};
\fill(a1) circle (.06) node[anchor=north west,inner sep=1pt] {$\,a^{1\fla}$};
\fill[S2set] (a2) circle (.06) node[below,S2text] {\LE{$X^0$}};

\LE{
  \draw[densely dotted] (a0) -- (a1);
  \draw[arrow] (2.2,1.2) node[above] {$\lambda_{1}$} .. controls (2.2,.8) .. (2.4,.6);
  \draw[densely dotted] (a0) -- (a2);
  \draw[arrow] (3.5,1.2) node[above] {$\lambda_{2}$} .. controls (3.5,.8) .. (3.4,.6);
}


\LE{
  \draw [arrow,S1text] (-.95,-2.15)
        -- +(1.2,.6);
  \node[right,S1text] at (-.15,-1.85) {$\rho^{1}\U0$};

  \draw
    [preaction={draw,ultra thick,white}]
    [arrow,S1text] (-1,-2.1)
      -- node[pos=.45,inner sep=0pt] {$\rho^{1}\!$}
    (.3,-.2);
}

} 
\begin{ctikzpicture}
\begin{scope}[yshift=-4.5cm] 
  \long\def\LE#1{}\long\def\RI#1{#1}
  \everything
  \draw[arrow,preaction={draw,ultra thick,white}]
    (-.8,2.1) -- node[swap] {$\phi_1 = \id_{\bar S^1}$} +(0,-4.2);
\end{scope}
\begin{scope}
  \long\def\LE#1{#1}\long\def\RI#1{}
  \everything
\end{scope}
\end{ctikzpicture}
\endgroup
\caption{%
An overview of some of the notation from Subsections~\ref{sect.notn.std} and \ref{sect.recti}
(for a plain val-chain $a^0, a^1$) and a bit of additional notation used later.}
\label{fig.notn}
\end{figure}

\begin{notn}\label{notn.chain-balls}
Given a val-chain $a^0, \dots, a^m$
with $a^\ell \in S^\ell$ and with distances $\lambda_1 > \dots > \lambda_{m+1}$, we set
\[
\bar B^\ell \coloneqq B_{> \lambda_{\ell + 1}}(\pr_{\le e_\ell}(a^0))
\subseteq \bmdl^{e_\ell}
 \qquad \text{for }
0 \le \ell \le m.
\]
\end{notn}

(If $\lambda_{m+1} = -\infty$, we set $\bar B^m = \bmdl^{e_m}$.)

Note that $\bar B^\ell$ also contains the projections
$\pr_{\le e_\ell}(a^1), \dots, \pr_{\le e_\ell}(a^\ell)$
and that it is the projection of 
the largest ball around $a^0$ which is disjoint from $X^{e_\ell - 1}$.

\begin{lem}\label{lem.val-chain-basics}
In the situation of Notations~\ref{notn.std} and \ref{notn.chain-balls}, the following hold
for $0 \le \ell \le m$.
\begin{enumerate}
\item
The ball $\bar B^\ell$ is contained in $\bar S^\ell$.
In particular, $\pr_{\le e_\ell}(a^0) \in \bar S^\ell$,
so the function $\rho^\ell$ is defined at the point $\pr_{\le e_\ell}(a^0)$.
\item
The function $\rho^\ell$ satisfies
\[
\vv(\rho^\ell(x^1) - \rho^\ell(x^2)) \ge \vv(x^1 - x^2)
\qquad
\text{for }
x^1, x^2 \in \bar B^\ell
.
\]
\item
For $\bbar a^0 \coloneqq \pr_{\le e_{\ell}}(a^0)$
and $a^{[\ell]} \coloneqq (\bbar a^0, \rho^{\ell}(\bbar a^0)) \in S^\ell$,
we have $\vv(a^0 - a^{[\ell]}) = \lambda_\ell$.
In particular, the sequences $a^0, a^{[1]}, a^2, \dots, a^m$ and $a^{[1]}, a^1, a^2, \dots, a^m$ are val-chains.
\end{enumerate}
\end{lem}

Note that in the last statement, we might have $a^{[1]} = a^0$, namely when $S^0 = S^1$.

\begin{proof}[Proof of Lemma~\ref{lem.val-chain-basics}]
(1)
We have $B := B_{>\lambda_{\ell + 1}}(a^0) \cap S^\ell \ne \emptyset$
but $B \cap \partial S^\ell = \emptyset$
(since $\partial S^\ell \subseteq X^{e_\ell - 1}$
and $\valdist(a^0, X^{e_\ell - 1}) = \lambda_{\ell + 1}$),
so Lemma~\ref{lem.frontier} implies
$\bar B^\ell = \pr_{\le e_\ell}(B) \subseteq \bar S^\ell$.

(2)
This is just the Mean Value Theorem argument from
Remark~\ref{rem.int-val}.

(3) The inequality $\vv(a^0 - a^{[\ell]}) \le \lambda_\ell$ follows from the
definition of val-chain, since $a^{[\ell]} \in X^{e_\ell}$. For the other inequality,
set $\bbar a^{\ell} \coloneqq  \pr_{\le e_{\ell}}(a^{\ell})$.
Then $\vv(a^0 - a^{\ell}) = \lambda_{\ell}$ implies $\vv(\bbar a^0 - \bbar a^{\ell}) \ge \lambda_{\ell}$, and then (2) yields $\vv(a^{[\ell]} - a^{\ell}) \ge \lambda_{\ell}$.
This together with $\vv(a^0 - a^{\ell}) = \lambda_{\ell}$ implies $\vv(a^0 - a^{[\ell]}) \ge \lambda_{\ell}$.

The ``in particular'' part is clear from the definition of val-chains. (Note that the second one is an augmented val-chain).
\end{proof}

\subsection{Rectilinearization}
\label{sect.recti}

In the setting of Notation~\ref{notn.std},
we will sometimes need to ``rectilinearize'' along the lower-dimensional strata: We will apply a map
that translates the coordinates $x\U{\ell}$ by
$\rho^{\ell+1}\U{\ell}$. (Note that the maps $\rho\U{j}^{\ell+1}$ for $j \le \ell - 1$ are not used for rectilinearization.) Here is our notation for this:

\begin{notn}[Rectilinearization]\label{notn.recti-1}
For $0 \le \ell \le m$, and suitable $x = (x\U m, \dots, x\U{\ell}) \in \bmdl^{e_\ell}$, we define
\begin{align*}
\phi_{\ell}(x) &\coloneqq x^\flat \qquad\text{where $x^\flat$ is given by}\\
x^\flat\U m &\coloneqq x\U m,\\
x^\flat\U{m-1} &\coloneqq x\U{m-1} - \rho^{m}\U{m-1}(x\U m),\\
x^\flat\U{m-2} &\coloneqq x\U{m-2} - \rho^{m-1}\U{m-2}(x\U m, x\U{m-1}),\\
&\qquad\vdots\\
x^\flat\U{\ell} &\coloneqq x\U{\ell} - \rho^{\ell+1}\U{\ell}(x\U m, \dots, x\U{\ell+1}).
\end{align*}
Here, ``suitable $x$'' means that all the involved maps $\rho^{j+1}\U{j}$ are defined, i.e., $\phi_{\ell}(x)$ is defined if $\pr_{\le e_j}(x) \in \bar S^j$ for $\ell < j \le m$.
\end{notn}

\begin{rem}\label{rem.recti-ind}
The definition of $\phi_\ell$ can also be written inductively:
\begin{alignat*}{2}
\phi_m(x) &= x
\qquad&&
\text{for } x \in \bmdl^{e_m}
\text{ and }
\\
\phi_{\ell}((\bar x, x\U{\ell}))
&= (\phi_{\ell + 1}(\bar x), x\U{\ell} - \rho^{\ell+1}\U{\ell}(\bar x))
\qquad&&
\text{for }
(\bar x, x\U{\ell}) \in \bmdl^{e_{\ell + 1}} \times \bmdl^{e_{\ell} - e_{\ell+1}}
,
0 \le \ell < m.
\end{alignat*}
\end{rem}

Note that if $e_0 = e_1$, then $\phi_0 = \phi_1$.

We fix some more notation:

\begin{notn}\label{notn.recti-2}
We set
\[
Y := \{x \in \bmdl^{e_0} : \pr_{\le e_\ell}(x) \in \bar S^\ell \text{ for } 0 \le \ell \le m\},
\]
which is a subset of the domain of $\phi_0$.
We write $Y^\flat \coloneqq \phi_0(Y)$ for the rectilinearization of $Y$ (note that $\phi_0$ induces a bijection $Y \fun Y^\flat$) and
\[
\rho^{\ell\flat} \coloneqq \rho^\ell \circ \phi_\ell^{-1}
\]
for the rectilinearization of $\rho^{\ell}$, where $0 \le \ell \le m$.
\end{notn}

Note that the domain of $\rho^{0\flat}$ is $Y^\flat$.

\begin{rem}\label{rem.recti-jac}
From $\vv(\Jac \rho^\ell) \ge 0$, one easily deduces
$\Jac \phi_\ell \in \GL_{e_\ell}(\valring)$ (at every point of the domain of $\phi_\ell$), using Remark~\ref{rem.recti-ind}.
(Intuitively, this follows because the expression of $\Jac \phi_\ell$ in terms of the partial derivatives of $\rho^{i}$, $\ell < i \leq m$, is a ``lower triangular matrix with identities on the diagonal''.)
\end{rem}

To transfer arguments between the rectilinearized and the unrectilinearized setting, we need the maps $\phi_\ell$
to be isometries with respect to the valuation. This is not true everywhere, but it is true on the balls
$\bar B^\ell = B_{> \lambda_{\ell + 1}}(\pr_{\le e_\ell}(a^0))$ introduced in Notation~\ref{notn.chain-balls}, which is what we really need. Here is the precise statement.

\begin{lem}\label{lem.recti-iso}
Suppose that $a^0, \dots, a^m$ is a val-chain with $a^\ell \in S^\ell$ and with distances $\lambda_\ell$.
For $0 \le \ell \le m$, $\phi_\ell$ is defined on $\bar B^\ell$ and the restriction $\phi_\ell \rest \bar B^\ell$
is a valuative isometry (i.e.,
$\vv(\phi_\ell(x^1) - \phi_\ell(x^2)) = \vv(x^1 - x^2)$)
with image $B_{> \lambda_{\ell + 1}}(\phi_\ell(\pr_{\le e_\ell}(a^0)))$.
\end{lem}

\begin{proof}
Use induction and Remark~\ref{rem.recti-ind}.
That $\phi_\ell \rest \bar B^\ell$ is defined follows
from Lemma~\ref{lem.val-chain-basics} (1), that it is an isometry follows from Lemma~\ref{lem.val-chain-basics} (2), and to obtain that the image is all of $B_{> \lambda_{\ell + 1}}(\phi_\ell(\pr_{\le e_\ell}(a^0)))$, consider its inverse (which is easy to specify explicitly).
\end{proof}

\subsection{Defining the stratification}

In this section, we construct a
stratification of the given set $X \sub \bmdl^n$.
(Afterwards, we will prove that this stratification has the desired properties.)
The stratification is obtained by constructing the
skeletons $\mathring X^{s}$ one after another, starting with $\mathring X^{\dim X}$.
More precisely, suppose that
$\mathring X^{s+1}, \dots, \mathring X^{\dim X}$ have already been constructed. We
obtain $\mathring X^{s}$ by starting with
$\mathring X^{s} \coloneqq X \setminus \bigcup_{i > s} \mathring X^{i}$ and by removing closed subsets of
dimension less than $s$ in four steps.

\medskip

\textbf{Step R1:}
We start by partitioning $\mathring X^{s}$ into bradycells (using Proposition~\ref{prop.brady-dec})
and remove all bradycells of dimension less than $s$.
Moreover, for each bradycell $S \subseteq X^{s}$ of dimension $s$, we remove its frontier $\partial S$ from $\mathring X^{s}$. This ensures that afterwards, each definably connected component of $\mathring X^{s}$ is a bradycell. (Recall that ``definably connected'' refers to the language $\LT$.)
Even though $\mathring X^{s}$ is not yet final, let us already call those connected components strata.

By removing an additional closed subset of lower dimension from $\mathring X^s$, we ensure that the ``border condition'' holds, i.e., that for any strata
$S \subseteq \mathring X^s$, $S' \subseteq \mathring X^{s'}$,
where $s' > s$, we have either $S \subseteq \cl(S')$ or
$S \cap \cl(S') = \emptyset$. (In the end, this will imply that $\cl(S')$ is a union of strata.)

Note that none of the properties achieved in this step can be destroyed by removing further closed, lower-dimensional subsets from $\mathring X^s$.

\medskip

\textbf{Step R2:}
Next, we choose a stratum $S \subseteq \mathring X^{s}$ (i.e., a bradycell of dimension $s$)
and an aligner $\aligner \in \CC_n$ of $S$ (see
Definition~\ref{defn.aligned}). For each of these
(finitely many) choices, we remove an \LT-definable subset from $\mathring X^{s}$ as follows.

As explained in Notation~\ref{notn.std}, we assume that $S$ itself is a aligned. This assumption does not cause definability issues of the sets we remove, since $\aligner$ is (by definition of $\CC_n$) \LT-definable.
Set $\bar S \coloneqq \pr_{\le s}(S)$ and denote by $\rho\colon \bar S \fun \bmdl^{n-s}$ the function whose graph is $S$.
Moreover, set $e_1 \coloneqq s$.
By Corollary~\ref{cor.second-sedate}, there is a subset $Z \subseteq \bar S$ of lower dimension
such that $\rho$ is $e_{[1,1]}$-\prep{c_2}-sedated on $\bar S \setminus Z$.
The preimage $S \cap \pr_{\le s}^{-1}(Z)$ is a subset of $S$ of dimension less than $s$;
we remove its closure $\cl(S \cap \pr_{\le s}^{-1}(Z))$ from $\mathring X^{s}$.

\medskip

\textbf{Step R3:}
The next shrinking of $\mathring X^s$ is similar, but instead of considering a single
stratum in $\mathring X^s$, we consider a whole sequence $\mdl S = (S^\ell)_{0 \le \ell \le m}$,
with $S^\ell \subseteq \mathring X^{e_\ell}$ for some $e_0 \ge e_1 > e_2 > \dots > e_m = s$, $m \ge 0$.
(In fact, Step R2 is a special case of Step R3, but for R3 to work, we will need that this special case has been carried out before.)
Similarly to Step R2, for any such sequence $\mdl S$ and any aligner $\aligner \in \CC_n$ of $\mdl S$,
we will obtain a subset $Z \subseteq \bar S^m$ of dimension less than $s$ (where we use Notation~\ref{notn.std}),
and for each $\mdl S$ and $\kappa$ as above, we remove the corresponding set $\cl(S^m \cap \pr_{\le s}^{-1}(Z))$
from $\mathring X^s$.

The goal of Step R3 is to ensure that certain functions on the set $Y^\flat \subseteq \bmdl^{e_0}$ from Notation~\ref{notn.recti-2} are $e_{[j,m]}$-\prep v-sedated. This will be achieved
using Proposition~\ref{prop.sedate} and Corollary~\ref{cor.second-sedate}, so we need to ensure that the functions are already $e_{[j,m']}$-\prep v-sedated for $m' < m$. We use Notation~\ref{notn.recti-2} and set
\begin{equation}\label{eq.def-delta}
\delta^{\flat} \coloneqq \rho^{0\flat} - \rho^{1\flat}\U{\star} \circ \pr_{\le e_1} \colon
Y^\flat \fun \bmdl^{n - e_0}.
\end{equation}
(If $e_0 = e_1$, then $\delta^{\flat} = \rho^{0\flat} - \rho^{1\flat}$.)
The precise goal of Step R3 is to ensure the following: \begin{equation}\label{eq.R3-goal}
\begin{alignedat}{2}
&\text{If $m = 0$ or $e_0 > e_1$: }&&
\rho^{0\flat}
\text{ is $e_{[0,m']}$-\prep{c_2}-sedated on $Y^\flat$ for } 0 \le m' \le m;
\\
&\text{If $m \ge 1$ and $e_0 > e_1$: }&&
\delta^{\flat}
\text{ is $e_{[0,m']}$-\prep b-sedated on $Y^\flat$ for } 1 \le m' \le m;
\\
&\text{If $m \ge 1$ and $e_0 = e_1$: }&&
\delta^{\flat}
\text{ is $e_{[1,m']}$-\prep a-sedated on $Y^\flat$ for } 1 \le m' \le m.
\end{alignedat}
\end{equation}
(Note that in Subsection~\ref{sect.sedate}, the numbering starts with $e_1$, whereas for \prep{c_2}- and \prep b-sedation, we now start with $e_0$.)

To obtain (\ref{eq.R3-goal})
for $m' < m$, nothing needs to be removed from $S^m$; instead, we deduce this inductively from the corresponding result obtained in the construction of
$\mathring X^{e_{m-1}}$ (using Lemma~\ref{lem.recti-sedated} and Step R2); then we can $e_{[j,m]}$-\prep v-sedate the functions using Proposition~\ref{prop.sedate} and Corollary~\ref{cor.second-sedate}. This is straightforward; here are the details.

\begin{proof}[Proof of (\ref{eq.R3-goal}) for $m' < m$.]
Fix $m' < m$. For any statement related to $\delta^{\flat}$, we shall implicitly assume $m' \ge 1$. We keep Notation~\ref{notn.recti-2} with respect to $\mdl S$, but we now additionally consider the shortened sequence $\hat {\mdl S} = (S^\ell)_{0 \le \ell \le m-1}$
and put a hat on various objects relative to $\hat {\mdl S}$ introduced in Notations~\ref{notn.recti-1} and \ref{notn.recti-2} and in (\ref{eq.def-delta}):
$\hat \phi_\ell$,
$\hat Y$, $\hat Y^\flat$, $\hat \rho^{\ell\flat}$,
$\hat \delta^{\flat}$.
Note that we have $Y \sub \hat Y$
and $\phi_\ell = \psi_\ell \circ \hat\phi_\ell$ (for $0 \le \ell < m$), where
$\psi_\ell = \phi_{m-1} \times \id_{\bmdl^{e_\ell - e_{m-1}}}$
is the map that rectilinearizes only with respect
to $\rho^m\U{m-1}$. In particular,
\[
\hat\rho^{0\flat}
= \rho^{0\flat} \circ \psi_0,
\qquad
\hat\rho^{1\flat}
= \rho^{1\flat} \circ \psi_1,
\qquad \text{and} \qquad
\hat\delta^{\flat}
= \delta^\flat \circ \psi_0.
\]
By Step~R3 for $\hat {\mdl S}$ (which has already been carried out when constructing $\mathring X^{e_{m-1}}$),
$\hat\rho^{0\flat}$ is $e_{[0,m']}$-\prep{c_2}-sedated
if $e_0 > e_1$
and
$\hat\delta^{\flat}$ is $e_{[1,m']}$-\prep a-sedated or $e_{[0,m']}$-\prep b-sedated (depending on whether $e_0 > e_1$).
The map $\psi_0$ is of the form required by Lemma~\ref{lem.recti-sedated}, since
$\rho^{m}\U{m-1}$ is $e_{[m,m]}$-\prep{c_2}-sedated by Step~R2,
so that lemma implies
(\ref{eq.R3-goal}) for $m' < m$.
\end{proof}

\begin{proof}[Obtaining (\ref{eq.R3-goal}) for $m' = m$.]
Suppose first that $m = 0$ or $e_0 > e_1$.
Using $\vv(\Jac_x \rho^0) \ge 0$ (for $x \in Y$) and
$\vv(\Jac_x \phi_0) = 0$ (by Remark~\ref{rem.recti-jac}), we obtain
$\vv(\Jac_{x^\flat} \rho^{0\flat}) \ge 0$, so we can apply
Corollary~\ref{cor.second-sedate} to $\rho^{0\flat}$
using $e_{[0,m]}$.
This yields a subset
$Z \subseteq \pr_{\le e_m}(Y^\flat) = \pr_{\le e_m}(Y) \subseteq \bar S^m$ of dimension less than $e_m = s$ such that
$\rho^{0\flat}$ is $e_{[0,m]}$-\prep{c_2}-sedated on
$Y^\flat \setminus Z'$, where $Z'$ is the preimage of $Z$ in
$\bmdl^{e_0}$ under the projection. We shrink $Y^\flat$
to $Y^\flat \setminus Z'$ by removing $\cl(S^m \cap \pr_{\le e_m}^{-1}(Z))$ from $S^m$.

In a similar way (but using Proposition~\ref{prop.sedate} \prep b), we ensure that $\delta^{\flat}$ is $e_{[0,m]}$-\prep{b}-sedated if $m \ge 1$.
For this, we have to check that
$\vv(\Jac_{x^\flat} \delta^{\flat}) \ge 0$;
this follows from the corresponding statements for $\rho^{0\flat}$ and $\rho^{1\flat}$.

Finally, if $m \ge 1$ and $e_0 = e_1$, then without checking any additional condition, we can apply Proposition~\ref{prop.sedate} \prep a to shrink $S^m$
in such a way that
$\delta^{\flat}$ becomes $e_{[1,m]}$-\prep{a}-sedated.
\end{proof}

\medskip

\textbf{Step R4:}
We keep the notation from Step R3 and remove one more set from $S^m$ (again, for each choice of $\mdl S$ and $\aligner$),
namely $S^m \cap \pr_{\le e_m}^{-1}(\partial(\pr_{\le e_m}(Y)))$.
This ensures that if we choose a sequence $(S^\ell)_\ell$ of strata after this step has been carried out and write $Y$ for the set corresponding to this new sequence, then
$\partial(\pr_{\le e_m}(Y)) \cap \bar S^m = \emptyset$ and hence,
since $S^m$ is connected, we have
either $\bar S^m \subseteq \pr_{\le e_m}(Y)$ or
$\bar S^m \cap \pr_{\le e_m}(Y) = \emptyset$.
(Later, only sequences for which the first of these cases occurs will be relevant.)

\medskip

This finishes the construction of $\mathring X^{s}$
and hence of the stratification of $X$. We will now prove that this stratification is indeed a valuative Lipschitz stratification.

\subsection{Relating the stratification to val-chains}

We fix a val-chain $a^0, \dots, a^m$ with $a^\ell \in S^\ell \subseteq \mathring X^{e_\ell}$, dimensions
$e_0 \ge e_1 > \dots > e_m$, and distances $\lambda_1 > \dots > \lambda_{m+1}$. We use Notations~\ref{notn.std},
\ref{notn.chain-balls}, \ref{notn.recti-1} and \ref{notn.recti-2}.
The main goal of this subsection is to prove Lemma~\ref{lem.key}, which can be considered as a bound on some kind of distance between the tangent spaces
$\bm T_{a^0}(\mathring X^{e_0})$ and $\bm T_{a^1}(\mathring X^{e_1})$.
The three different properties obtained in (\ref{eq.R3-goal}) will roughly correspond to the following three different kinds of val-chains (in this order): augmented val-chains with $S^0 = S^1$, plain val-chains, and augmented val-chains with $S^0 \ne S^1$.

\begin{notn}\label{notn.a-a0}
We set $a := a^0$ and $\bar a \coloneqq \pr_{\le e_0}(a)$.
By Lemma~\ref{lem.val-chain-basics} (1), we have $\pr_{\le e_\ell}(a) \in \bar S^\ell$
for $0 \le \ell \le m$, so
$\bar a \in Y$ and we can define $\bar a^{\flat} \coloneqq \phi_0(\bar a) \in Y^\flat$.
\end{notn}

\begin{rem}\label{rem.a-in-Y}
Since $\pr_{\le e_m}(a) \in \pr_{\le e_m}(Y) \cap \bar S^m$, this intersection is non-empty, so Step~R4 implies
$\bar S^m \subseteq \pr_{\le e_m}(Y)$ and hence $\bar S^m = \pr_{\le e_m}(Y)$.
\end{rem}

We apply Notation~\ref{notn.zeta-xi-sigma} to $\bar a^{\flat}$, relative to the set $Y^\flat$, starting with $e_0$ instead of $e_1$, and we allow ourselves to use that notation even if $e_0 = e_1$:
\begin{alignat}{2}
\zeta_\ell(\bar a^\flat) &= \dist(\pr_{\le e_\ell}(\bar a^\flat), \bmdl^{e_\ell} \setminus \pr_{\le e_\ell}(Y^\flat))
\qquad&&
\text{for }0 \le \ell \le m
\label{eq.recall-zeta}
\\
\sigma_\ell(\bar a^\flat) &= \max\{1, \norm{\pr_{>e_{\ell}}(\bar a^\flat)} \cdot \zeta_{\ell-1}(\bar a^\flat)^{-1}\}
\qquad&&
\text{for }1 \le \ell \le m.
\label{eq.recall-sigma}
\end{alignat}
(Concerning the case $e_0 = e_1$, we consider the norm of the empty tuple as being $0$ and its valuation as being $\infty$.)

\begin{lem}\label{lem.lambda}
We have
\begin{equation}\label{eq.xi-lambda-zeta}
\vv(\pr_{>e_{\ell+1}}(\bar a^\flat))
\overset{(1)}{\ge}
\lambda_{\ell+1}
\overset{(2)}{\ge}
\vv(\zeta_{\ell}(\bar a^\flat))
\end{equation}
for $0 \le \ell \le m - 1$ at (1) and
$0 \le \ell \le m$ at (2). In particular,
\begin{equation}\label{eq.hence-sigma}
\vv(\sigma_\ell(\bar a^\flat)) = 0
\qquad \text{for } 1 \le \ell \le m.
\end{equation}
\end{lem}

\begin{proof}
The ``in particular'' part follows directly from (\ref{eq.xi-lambda-zeta}) and (\ref{eq.recall-sigma}).

(1)
We have $\pr_{>e_{\ell+1}}(\bar a^\flat)
= (a^\flat\U{\ell}, \dots, a^\flat\U0)$, so it suffices to
check that $\vv(a^\flat\U{j}) \ge \lambda_{j+1}
\mathop{}(\ge \lambda_{\ell+1})$
for $0 \le j \le \ell$. This follows from Lemma~\ref{lem.val-chain-basics} (3); indeed,
$a^\flat\U{j} = a\U{j} - \rho\U{j}^{j+1}(\pr_{\le e_{j+1}}(a)))$ is just one of the coordinates of
$a - a^{[j+1]}$, where the notation $a^{[j+1]}$ is the
one from the Lemma~\ref{lem.val-chain-basics}.

(2) It is enough to check that we have an inclusion
\begin{equation}\label{eq.check-Bflat}
B_{> \lambda_{\ell+1}}(\pr_{\le e_{\ell}}(\bar a^\flat))) \subseteq \pr_{\le e_{\ell}}(Y^\flat)
= \phi_\ell(\pr_{\le e_\ell}(Y)).
\end{equation}
By Lemma~\ref{lem.recti-iso}, we have
$B_{> \lambda_{\ell+1}}(\pr_{\le e_{\ell}}(\bar a^\flat)))
= \phi_\ell(\bar B^\ell)$ (where $\bar B^\ell$ was defined as
$B_{> \lambda_{\ell+1}}(\pr_{\le e_{\ell}}(\bar a)))$;
see Notation~\ref{notn.chain-balls}), so
(\ref{eq.check-Bflat}) is equivalent to
\begin{equation}\label{eq.check-B}
\bar B^\ell \subseteq \pr_{\le e_{\ell}}(Y)
.
\end{equation}
The definition of $Y$ yields
\begin{equation}
\pr_{\le e_{\ell}}(Y) = \pr_{\le e_\ell}(Y') \cap Y_{\ell+1} \cap \dots \cap Y_m,
\end{equation}
where
\begin{equation}
Y' = \{x \in \bmdl^{e_0} : \pr_{\le e_j}(x) \in \bar S^j \text{ for } 0 \le j \le \ell\}.
\end{equation}
and where $Y_j$ is the preimage of $\bar S^j$ under the projection $\bmdl^{e_\ell} \fun \bmdl^{e_j}$ (for $\ell + 1 \le j \le m$).
By Lemma~\ref{lem.val-chain-basics}~(1), for $j \ge \ell + 1$ we have
$\pr_{\le e_j}(\bar B^\ell) \subseteq \bar B^j \subseteq \bar S^j$ and hence
$\bar B^\ell \subseteq Y_j$.
By Remark~\ref{rem.a-in-Y} applied to the val-chain $a^0, \dots, a^\ell$, we have $\bar S^\ell  \subseteq \pr_{\le e_\ell}(Y')$. Together
with $\bar B^\ell \subseteq \bar S^\ell$
this implies (\ref{eq.check-B}).
\end{proof}

Suppose now that $m \ge 1$. We keep Notation~\ref{notn.a-a0} and additionally set $b \coloneqq a^1$, $\bar b \coloneqq \pr_{\le e_1}(b)$
and $\bar b^\flat \coloneqq \phi_1(\bar b)$. (This is well-defined by the same argument as for $\bar a^\flat$, applied to the val-chain $a^1, \dots a^m$).
Recall that in Notation~\ref{notn.recti-2}, we introduced the rectilinearized maps 
$\rho^{\ell\flat} \coloneqq \rho^\ell \circ \phi_\ell^{-1}$.
The following is a key intermediate result.

\begin{lem}[Bounding the difference of derivatives]\label{lem.key}
Suppose that $m \ge 1$. Then for $1 \le i \le e_m$, we have
\[
\vv(\partial_i \rho^{0\flat}(\bar a^{\flat}) - \partial_i \rho^{1\flat}\U{\star}(\bar b^{\flat})) \ge \lambda_1 - \lambda_{m+1}.
\]
\end{lem}

\begin{proof}
Set $\bar c \coloneqq \pr_{\le e_1}(a)$,
$c \coloneqq (\bar c, \rho^1(\bar c))$ and
$\bar c^\flat \coloneqq \phi_1(\bar c)$. Note that
$c = a^{[1]}$ in the notation of Lemma~\ref{lem.val-chain-basics}, so $c, a^2, \dots a^m$
is a val-chain and hence well-definedness of $\bar c^\flat$ follows as for $\bar a^\flat$ and $\bar b^\flat$.

To prove the lemma, we ``use $c$ as an intermediate step'', i.e., it suffices to prove
\begin{align}
\vv(\partial_i \rho^{0\flat}(\bar a^{\flat}) - \partial_i \rho^{1\flat}\U{\star}(\bar c^{\flat})) &\ge \lambda_1 - \lambda_{m+1}
\qquad \text{and}
\label{eq.ac-flat}
\\
\vv(\partial_i \rho^{1\flat}\U{\star}(\bar c^{\flat}) - \partial_i \rho^{1\flat}\U{\star}(\bar b^{\flat})) &\ge \lambda_1 - \lambda_{m+1}.
\label{eq.cb-flat}
\end{align}
Since $a, c, a^2, \dots, a^m$ and
$c, b, a^2, \dots, a^m$ are val-chains (by Lemma~\ref{lem.val-chain-basics}), these two inequalities follow from two special cases of the lemma itself: (\ref{eq.ac-flat}) is just the special case $b = a^{[1]}$,
and (\ref{eq.cb-flat}) follows from the special case where
$S^0 = S^1$. (The special case yields (\ref{eq.cb-flat}) with $\rho^{1\flat}\U{\star}$ replaced by $\rho^{1\flat}$.)
Thus we will now prove the lemma in these two cases.

\medskip

Case $b = a^{[1]}$: In this case, $\pr_{\le e_1}(\bar a) = \bar b$ and hence also $\pr_{\le e_1}(\bar a^\flat) = \bar b^\flat$.
Recall the definition of $\delta^{\flat}$ from Step~R3; we have
\[
\delta^{\flat}(x^\flat) = \rho^{0\flat}(x^\flat)
- \rho^{1\flat}\U{\star}(\pr_{\le e_1}(x^\flat))
\qquad \text{for } x^\flat \in Y^\flat
\]
and hence
\[
\partial_i \rho^{0\flat}(\bar a^{\flat}) - \partial_i \rho^{1\flat}\U{\star}(\bar b^{\flat}) = \partial_i\delta^{\flat}(\bar a^{\flat}).
\]
We now distinguish two sub-cases. If $e_0 > e_1$, then
since $\delta^{\flat}$ is $e_{[0,m]}$-\prep b-sedated on $Y^\flat$ (by (\ref{eq.R3-goal})), we get (for $1 \le i \le e_m$)
\begin{align*}
\vv(\partial_{i}\delta^{\flat}(\bar a^\flat))
 &\overset{(\ref{eq.sedated})}{\ge}
 \min\{\vv(\delta^{\flat}(\bar a^\flat)), \vv(\pr_{>e_{1}}(\bar a^\flat))\} - \vv(\zeta_{m}(\bar a^\flat)) + \sum_{\ell=1}^{m} \vv(\sigma_\ell(\bar a^\flat))
  \\
 &\overset{(\ref{eq.xi-lambda-zeta}), (\ref{eq.hence-sigma})}{\ge}
 \min\{\vv(\delta^\flat(\bar a^\flat)), \lambda_1\} - \lambda_{m+1}.
\end{align*}
If, on the other hand, $e_0 = e_1$, then
$\delta^{\flat}$ is $e_{[1,m]}$-\prep a-sedated on $Y^\flat$ and we get
\begin{align*}
\vv(\partial_{i}\delta^{\flat}(\bar a^\flat))
 &\overset{(\ref{eq.sedated})}{\ge}
 \vv(\delta^{\flat}(\bar a^\flat)) - \vv(\zeta_{m}(\bar a^\flat)) + \sum_{\ell=2}^{m} \vv(\sigma_\ell(\bar a^\flat))
  \\
 &\overset{(\ref{eq.xi-lambda-zeta}), (\ref{eq.hence-sigma})}{\ge}
\vv(\delta^\flat(\bar a^\flat)) - \lambda_{m+1}.
\end{align*}

In both cases, $\delta^\flat(\bar a^\flat)  =
\rho^0(\bar a) - \rho^1\U{\star}(\bar b) =
(a - b)\U{\star}$,
so the valuation of this is at least $\lambda_1$
(since $a = a^0$ and $b = a^1$)
and we get $\vv(\partial_i \delta^\flat(\bar a^\flat)) \ge \lambda_1 - \lambda_{m+1}$, as desired.

\medskip

Case $S_0 = S_1$: In that case, we have $\rho^{1\flat}\U{\star} = \rho^{0\flat}$, so the claim of the lemma is
\begin{equation}\label{eq.S0=S1}
\vv(\partial_i \rho^{0\flat}(\bar a^{\flat}) - \partial_i \rho^{0\flat}(\bar b^{\flat})) \ge \lambda_1 - \lambda_{m+1};
\end{equation}
we will prove this using the Mean Value Theorem argument from Remark~\ref{rem.int-val}.

Set $B := B_{\ge \lambda_1}(\bar a)$. By Lemma~\ref{lem.recti-iso}, $B^\flat := \phi_0(B)$ is also a ball
(note that $\phi_0 = \phi_1$ and that $ \bar B^0 \sub B \subseteq \bar B^1$) and, since $B$ contains $\bar a$ and $\bar b$ and $\phi_0$ is a valuative isometry on $B$, we have $\vv(\bar a^{\flat} - \bar b^\flat) = \lambda_1$. Thus for Remark~\ref{rem.int-val}
to yield (\ref{eq.S0=S1}), it remains to verify that on the entire ball $B^\flat$, we have
\[
\vv(\Jac \partial_i \rho^{0\flat}) \ge -\lambda_{m+1}
.
\]
Given any $\bar c^\flat \in B^\flat$, let
$\bar c$ be its preimage in $B$ and
$c := (\bar c, \rho^0(\bar c)) \in S^0$.
Applying (\ref{eq.R3-goal}) to the strata $S^1, \dots, S^m$ yields
that $\rho^{0\flat} = \rho^{1\flat}$ is $e_{[1,m]}$-\prep{c_2}-sedated on $Y^\flat$. (Note that the set $Y^\flat$ corresponding to $S^1, \dots, S^m$ is the same as the one corresponding to $S^0, \dots, S^m$.)
Together with Lemma~\ref{lem.lambda}, this yields
\begin{align*}
\vv(\Jac\partial_{i}\rho^{0\flat}(\bar c^\flat))
&\overset{(\ref{eq.second-sedated})}{\ge}
- \vv(\zeta_{m}(\bar c^\flat)) + \sum_{\ell=2}^{m} \vv(\sigma_\ell(\bar c^\flat))
\\
&\overset{(\ref{eq.xi-lambda-zeta}), (\ref{eq.hence-sigma})}{\ge}
- \lambda_{m+1}.
\end{align*}
which is what we had to prove.
\end{proof}

\subsection{Proving that we have a valuative Lipschitz stratification}

We will use the characterization of valuative Lipschiz Stratifications given by Proposition~\ref{prop.flag-lip}. Thus suppose that $a^0, \dots, a^m$ is a val-chain with $a^\ell \in S^\ell \subseteq \mathring X^{e_\ell}$, with dimensions $e_0 \ge e_1 > \dots > e_m$, and with distances $\lambda_1 > \dots > \lambda_{m+1}$.
We need to find vector spaces
\begin{equation}\label{eq.flags-inc}
V_{k,m} \subseteq V_{k,m-1} \subseteq \dots \subseteq V_{k,k+1} \subseteq V_{k,k} = \bm T_{a^k}S^k
\quad \text{for }0 \le k \le m
\end{equation}
with $\dim V_{k,\ell} = e_\ell$ satisfying
\begin{equation}\label{eq.flags-dist}
\Delta(V_{k,\ell}, V_{k+1,\ell}) \ge \lambda_{k+1} - \lambda_{\ell+1}  \quad \text{for } 0 \le k < \ell \le m ,
\end{equation}
(where $\Delta(W_1, W_2)$ is the valuative metric on the Grassmannian; see Definition~\ref{defn.Delta}).
The strategy is as follows.
Given any val-chain as above and any aligner $\aligner \in \CC_n$ of $(S^\ell)_{0 \le \ell \le m}$,
we will define an $e_m$-dimensional space denoted by $V_{0,m}$ depending only on the val-chain and on $\aligner$.
Let $V_{k,\ell}$ be the space obtained by applying the same definition to the sub-val-chain $a^k, a^{k+1}, \dots, a^\ell$ (and the same aligner $\aligner$).
Once the spaces are defined, we will prove:
\begin{eqnarray}
&V_{0,0} = \bm T_{a^0}S^0\label{eq.single-tang}&\text{(in the case $m = 0$)};\\
&V_{0,m} \subseteq V_{0,m-1}\label{eq.single-inc}&\text{if }m \ge 1; \\
&\Delta(V_{0,m}, V_{1,m}) \ge \lambda_1 - \lambda_{m+1}&\text{if }m \ge 1\label{eq.single-dist}.
\end{eqnarray}
By applying these results to various sub-val-chains of $a^0, \dots, a^m$ one then obtains (\ref{eq.flags-inc}) and (\ref{eq.flags-dist}), i.e., we then are done with the proof of the theorem.

\medskip

We start by defining $V_{0, m}$.
As usual, we use Notation~\ref{notn.std},
\ref{notn.chain-balls}, \ref{notn.recti-1} and \ref{notn.recti-2}.
In particular, we assume that the coordinate system has been transformed using $\kappa$. This is harmless, since such a transformation preserves the notion of val-chains on the one hand, and the properties we are about to prove on the other hand.

\begin{notn}
For $0 \le \ell \le m$ and suitable $x = (\bar x, x') \in \bmdl^{e_\ell} \times \bmdl^{n - e_\ell}$, we define a variant of the rectilinearization maps, where ``all coordinates of $\bmdl^{n - e_\ell}$ are rectilinearized along $S^\ell$'':
\begin{equation}\label{eq.phi-tilde}
\tilde{\phi}_\ell(x) \coloneqq (\phi_\ell(\bar x), x' - \rho^\ell(\bar x)).
\end{equation}
(Note that if $S^0 = S^1$, then $\tilde{\phi}_0 = \tilde{\phi}_1$.)
We moreover set
\begin{align*}
a &\coloneqq a^0,\\
W &\coloneqq \bmdl^{e_m} \times \{0\}^{n - e_m}
\qquad\text{and}\\
V_{0,m} &\coloneqq
(\Jac_a\tilde\phi_0)^{-1}(W)
.
\end{align*}
\end{notn}

That $\tilde\phi_0$ is defined at $a$ follows from
Remark~\ref{rem.a-in-Y}.
As required, we have $\dim V_{0,m} = e_m$,
so to finish the proof of the theorem, it remains to prove
(\ref{eq.single-tang}),
(\ref{eq.single-inc}) and
(\ref{eq.single-dist}).

\begin{proof}[Proof of (\ref{eq.single-tang})]
In the case $m = 0$, $\tilde\phi_0^{-1}$ sends $\bar S^0 \times \{0\}^{n - e_0}$ to $S^0$, and
we have $W = \bmdl^{e_0} \times \{0\}^{n-e_0}$. Thus
$(\Jac_a\tilde\phi_0)^{-1}(W)$ is the tangent space to $S^0$ at $a$, as required.
\end{proof}

\begin{proof}[Proof of (\ref{eq.single-inc})]
Suppose that $m \ge 1$. We have
\begin{equation}
V_{0,m-1} =(\Jac_a \hat{\tilde\phi}_0)^{-1}(\hat W)
\end{equation}
where
\begin{eqnarray}
&\hat W = \bmdl^{e_{m-1}} \times \{0\}^{n - e_{m-1}} \quad\text{and}\\
&\hat{\tilde\phi}_0\colon  (x\U m,x\U{m-1},x\U{m-2} \dots, x\U 0, x\U\star) \efun (x\U m,x\U{m-1}, x^\flat\U{m-2},\dots, x^\flat\U 0, x\U\star - \rho^0(\pr_{\le e_0}(x))).
\end{eqnarray}
An easy computation shows that $(\Jac_a\hat{\tilde\phi}_0)^{-1}(\hat W) = (\Jac_a\tilde\phi_0)^{-1}(\hat W)$; indeed, we have $\tilde\phi_0 = \psi \circ \hat{\tilde\phi}_0$, where
$\psi = \phi_{m - 1} \times \id_{\bmdl^{n - e_{m-1}}}$,
and $(\Jac_x \psi)^{-1}(\hat W) = \hat W$ for any $x$.

Together with $\hat W \supseteq W$, this implies
$V_{0,m-1} \supseteq V_{0,m}$, as required.
\end{proof}

\begin{proof}[Proof of (\ref{eq.single-dist})]
We have
$V_{1,m} =
(\Jac_b\tilde\phi_1)^{-1}(W)$ where
$b \coloneqq a^1$ (and $\tilde{\phi}_1$ has been defined in (\ref{eq.phi-tilde})).
To obtain $\Delta(V_{0,m}, V_{1,m}) \ge \lambda_1 - \lambda_{m+1}$, it suffices to prove that
\begin{equation}\label{eq.a-b}
\vv\left( (\Jac_{a}\tilde\phi_0)^{-1}\rest W - (\Jac_{b}\tilde\phi_1)^{-1}\rest W\right) \ge \lambda_1 - \lambda_{m+1}
\end{equation}
(by Lemma~\ref{lem.Delta}).

From the definition of $\tilde{\phi}_0$, we get
\begin{eqnarray}
&\Jac_a\tilde\phi_0 = \left(
\begin{array}{cc}
\Jac_{\bar a}\phi_0 & 0 \\
-\Jac_{\bar a}\rho^0 & 1
\end{array}
\right)
\quad\text{and hence}\label{eq.not-inv}\\
&
(\Jac_a\tilde\phi_0)^{-1} = \left(
\begin{array}{cc}
(\Jac_{\bar a}\phi_0)^{-1} & 0 \\
(\Jac_{\bar a}\rho^0) \circ (\Jac_{\bar a}\phi_0)^{-1} & 1
\end{array}
\right)
= \left(
\begin{array}{cc}
(\Jac_{\bar a}\phi_0)^{-1} & 0 \\
\Jac_{\bar a^\flat}(\rho^0 \circ \phi_0^{-1}) & 1
\end{array}
\right),
\label{eq.inv-star}
\end{eqnarray}
where $\bar a^\flat = \phi_0(\bar a)$.
If we moreover set $\bbar a  = \pr_{\le e_1}(\bar a)$ and $\bbar a^\flat =\phi_1(\bbar a)$, then
we have
$\phi_0(\bar a) = \left(\phi_1(\bbar a), \bar a\U0 - \rho^{1}\U0(\bbar a)\right)$,
and exactly the same computation as in (\ref{eq.not-inv}) and (\ref{eq.inv-star}) yields
\begin{equation}
(\Jac_{\bar a}\phi_0)^{-1} = \left(
\begin{array}{cc}
(\Jac_{\bbar a}\phi_1)^{-1} & 0 \\
\Jac_{\bbar a^\flat}(\rho^1\U0 \circ \phi_1^{-1}) & 1
\end{array}
\right).
\label{eq.inv-0}
\end{equation}

Combining (\ref{eq.inv-star}) with (\ref{eq.inv-0}) yields
\begin{equation}\label{eq.mat-a}
(\Jac_a\tilde\phi_0)^{-1}
 =
\left(
\begin{array}{ccc}
(\Jac_{\bbar a}\phi_1)^{-1} & 0                          & 0\\
\Jac_{\bbar a^\flat}(\rho^{1}\U0 \circ \phi_1^{-1}) & 1 & 0 \\[.4ex]
\cline{1-2}
\multicolumn{2}{c|}{\vrule height 2.5ex width0pt\Jac_{\bar a^\flat}(\rho^{0} \circ \phi_0^{-1})} & 1
\end{array}
\right)
 =
\left(
\begin{array}{ccc}
(\Jac_{\bbar a}\phi_1)^{-1}
  & 0                          & 0\\
\Jac_{\bbar a^\flat}\rho^{1\flat}\U0  & 1 & 0 \\[.4ex]
\cline{1-2}
\multicolumn{2}{c|}{\vrule height 2.5ex width0pt\Jac_{\bar a^\flat}\rho^{0\flat}} & 1
\end{array}
\right),
\end{equation}
where the coordinates are grouped according to
$\bmdl^{e_1} \times \bmdl^{e_0 - e_1} \times \bmdl^{n - e_0}$.

We also do the computation from (\ref{eq.not-inv}) and (\ref{eq.inv-star}) for
$\tilde{\phi}_1(b)
= (\phi_1(\bar b), (b\U0, b\U{\star}) - \rho^1(\bar b))$, where $\bar b = \pr_{\le e_1}(b)$ and $\bar b^\flat = \phi_1(\bar b)$, and obtain (with the same grouping of coordinates as before)
\begin{equation}\label{eq.mat-b}
(\Jac_b\tilde\phi_1)^{-1} = \left(
\begin{array}{ccc}
(\Jac_{\bar b}\phi_1)^{-1} & 0                          & 0\\
\Jac_{\bar b^\flat}(\rho^{1}\U0 \circ \phi_1^{-1}) & 1 & 0 \\[.4ex]
\cline{1-2}
\vrule height 2.5ex width0pt\Jac_{\bar b^\flat}(\rho^{1}\U\star \circ \phi_1^{-1}) & \multicolumn{1}{c|}{0} & 1
\end{array}
\right)
 =
\left(
\begin{array}{ccc}
(\Jac_{\bar b}\phi_1)^{-1}
& 0                          & 0\\
\Jac_{\bar b^\flat}\rho^{1\flat}\U0  & 1 & 0 \\[.4ex]
\cline{1-2}
\vrule height 2.5ex width0pt\Jac_{\bar b^\flat}\rho^{1\flat}\U\star & \multicolumn{1}{c|}{0} & 1
\end{array}
\right).
\end{equation}
To prove (\ref{eq.a-b}), we have to prove the corresponding statements for the three sub-matrices where
(\ref{eq.mat-a}) and (\ref{eq.mat-b}) differ.

For the lower most, this is exactly the statement of Lemma~\ref{lem.key}. 
For the middle sub-matrix, the result is obtained by applying Lemma~\ref{lem.key} to the augmented val-chain
$a^{[1]}, a^1, a^2, \dots, a^m$, where
$a^{[1]} = (\bbar a, \rho^1(\bbar a)) \in S^1$
and $\vv(a^{[1]} -  a^1) \ge \lambda_1$ by Lemma~\ref{lem.val-chain-basics} (3).

Finally, for the upper-most sub-matrix, we
use an inductive argument.
If $m = 1$, then $\phi_1$ is the identity, so suppose now $m \ge 2$.
Set $\bbbar a \coloneqq \pr_{\le e_2}(a)$,
$a^{[2]} \coloneqq (\bbbar a, \rho^2(\bbbar a)) \in S^2$ and similarly
$\bbbar b \coloneqq \pr_{\le e_2}(b)$,
$b^{[2]} \coloneqq (\bbbar b, \rho^2(\bbbar b)) \in S^2$.
Using Lemma~\ref{lem.val-chain-basics}, we obtain that
$a^{[2]}, b^{[2]}, a^3, \dots, a^m$ is a val-chain with
$\vv(a^{[2]} - b^{[2]}) \ge \lambda_1$.
By induction, we may assume that (\ref{eq.a-b})
holds for this shorter val-chain, i.e.,
\begin{equation}\label{eq.a-b-ind}
\vv\left( (\Jac_{a^{[2]}}\tilde\phi_2)^{-1}\rest W - (\Jac_{b^{[2]}}\tilde\phi_2)^{-1}\rest W\right) \ge \lambda_1 - \lambda_{m+1}.
\end{equation}
This implies the desired inequality
\begin{equation}\label{eq.a-b-want}
\vv\left( (\Jac_{\bbar a}\phi_1)^{-1}\rest W - (\Jac_{\bar b}\phi_1)^{-1}\rest W\right) \ge \lambda_1 - \lambda_{m+1},
\end{equation}
using that $\phi_1$ is obtained from $\tilde\phi_2$ by omitting some coordinates and that the derivatives of these functions only depend on the first $e_2$ coordinates.
More precisely, a computation as in (\ref{eq.inv-star}) and
(\ref{eq.inv-0}) (applied to $\tilde\phi_2$ and $\phi_1$) yields that
$(\Jac_{\bbar x}\phi_1)^{-1}$ is a sub-matrix of
$(\Jac_x\tilde\phi_2)^{-1}$ (for suitable $x \in \bmdl^n$ and $\bbar x = \pr_{\le e_1}(x)$)
and that this sub-matrix only depends on $\pr_{\le e_2}(x)$.
\end{proof}

This finishes the proof of Theorem~\ref{thm.main-na}, and hence also of Theorem~\ref{main}.

\end{document}